%% file: main.tex
\documentclass[final,hidelinks,onefignum,onetabnum]{siamart250211}

\input{ex_shared}

\ifpdf
\hypersetup{
  pdftitle={Entropy Stable Nodal Discontinuous Galerkin Methods via Quadratic Knapsack Limiting},
  pdfauthor={B. Christner and J. Chan}
}
\fi

\usepackage{subfig}

\usepackage{cprotect}
\usepackage{booktabs} % FIXME should probably cut this...
\usepackage{multirow} % FIXME same here...
\usepackage{colortbl} % for cell colors FIXME and here
\usepackage{bm}
\usepackage[normalem]{ulem}
\newcommand{\comment}[1]{}
\DeclareMathOperator*{\clip}{clip}
\newcommand{\ub}[2]{\underbrace{#1}_{#2}}
\newcommand{\fnt}[1]{\bm{\mathsf{#1}}}
\newcommand{\vecf}[1]{\bm{#1}}
\newcommand{\vect}[1]{\bm{\mathsf{#1}}}
\newcommand{\real}{\mathbb{R}}

\renewcommand{\aligned}[1]{&#1}
\newcommand{\labell}[1]{\addtocounter{equation}{1}\tag{\theequation}\label{#1}}

\usepackage{stmaryrd}
\renewcommand{\hat}[1]{\widehat{#1}} 
\newcommand{\norm}[1]{\left\| #1 \right\|} % norm
\renewcommand{\d}[0]{{\rm d}}

\DeclareRobustCommand{\revdel}[1]{%
  \ifmmode
    {\color{red}\cancel{#1}}%
  \else
    \textcolor{red}{\sout{#1}}%
  \fi
}

\begin{document}

\maketitle

% REQUIRED
\begin{abstract}
\cite{Lin24} enforces a cell entropy inequality for nodal discontinuous Galerkin methods by combining flux corrected transport (FCT)-type limiting and a knapsack solver, which determines optimal limiting coefficients that result in a semi-discrete cell entropy inequality while preserving nodal bounds. In this work, we provide a slight modification of this approach, where we utilize a quadratic knapsack problem instead of a standard linear knapsack problem. We prove that this quadratic knapsack problem can be reduced to efficient scalar root-finding. Numerical results demonstrate that the proposed quadratic knapsack limiting strategy is efficient and results in a semi-discretization with improved regularity in time compared with linear knapsack limiting, while resulting in fewer adaptive timesteps in shock-type problems. % better time accuracy than the linear knapsack limiting approach, and behaves better under adaptive time-stepping.
\end{abstract}

% REQUIRED
\begin{keywords}
Discontinuous Galerkin, Spectral Element Methods, Entropy Stable, Knapsack Limiting, Cell Entropy Inequality
\end{keywords}

% REQUIRED
\begin{MSCcodes}
68Q25, 68R10, 68U05
\end{MSCcodes}

\section{Introduction}
\label{sec:Introduction}
%In computational science, particularly when addressing multi-dimensional systems, numerically solving partial differential equations (PDEs) is critical. 
A promising method that has emerged in the field of computational fluid dynamics (CFD) is the discontinuous Galerkin spectral element method (DGSEM) \cite{Kopriva10}. One of the primary advantages of DGSEM is its ability to achieve arbitrarily high spatial accuracy, making it a valuable tool for capturing fine details of complex systems. %This high order accuracy enables the numerical solution to approximate the true solution with a coarser mesh, reducing computational time and resource expenditure.
Despite its advantages, high order DGSEM is not without limitations, the most notable being a lack of robustness. High order methods often exhibit robustness issues or may fail to converge to accurate solutions. In response, entropy stability techniques, such as entropy stable discontinuous Galerkin (ESDG) methods, have been proposed to mitigate these issues \cite{Carpenter14, Gassner16, Chen17, Crean18, Chan18}. For example, Chan introduced a strategy where the method enforces a discrete entropy inequality in the numerical scheme, promoting stability under high order settings \cite{Chan18}. These papers utilize the same approach for constructing ESDG methods based on entropy-conservative volume fluxes and entropy-stable numerical fluxes. These fluxes are designed to satisfy symmetry, consistency, and an entropy conservation condition specific to the PDE being solved. However, evaluating these fluxes can be computationally expensive, and deriving them can be complex \cite{Chandrashekar13,Ranocha21,Ranocha23}. 

Lin and Chan proposed an alternative method for constructing an entropy stable DG method \cite{Lin24}. By incorporating cell entropy inequality constraints at the element level and solving a linear knapsack optimization problem at each right-hand side evaluation, Lin’s approach identifies optimal flux-corrected transport (FCT)-type blending coefficients while ensuring entropy stability. The resulting scheme remains as close as possible to the original high order method while satisfying a local entropy inequality. Vilar further demonstrated that this blended scheme successfully retains high order accuracy while maintaining robustness, showing its potential as a reliable method for high-dimensional, complex systems \cite{Vilar25}. Moreover, this knapsack limiting approach naturally preserves limiting-enforced bound constraints, which can be computed using convex limiting techniques (see, for example, \cite{Pazner21,RuedaRamirez23}). 

In this work, we introduce a quadratic knapsack problem as the solution to the blended update. Unlike the standard linear knapsack problem, the quadratic knapsack problem is continuous with respect to the solution. We will show that, when compared against linear knapsack limiting, quadratic knapsack limiting results in a higher order of convergence in time for shock-type problems, and requires far fewer adaptive timesteps, especially for high orders of approximation in space. Lastly, we show numerically that knapsack limiting overcomes issues of linear instability found in \cite{Gassner21}, where non-physical oscillations emerge and are amplified over time. 

\section{Entropy Stable Discontinuous Galerkin (DG) Methods}

DG methods were first introduced by Reed and Hill \cite{Reed73}. DG methods are now popular among CFD applications, due to their arbitrarily high order of accuracy and easy parallelizability. The domain is partitioned into a non-overlapping set of elements. Over each element, the DG method behaves like a finite element method. Across boundaries, DG methods do not enforce continuity. Rather, DG methods couple neighboring elements together using numerical fluxes. 

One particular form of DG method is the discontinuous Galerkin spectral element method (DGSEM) \cite{Kopriva10}. Over each element, the solution is approximated as a linear combination of degree $N$ Lagrange polynomials, based over Legendre-Gauss-Lobatto (LGL) quadrature nodes. The scheme utilizes collocation of the LGL nodes for its integral quadrature and interpolation. 

In this section, we start by introducing the discontinuous Galerkin spectral element method (DGSEM), along with summation-by-parts (SBP) operators. We then move on to the flux differencing scheme. Last, we describe a subcell low order method which will be used to stabilize the DGSEM solution. 

\subsection{Discontinuous Galerkin Spectral Element Method (DGSEM)}

We will now discuss the DGSEM discretization. For the sake of brevity, we will only describe how to construct summation-by-parts (SBP) operators in $d\in\{1,2\}$ dimensions, with the note that extension to higher dimensions is possible. All schemes will be written in such a way to accommodate an arbitrary number of dimensions.

We consider a hyperbolic conservation law imposed over domain $\Omega\subseteq\real^d$, with the imposed linear or nonlinear fluxes $\vecf{f}_k:\real^v\to\real^v$:
\begin{align}\frac{\partial\vecf{u}}{\partial t}+\sum_{k=1}^d \frac{\partial \vecf{f}_k\left(\vecf{u}\right)}{\partial x_k}=\vect{0},\label{eq:LitReviewConservationLaw}\end{align}
where $\vecf{u}:\real^{d}\times\real_{\geq 0}\to\real^{v}$ and $v$ is the number of variables. We partition $\Omega$ into equally sized elements $D^m$. In the case of one dimension, each $D^m$ is a segment, and two dimensions, $D^m$ are quadrilaterals. We define $\hat{D}$ as the \textit{reference element}, such that $\vecf{\Phi}^m:\hat{D}\to D^m$ is an invertible mapping. 

Over each $D^m$, integrals are discretized using quadrature rules with nodes, $\fnt{x}^m = \vecf{\Phi}^m(\hat{\vect{x}})$, where $\hat{\vect{x}}$ are nodes over $\hat{D}$. In the case of one dimension, $\hat{\vect{x}}$ are the $N+1$ Legendre-Gauss-Lobatto (LGL) nodes, and in two dimensions, it is the set of Cartesian products of the same LGL nodes. We define $\ell_i^m:D^m\to \real$ the Lagrange basis polynomials based over the nodes $\vect{x}^m$, and $\ell_i:\hat{D}\to \real$ the Lagrange basis polynomials. % based over $\hat{\vect{x}}$.

The DGSEM discretization is chosen so that over each element, the estimated solution is a linear combination of the Lagrange basis polynomials. To do this, it asserts that the residual of the inner products of \eqref{eq:LitReviewConservationLaw} with each Lagrange basis polynomial vanishes over each element. Then, the conservation law can be expressed in its ``strong form'' \cite{Hesthaven07}: 
\begin{align}
\int _{D^m }^{ }\left(\frac{\partial \vecf{u}}{\partial t}+\sum_{k=1}^d\frac{\partial \vecf{f}_k\left(\vecf{u}\right)}{\partial x_k}\right)\ell_i^mdx+\int _{\partial D^m }^{ }\left(\vecf{f}^{\ast }\left(\vecf{u}^+,\vecf{u}, \hat{\vecf{n}}\right)-\vecf{f}\left(\vecf{u}\right)\cdot\hat{\vecf{n}}\right)\ell_i^mdx=\fnt{0},
\end{align}
where $\vecf{f}:\real^v\to\real^{v\times d}$ is a matrix-valued function, such that the $k$th column $(\vecf{f}(\vecf{u}))_k = \vecf{f}_k(\vecf{u})$, $\hat{\vecf{n}}$ denotes the outward normal on $\partial D^m$, $\vecf{u}^+$ denotes the corresponding solution on the neighboring element with $D^m$, and $\vecf{f}^\ast:\real^v\times\real^v\times\real^d\to\real^v$ describes a chosen numerical flux along the normal direction. 
Then, expressing $\vecf{u} = \sum_i \vect{u}_i\ell_i^m(x)$ as a linear combination of the Lagrange polynomials, such that $\vect{u} = \vecf{u}(\vect{x}^m)$, and discretizing with the integral quadrature, the discretized form over element $D^m$ can be expressed by
\begin{align}
\fnt{M}\frac{\d\fnt{u}}{\d t}+\sum_{k=1}^d\fnt{Q}_k\vecf{f}_k\left(\fnt{u}\right)+\fnt{E}^T\left(\vecf{f}^{\ast }\left(\fnt{u_f}^+,\fnt{u_f},\hat{\fnt{n}}\right)-\vecf{f}\left(\fnt{u_f}\right)\cdot\hat{\fnt{n}}\right)=\fnt{0},
\label{eq:DiscreteDG}
\end{align}
where $\vect{u_f}=\vecf{u}(\vect{x_f})$, and $\vect{x_f}$ are the face nodes of $\vect{x}^m$. Here $\fnt{M}\in \real^{n\times n}$ is the mass matrix, $\fnt{E}$ is the boundary matrix, and $\fnt{Q}=\fnt{MD}_k$ where $\fnt{D}_k\in \real^{n\times n}$ are differentiation matrices, where $n=|\hat{\fnt{x}}|$ is the number of nodes within each element. $\fnt{M}$ and $\fnt{Q}_k$ are referred to as summation-by-parts operators, and in one dimension take the form
\begin{align}
    \fnt{M}^\text{1D} = \frac{h}{2} \cdot\text{diag}(\omega_1,...,\omega_{N+1}), && \fnt{D}^\text{1D}_{ij} = \ell_i'(x_j),
    &&
    \fnt{E}^\text{1D} = \begin{pmatrix}
        1 & \fnt{0}^T\\
        \fnt{0}^T & 1
    \end{pmatrix},
\end{align}
where $\omega_1,...,\omega_{N+1}$ are the corresponding quadrature weights for the LGL nodes $\hat{\fnt{x}}$. In two dimensions, they can be expressed using Kronecker products of their lower dimensional operators:
\begin{align}\begin{matrix}\fnt{M}=\fnt{M}_{\text{1D}}\otimes\fnt{M}_{\text{1D}}&\fnt{E}=\fnt{I}_{N+1}\otimes\fnt{E}_{\text{1D}}\\\fnt{Q}_1=\fnt{M}_{\text{1D}}\otimes\fnt{Q}_{\text{1D}}&\fnt{Q}_2=\fnt{Q}_{\text{1D}}\otimes\fnt{M}_{\text{1D}}\end{matrix}.\end{align}

It was shown that the resulting scheme is accurate up to order $\mathcal{O}(h^{N+1})$ in sufficiently regular cases \cite{Kopriva10}, though there are pathological cases where slightly degraded accuracy is observed \cite{Peterson}. Thus, this method can be used to construct arbitrarily high order accurate schemes. However, its solution is not always robust. Flux differencing can be used to alter the scheme \eqref{eq:DiscreteDG} into one that that satisfies a semi-discrete entropy inequality.

\subsection{Flux Differencing} \label{sec:FluxDiff}

In this section, we introduce the flux differencing form of DGSEM. Critical to its derivation are the summation-by-parts property $\fnt{Q}_k+\fnt{Q}_k^T=\fnt{E}^T\fnt{B}_k$ is diagonal, the conservation property $\fnt{Q}_k\fnt{1}=\fnt{0}$, and $\fnt{B}_k\vecf{f}_k(\fnt{u})=\vecf{f}_k(\fnt{u_f}) \hat{\fnt{n}}_k$, where
\begin{align*}
    \fnt{B}_\text{1D} = \begin{pmatrix}-1 & \fnt{0} \\ \fnt{0} & 1\end{pmatrix},
    &&
    \fnt{B}_1 = \fnt{M}_\text{1D} \otimes \fnt{B}_\text{1D},
    &&
    \fnt{B}_2 = \fnt{B}_\text{1D} \otimes \fnt{M}_\text{1D}.
\end{align*}
The first property is essentially a discrete form of integration by parts for degree $N$ polynomials. The second property arises from the fact that the derivative of a constant function is $\fnt{0}$. Since a constant is a degree $0$ polynomial, it is differentiated exactly by $\fnt{Q}_k$. Using these three properties, one can express
\begin{align}
    \sum_{k=1}^d(\fnt{Q}_k\vecf{f}_k(\vect{u}))_i - (\fnt{E}^T\vecf{f}(\fnt{u_f})\cdot \hat{\fnt{n}})_i = \sum_{k=1}^d\sum_j (\fnt{Q}_{k,ij} - \fnt{Q}_{k,ji})\frac{\vecf{f}_k(\vect{u}_i) + \vecf{f}_k(\vect{u}_j)}{2}\label{eq:CancelOut}.
\end{align}
This identity \eqref{eq:CancelOut} is used to remove the presence of the additional surface term in the numerical flux from \eqref{eq:DiscreteDG}. Therefore, if a normal $\fnt{n}_{ij}$ is defined such that $\fnt{n}_{ij,k}=(\fnt{Q}_k - \fnt{Q}^T_k)_{ij}$, then
\begin{align}
    \sum_{k=1}^d\sum_j (\fnt{Q}_{k,ij} - \fnt{Q}_{k,ji})\frac{\vecf{f}_k(\vect{u}_i) + \vecf{f}_k(\vect{u}_j)}{2} =\sum _j^{ }\norm{\fnt{n}_{ij}}\left( \frac{\vecf{f}\left(\fnt{u}_i\right)+\vecf{f}\left(\fnt{u}_j\right)}{2}\cdot\frac{\fnt{n}_{ij}}{\norm{\fnt{n}_{ij}}}\right). \label{eq:NormalProof}
\end{align}
Importantly, $\vecf{f}^\text{central}(\fnt{u}_i, \fnt{u}_j, \hat{\fnt{n}}_{ij}) = \left(\frac{\vecf{f}\left(\fnt{u}_i\right)+\vecf{f}\left(\fnt{u}_j\right)}{2}\cdot\hat{\fnt{n}}_{ij}\right)$ is known as the central flux. Therefore, \eqref{eq:DiscreteDG} can be generalized to
\begin{align}
\fnt{M}\frac{\d \fnt{u}^H}{\d t}+\fnt{r}^H+ \fnt{E}^T\vecf{f}^{\ast }\left(\fnt{u_f}^+,\fnt{u_f}, \hat{\fnt{n}}\right)=\fnt{0}\label{eq:HighOrderScheme}\\
\fnt{r}^H_i = \sum _j^{ }\norm{\fnt{n}_{ij}}\vecf{f}^\text{vol}\left(\fnt{u}_i, \fnt{u}_j, \frac{\fnt{n}_{ij}}{\norm{\fnt{n}_{ij}}}\right)\nonumber,
\end{align}
where $\vecf{f}^\text{vol}$ is a suitable volume flux which satisfies
\begin{align}
    \vecf{f}^\text{vol}(\vect{u}, \vect{u}, \hat{\vect{n}}) &= \vecf{f}(\vect{u})\cdot\hat{\fnt{n}} && \text{Consistency}\nonumber\\
    \vecf{f}^\text{vol}(\vect{u}_i, \vect{u}_j, \hat{\vect{n}}_{ij}) &= -\vecf{f}^\text{vol}(\vect{u}_j, \vect{u}_i, -\hat{\vect{n}}_{ij}) && \text{Skew-Symmetry}\label{eq:VolumeFlux}.
\end{align}

If the volume flux is the central flux, the original DGSEM formulation is recovered. The properties of $\vecf{f}^\ast$ and $\vecf{f}^\text{vol}$ determine the properties of the resulting solution. If the volume flux is the central flux, we refer to the resulting scheme as DGSEM. If the volume flux is entropy conservative and the numerical flux is entropy stable as in \eqref{eq:EntropyStable}, the scheme satisfies an entropy inequality \cite{Chan18}. We refer to this scheme as entropy stable flux differencing (ESFD). Note that for the purposes of this work, unless otherwise specified, we will assume that the numerical flux for DGSEM is the entropy stable, local Lax-Friedrichs flux
\begin{align}
    \vecf{f}^\text{LxF}(\fnt{u}_i, \fnt{u}_j, \hat{\fnt{n}}_{ij}) = \vecf{f}^\text{central}(\fnt{u}_i, \fnt{u}_j, \hat{\fnt{n}}_{ij})-\frac{\lambda}{2}(\fnt{u}_j - \fnt{u}_i),\label{eq:LaxFriedrichs}
\end{align}
where $\lambda$ is the maximum wavespeed of the 1D Riemann problem with $\fnt{u}_i$ and $\fnt{u}_j$.
% state of the art for entropy stable method by first satisfying a cell entropy inequality, and adding dissipation to satisfy a cell entropy ineqality. 
Entropy conservative fluxes tend to be computationally expensive and complex \cite{Ranocha21, Ranocha23, Chandrashekar13}. Moreover, the use of entropy conservative volume fluxes in flux differencing results in a locally linearly unstable scheme \cite{Gassner21}. This strategy is a state-of-the-art entropy stable method which first satisfies a cell entropy equality, and then adds dissipation to satisfy a cell entropy inequality. This work aims to investigate knapsack limiting \cite{Lin24}, where a cell-entropy \textit{inequality} is satisfied instead, while keeping the advantages of entropy stable flux differencing. We will use entropy stable flux differencing as a baseline to compare our resulting models. 

\subsection{Low Order Method}

Low order methods are often highly diffusive, stabilizing their simulations \cite{Hesthaven07}. We will blend in a low order method to stabilize our high order solution, in such a way that the resulting scheme will be accurate, conservative, and stable. %In this section, we will introduce the low order method we use. 

The low order method can be interpreted both as a low order finite volume method and as a flux differencing scheme with the corresponding low order finite volume operators:
\begin{align}
    \fnt{M}^L = \fnt{M}, && \fnt{Q}^L_\text{1D} = \frac{1}{2}\begin{pmatrix}
        -1&1&&&\\
        -1&0&1\\
        &-1&0&1\\
        &&&\ddots
    \end{pmatrix},\\
\fnt{Q}_1^L=\fnt{I}_{N+1}\otimes \fnt{Q}_\text{1D}^L,&&\fnt{Q}_2^L=\fnt{Q}_\text{1D}^L\otimes\fnt{I}_{N+1}.
\end{align}

The scheme can then be expressed in terms of flux differencing, where
\begin{gather}
\fnt{M}\frac{\d \fnt{u}^L}{\d t}+\fnt{r}^L+\fnt{E}^T\vecf{f}^{\ast }\left(\fnt{u_f}^+,\fnt{u_f}, \hat{\fnt{n}}\right)=\fnt{0}\label{eq:LowOrderScheme},\\
\fnt{r}_i^L = \sum _j^{ }\norm{\fnt{n}^L_{ij}}\vecf{f}^\text{LxF}\left(\fnt{u}_i, \fnt{u}_j, \frac{\fnt{n}^L_{ij}}{\norm{\fnt{n}^L_{ij}}}\right), \quad \fnt{n}^L_{ij,k}=(\fnt{Q}_k^L-(\fnt{Q}_k^L)^T)_{ij}\nonumber,
\end{gather}
where both $\vecf{f}^\text{LxF}$ and $\vecf{f}^\ast$ are the local Lax-Friedrichs flux. The resulting scheme is equivalent to the low order scheme in \cite{Pazner21}, though with different notation. 
%The Lax-Friedrichs flux is known to be strongly dissipative, so the resulting solution is strongly dissipated in underresolved circumstances. However, since $\fnt{Q}^L$ are highly sparse, the volume flux is only evaluated $\mathcal{O}(N^d)$ times, making the resulting scheme more efficient than ESFD which requires $\mathcal{O}(N^{d+1})$ entropy conservative flux evaluations. 
Moreover, the resulting scheme is provably entropy stable without the use of entropy conservative volume fluxes, and it is provably positivity preserving for the compressible Euler and Navier-Stokes equations under appropriate choices of the wavespeed $\lambda$ \cite{Lin23}.

\subsection{Entropy Stability and a Cell Entropy Inequality}

Entropy stable schemes are numerical methods which satisfy a form of an entropy inequality. The inequality is derived so that a discrete form of the 2nd law of thermodynamics is satisfied by the scheme for the compressible Euler equations. Assuming admissible quantities, such as positive pressure and density for the compressible Euler equations, a scheme which satisfies a (semi)-discrete entropy inequality typically behaves much more robustly. The result is a solution which dissipates mathematical entropy over time, which is often sufficient to stabilize the solution without relying on heuristic stabilization techniques.

We are particularly interested in equations which satisfy an entropy inequality:
\begin{align*}\frac{\partial\eta\left(\vecf{u}\right)}{\partial t}+\sum_{k=1}^{d}\frac{\partial F_k\left(\vecf{u}\right)}{\partial x_{k}}\le0\labell{eq:EntropyInequality},\end{align*}
where $\eta:\mathbb{R}^{v}\to\mathbb{R}$ is called the \textit{entropy function}, and $F_1,...,F_d:\mathbb{R}^{v}\to\mathbb{R}$ are called the \textit{entropy fluxes}. In particular, we choose $\eta$ and $F_k$ such that $\left(\eta,F_k\right)$ are a \textit{convex entropy, entropy-flux pair}. Here, $\eta$ is convex, and
\begin{align*}\frac{\partial F_k}{\partial \vecf{u}} = \vecf{v}^T \frac{\partial \vecf{f}_k}{\partial \vecf{u}}\labell{eq:EntropyRule},\end{align*}
where $\vecf{v}=\nabla_{\vecf{u}}\eta$ are referred to as the \textit{entropy variables}. We will also define $\vecf{F}:\real^v\to\real^d$ such that $\vecf{F}(\vecf{u})_k=F_k(\vecf{u})$, where
$F_k$ are given by
\begin{align*}F_k\left(\vecf{u}\right)=\vecf{v}^{T}\vecf{f}_k\left(\vecf{u}\right)-\psi_k\left(\vecf{u}\right)\labell{eq:EntropyPotential},\end{align*}
where $\psi_1,...,\psi_d:\real^v\to\real$ are called the \textit{entropy potential fluxes} \cite{Godlewski13}. We will also define $\vecf{\psi}:\real^v\to\real^d$ such that $\vecf{\psi}(\vecf{u})_k=\psi_k(\vecf{u})$. We define an entropy \textit{stable} or \textit{conservative} flux as one which satisfies one of the following properties:
\begin{align*}\left(\fnt{v}_{j}-\fnt{v}_{i}\right)^{T}\vecf{f}\left(\fnt{u}_{i},\fnt{u}_{j},\hat{\fnt{n}}_{ij}\right)\aligned{=}\left(\psi\left(\fnt{u}_{j}\right)-\psi\left(\fnt{u}_{i}\right)\right)\cdot\hat{\fnt{n}}_{ij}&\text{Entropy Conservative},\\
\left(\fnt{v}_{j}-\fnt{v}_{i}\right)^{T}\vecf{f}\left(\fnt{u}_{i},\fnt{u}_{j},\hat{\fnt{n}}_{ij}\right)\aligned{\le}\left(\psi\left(\fnt{u}_{j}\right)-\psi\left(\fnt{u}_{i}\right)\right)\cdot\hat{\fnt{n}}_{ij}&\text{Entropy Stable}.\labell{eq:EntropyStable}\end{align*}
where $\fnt{v}=\vecf{v}\left(\fnt{u}\right)$ and $\fnt{v_f}=\vecf{v}(\vect{u_f})$. A \textit{cell entropy inequality} can be derived by integrating \eqref{eq:EntropyInequality} over each element:
\begin{align*}\int_{D^{m}}^{ }\frac{\partial\eta\left(\vecf{u}\right)}{\partial t}d\vecf{x}+\int_{D^{m}}^{ }\sum_{k=1}^{d}\frac{\partial F_k\left(\vecf{u}\right)}{\partial x_{k}}d\vecf{x}\le0.\end{align*}
In this work we will enforce the following discrete version of the continuous cell entropy inequality, which one can show is formally high order accurate using the chain rule and accuracy of the Gauss-Lobatto quadrature:
\begin{align*}\fnt{v}^{T}\fnt{M}\frac{\d \fnt{u}}{\d t}+\fnt{1}^T\fnt{E}^{T}\left(\fnt{v_f}^T\vecf{f}^\ast(\fnt{u_f}^+,\fnt{u_f},\hat{\fnt{n}})-\vecf{\psi}(\fnt{u_f})\cdot \hat{\fnt{n}}\right)\le0\labell{eq:DiscreteCellEntropyInequality},\end{align*}
In particular, using collocated Gauss-Lobatto quadrature, we can discretize the rate of change of entropy as follows:
\begin{align*}\fnt{v}^{T}\fnt{M}\frac{\d \fnt{u}}{\d t}=\sum_{i}^{ }\omega_{i}\frac{\partial\eta\left(\fnt{x}_{i}\right)}{\partial t}\approx\int_{D^{m}}^{ }\frac{\partial\eta\left(\vecf{u}\right)}{\partial t}d\vecf{x}.\end{align*}
%making the discrete estimation relevant in cases where a shock is present. 

\section{Subcell Limiting for Enforcing a Cell Entropy Inequality}

In this section, we will discuss the blending of the high and low order updates $\frac{\d \fnt{u}^{H}}{\d t}$ from \eqref{eq:HighOrderScheme} and $\frac{\d\fnt{u}^{L}}{\d t}$ from \eqref{eq:LowOrderScheme}, into a blended scheme in flux differencing form:
\begin{align*}\fnt{M}\frac{\d\fnt{u}}{\d t}+\fnt{r}\left(\fnt{u}\right)+\fnt{E}^T\vecf{f}^{\ast }\left(\fnt{u_f}^+,\fnt{u_f}, \hat{\fnt{n}}\right)=0\labell{eq:BlendedScheme},\end{align*}
for some volume term $\fnt{\fnt{r}\left(\fnt{u}\right)}$, such that the resulting blended scheme is high order in sufficiently regular cases, locally conservative, positivity preserving, and satisfies a cell entropy inequality. We will next discuss linear knapsack limiting \cite{Lin24}, the current state-of-the-art strategy for determining optimal blending coefficients. Last, we introduce \textit{quadratic} knapsack limiting, the primary subject of this paper. We will prove that the (continuous) quadratic knapsack problem, the driver behind quadratic knapsack limiting, can be solved using an efficient scalar quasi-Newton iteration with at most $n+2$ iterations.

First, we derive a \textit{volume form} of the semi-discrete cell entropy inequality \eqref{eq:DiscreteCellEntropyInequality}, which will describe for what volume terms $\fnt{r}$ the scheme \eqref{eq:BlendedScheme} satisfies the cell entropy inequality. This inequality will aid in the blending strategy. By replacing $\fnt{M}\frac{\d \fnt{u}}{\d t}$ with its definition in \eqref{eq:BlendedScheme}, we obtain:
\begin{align*}
-\fnt{v}^{T}\fnt{r}-\fnt{1}^T\fnt{E}^T(\vecf{\psi}(\fnt{u_f})\cdot \hat{\fnt{n}})\\
+\fnt{1}^T\fnt{E}^T(\fnt{v_f}^T \vecf{f}^\ast(\fnt{u_f}^+,\fnt{u_f},\hat{\fnt{n}}))-\fnt{v}^T\fnt{E}^T\vecf{f}^\ast(\fnt{u_f}^+,\fnt{u_f},\hat{\fnt{n}})\aligned{\le}0\labell{eq:CEILongForm}.\end{align*}
By utilizing the property of the boundary matrix $\fnt{E}$ that $\fnt{E}_{ij}^T = \delta (\fnt{x_f}_{,j}, \fnt{x}_i)$, we can simplify this expression to
\begin{align*}-\fnt{v}^{T}\fnt{r}-\fnt{1}^T\fnt{E}^T(\vecf{\psi}(\fnt{u_f})\cdot \hat{\fnt{n}})\le0\labell{eq:CEI}.\end{align*}
For sufficiently regular solutions, Vilar showed in \cite{Vilar25} that the high order scheme \eqref{eq:HighOrderScheme} violates the cell entropy inequality \eqref{eq:CEI} with a violation proportional to $\mathcal{O}\left(h^{N+1}\right)$. This bound was sharpened in \cite{JesseAV}. Moreover, the low order method \eqref{eq:LowOrderScheme} provably satisfies the cell entropy inequality. Therefore, if we blend the high order scheme towards the low order scheme in order to satisfy the cell entropy inequality, the resulting blended scheme will hopefully retain high order accuracy in sufficiently regular cases, while also satisfying the cell entropy inequality. Moreover, the provable positivity properties of the low order scheme can be baked into the blending strategy to ensure that the blended scheme satisfies positivity. 

\subsection{Subcell Blending/Limiting}

In this section, we discuss strategies for blending high order DGSEM \eqref{eq:HighOrderScheme} toward the low order method \eqref{eq:LowOrderScheme} in order to satisfy the volume form of the cell entropy inequality \eqref{eq:CEI}, while preserving local conservation and provable positivity. We say that a scheme with volume term $\fnt{r}$ is \textit{locally conservative} if the sum of the components $\fnt{1}^{T}\fnt{r}=\fnt{0}\in\mathbb{R}^{v}$. This will ensure that the variables in the conservation law remain discretely conserved. Before proceeding, it is important to note that both the high and low order methods \eqref{eq:HighOrderScheme} and \eqref{eq:LowOrderScheme} are locally conservative due to the skew-symmetry property of the volume fluxes \eqref{eq:VolumeFlux}, and the skew-symmetry of $\fnt{Q}_k-\fnt{Q}_k^T$ and $\fnt{Q}_k^L-(\fnt{Q}_k^L)^T$.

 We want to blend these volume terms together such that $\fnt{r}$ remains locally conservative. To do this, we introduce operators $\fnt{\Delta}\in\mathbb{R}^{n\times L}$ and $\fnt{R}\in\mathbb{R}^{L\times n}$ such that
\begin{align*}\left(\fnt{\Delta R}-\fnt{I}_{n}\right)\fnt{D}\aligned{=}0&\text{Feasibility/Consistency}\\
\diag\left(\fnt{\Delta}^{T}\fnt{1}\right)\fnt{RD}\aligned{=}0\labell{eq:SLO}&\text{Conservation},\end{align*}
where $\fnt{D}\in\mathbb{R}^{n\times\left(n-1\right)}$ is the transpose of the differencing operator. Importantly, $\mathcal{R}\left(\fnt{D}\right)=\mathcal{N}\left(\fnt{1}^{T}\right)$. If $\fnt{\Delta}$ and $\fnt{R}$ satisfy the properties in \eqref{eq:SLO}, we call them \textit{subcell limiting operators}. One example of a set of valid subcell limiting operators, the ones used in the numerical section of this paper, are
\begin{align*}\fnt{\Delta}=\begin{pmatrix}-1&1&&&\\
&-1&1&&\\
&&\ddots&\ddots&\\
&&&-1&1\end{pmatrix}\in \real^{n\times (n+1)},&&\fnt{R}=\begin{pmatrix}0&&&&\\
1&0&&&\\
1&1&0&&\\
1&1&1&&\\
&&&\ddots&\\
1&1&1&\dots&1\end{pmatrix}\in\mathbb{R}^{\left(n+1\right)\times n}\labell{eq:BasicSLO}.\end{align*}
These operators follow the definition in \eqref{eq:SLO}. The corresponding $\fnt{\Delta}$ and $\fnt{R}$ mimic an elementwise differencing operation and a cumulative sum operation respectively. To apply the blending between the high and low order volume terms, we define the \textit{blending coefficients} $\fnt{\theta}\in\left[0,1-\fnt{\ell^{c}}\right]$, where $\fnt{\ell^{c}}\in\left[0,1\right]^{L}$ are called the \textit{limiting coefficients}. The blending coefficients parameterize the blending, and are allowed to vary elementwise between $\fnt{0}$ and $1-\fnt{\ell^{c}}$. The limiting coefficients $\fnt{\ell^{c}}$ are chosen so that the blended scheme satisfies a positivity condition. Then, the blending is performed as
\begin{align*}\fnt{r}\left(\fnt{\theta}\right)=\fnt{r}^{H}+\fnt{\Delta}\diag\left(\fnt{\theta}+\fnt{\ell^{c}}\right)\fnt{R}\left(\fnt{r}^{L}-\fnt{r}^{H}\right)\labell{eq:Blending},\end{align*}
where $\fnt{r}\left(\fnt{\theta}\right)=\fnt{r}\left(\fnt{u},\fnt{\theta}\right)$ but we have dropped $\fnt{u}$ for simplicity of notation. Note that the blending above, when coupled with the operators in \eqref{eq:BasicSLO} defines subcell limiting as in \cite{Pazner21, RuedaRamirez23}. The definition of subcell limiting operators in \eqref{eq:SLO} is chosen such that the blended scheme \eqref{eq:Blending} satisfies local conservation for any set of blending coefficients and any set of locally conservative, high and low order volume terms. It also ensures that the low order volume term remains feasible.

\begin{theorem}
    For any locally conservative $\fnt{r}^L$ and $\fnt{r}^H$, and for any blending coefficients $\fnt{\theta}\in [0,1-\fnt{\ell^c}]$, $\fnt{r}(\fnt{\theta})$ is locally conservative.
\end{theorem}
\begin{proof}
    To show $\fnt{r}(\fnt{\theta})$ is locally conservative, we take its component sum:
    \begin{align*}\fnt{1}^{T}\fnt{r}\left(\fnt{\theta}\right)\aligned{=}\ub{\fnt{1}^{T}\fnt{r}^{H}}{0}+\fnt{1}^{T}\fnt{\Delta}\diag\left(\fnt{\theta}+\fnt{\ell^{c}}\right)\fnt{R}\left(\fnt{r}^{L}-\fnt{r}^{H}\right)\\
\aligned{=}\left(\fnt{\theta}+\fnt{\ell^{c}}\right)^{T}\diag\left(\fnt{\Delta}^{T}\fnt{1}\right)\fnt{R}\left(\fnt{r}^{L}-\fnt{r}^{H}\right).\end{align*}
Since $\fnt{1}^{T}\left(\fnt{r}^{L}-\fnt{r}^{H}\right)=\fnt{0}$, $\fnt{r}^{L}-\fnt{r}^{H}\in\mathcal{N}\left(\fnt{1}^{T}\right)=\mathcal{R}\left(\fnt{D}\right)$. Therefore, $\fnt{r}^{L}-\fnt{r}^{H}=\fnt{D}\tilde{\fnt{r}}$ for some $\tilde{\fnt{r}}\in\mathbb{R}^{n-1}$. So,
\begin{align*}\fnt{1}^{T}\fnt{r}\left(\fnt{\theta}\right)=\left(\fnt{\theta}+\fnt{\ell^{c}}\right)^{T}\ub{\diag\left(\fnt{\Delta}^{T}\fnt{1}\right)\fnt{R}\fnt{D}}{0}\tilde{\fnt{r}}=\fnt{0}.\end{align*}
\end{proof}

We can also express the cell entropy inequality \eqref{eq:CEI} in terms of the blending coefficients by substitution:
\begin{align*}-\fnt{v}^{T}\fnt{r}\left(\fnt{\theta}\right)-\fnt{1}^T\fnt{E}^T(\vecf{\psi}(\fnt{u_f})\cdot \hat{\fnt{n}})\aligned{\le}0\\
\iff-\fnt{v}^{T}\left(\fnt{r}^{H}+\fnt{\Delta}\diag\left(\fnt{\theta}+\fnt{\ell^{c}}\right)\fnt{R}\left(\fnt{r}^{L}-\fnt{r}^{H}\right)\right)\aligned{\le}\fnt{1}^T\fnt{E}^T(\vecf{\psi}(\fnt{u_f})\cdot \hat{\fnt{n}})\\
\iff-\fnt{v}^{T}\fnt{\Delta}\diag\left(\fnt{\theta}+\fnt{\ell^{c}}\right)\fnt{R}\left(\fnt{r}^{L}-\fnt{r}^{H}\right)\aligned{\le}\fnt{1}^T\fnt{E}^T(\vecf{\psi}(\fnt{u_f})\cdot \hat{\fnt{n}})+\fnt{v}^{T}\fnt{r}^{H}\\
\iff-\fnt{v}^{T}\fnt{\Delta}\diag\left(\fnt{R}\left(\fnt{r}^{L}-\fnt{r}^{H}\right)\right)\left(\fnt{\theta}+\fnt{\ell^{c}}\right)\aligned{\le}\fnt{1}^T\fnt{E}^T(\vecf{\psi}(\fnt{u_f})\cdot \hat{\fnt{n}})+\fnt{v}^{T}\fnt{r}^{H}\\
\iff\fnt{a}^T\fnt{\theta} &\geq b\labell{eq:BlendedCEI},\end{align*}
where $\fnt{a}=\diag\left(\fnt{R}\left(\fnt{r}^{L}-\fnt{r}^{H}\right)\right)\fnt{\Delta}^{T}\fnt{v}$ and $b=-\fnt{1}^T\fnt{E}^T(\vecf{\psi}(\fnt{u_f})\cdot \hat{\fnt{n}})-\fnt{v}^{T}\fnt{r}^{H}-\fnt{a}^{T}\fnt{\ell^{c}}$.

\subsection{Knapsack Limiting}

In this section, we discuss knapsack limiting for determining optimal blending coefficients that ensure that the blended scheme \eqref{eq:BlendedScheme} is high order accurate in sufficiently regular cases, satisfies the cell entropy inequality $\fnt{a}^{T}\fnt{\theta}\ge b$, and satisfies a positivity condition governed by $0\le\fnt{\theta}\le1-\fnt{\ell^{c}}$. We will introduce linear knapsack limiting as in \cite{Lin24}, using it to motivate quadratic knapsack limiting. We will then show that the solution to the quadratic knapsack problem can be solved using an efficient, scalar quasi-Newton iteration with at most $n+2$ iterations.

It can be seen in \eqref{eq:Blending} that as the elements of $\fnt{\theta}$ approach 0, the blended volume term moves towards the high order volume term. Ideally, we want the blended term $\fnt{r}\left(\fnt{\theta}\right)$ to be as close to the high order volume term as possible, subject to the cell entropy inequality and positivity constraints being satisfied. Therefore, we want to minimize the blending coefficients in some sense, subject to $\fnt{a}^{T}\fnt{\theta}\ge b$ and $0\le\fnt{\theta}\le1-\fnt{\ell^{c}}$. \textit{Linear knapsack limiting} \cite{Lin24} sets the blending coefficients as the solution to the continuous knapsack problem:
\begin{align*}\min_{\begin{matrix}\fnt{a}^{T}\fnt{\theta}\ge b\\
0\le\fnt{\theta}\le1-\fnt{\ell^{c}}\end{matrix}}\fnt{1}^{T}\fnt{\theta}\labell{eq:LinearKnapsackProblem}.\end{align*}
Since the blending coefficients are constrained to nonnegativity, the objective function amounts to an $L^{1}$ norm, as in $\fnt{1}^{T}\fnt{\theta}=\left\lVert\fnt{\theta}\right\rVert_{1}$. The linear knapsack problem was chosen due to the existence of an efficient greedy algorithm for its solution.\comment{, shown in Algorithm \eqref{alg:algorithm1}.

\begin{algorithm}
\caption{Continuous knapsack problem solution}
\label{alg:algorithm1}
\begin{algorithmic}
\REQUIRE $\fnt{a}$, $b$, $\fnt{\ell^{c}}$
\ENSURE Optimal $\fnt{\theta}$

\STATE Initialize $\fnt{\theta}\gets\fnt{0}$
\STATE \textcolor{red}{Sort items $\fnt{a}_{i}$ in descending order}
\FOR{each item $i$ in sorted order}
    \IF{$\fnt{a}_{i}\ge0$ and $b>0$}
        \STATE Assign maximum feasible $\fnt{\theta}_{i}=\min\left(1-\fnt{\ell^{c}}_{i},\frac{b}{\fnt{a}_{i}}\right)$
        \STATE Update remaining capacity $b\gets b-\fnt{a}_{i}\fnt{\theta}_{i}$
    \ELSE
        \STATE Break
    \ENDIF
\ENDFOR

\RETURN $\fnt{\theta}$
\end{algorithmic}
\end{algorithm}

} Moreover, it was shown theoretically in \cite{Vilar25} that the resulting scheme preserves high order accuracy $\mathcal{O}\left(h^{N+1}\right)$ in sufficiently smooth regions. However, the solution to \eqref{eq:LinearKnapsackProblem} is not continuous as a function of the components of the parameter vector $\fnt{a}$, due to the presence of sorting in its solution algorithm. In Figure \eqref{fig:L1Plots}, the components of the solution $\fnt{\theta}\in\mathbb{R}^{2}$ to a two-dimensional continuous knapsack problem are shown as a function of the components of the vector $\fnt{a}\in\mathbb{R}^{2}$. It can be observed that a discontinuity is formed along $\fnt{a}_{1}=\fnt{a}_{2}$, where the ordering of the vector $\fnt{a}$ changes. Because the solution of the linear knapsack problem is incorporated into the volume term of the blended solution \eqref{eq:Blending}, the discontinuity can pollute the (temporal) accuracy of the resulting scheme. As we will show, this discontinuity results in a reduced order of time convergence in cases where shock is present in the solution. We will also show numerically that the adaptive timestep count is much larger than that for ESFD.

\begin{figure}%
    \centering
    \subfloat[\centering First component]{{\includegraphics[width=6cm]{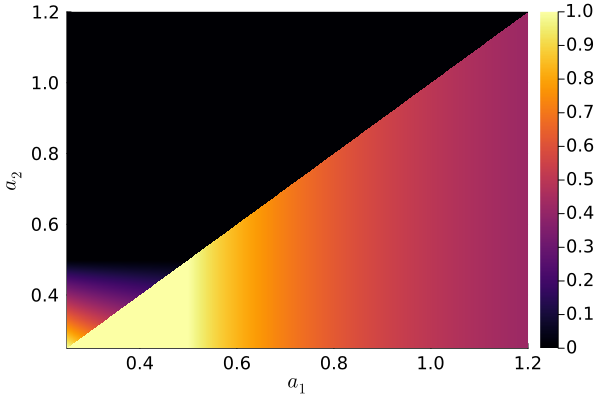} }}%
    \quad
    \subfloat[\centering Second component]{{\includegraphics[width=6cm]{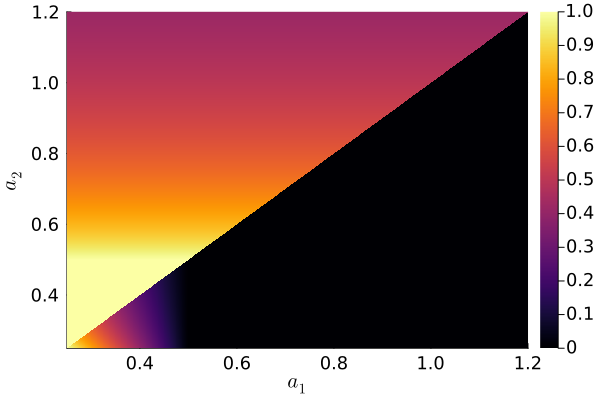} }}%
    \caption{First and second components of the solution to a two-dimensional continuous knapsack problem with weight vector $\fnt{a}$ and right-hand-side $b$.}%
    \label{fig:L1Plots}%
\end{figure}

Consider now the (continuous) \textit{quadratic} knapsack problem, which replaces the $L^{1}$ norm in the objective function with an $L^{2}$ norm:
\begin{align*}\min_{\begin{matrix}\fnt{a}^{T}\fnt{\theta}\ge b\\0\le\fnt{\theta}\le1-\fnt{\ell^{c}}\end{matrix}}
\fnt{\theta}^{T}\fnt{\theta}\labell{eq:QKP}.\end{align*}
The solution to \eqref{eq:QKP} is continuous as a function of the input parameters. This is due to the fact that the square of the $L^{2}$ norm is differentiable. This is opposed to the $L^{1}$ norm objective function for the linear knapsack problem \eqref{eq:LinearKnapsackProblem}, which is differentiable only when every component of its parameter vector is nonzero. In Figure \eqref{fig:L2Plots}, the components of the solution $\fnt{\theta}\in\mathbb{R}^{2}$ to a two-dimensional quadratic knapsack problem are shown as a function of the components of the vector $\fnt{a}\in\mathbb{R}^{2}$. It can be observed that each component is continuous as a function of $\fnt{a}$.

\begin{figure}%
    \centering
    \subfloat[\centering First component]{{\includegraphics[width=6cm]{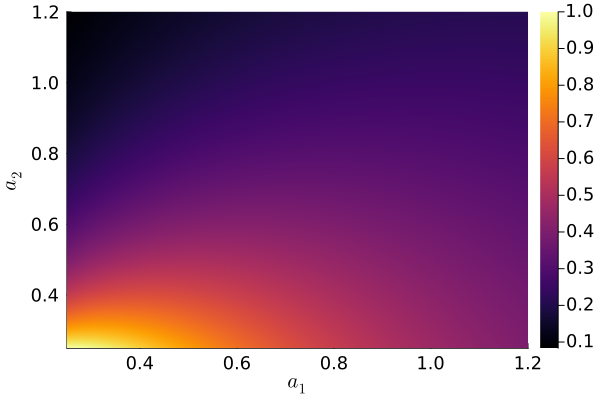} }}%
    \quad
    \subfloat[\centering Second component]{{\includegraphics[width=6cm]{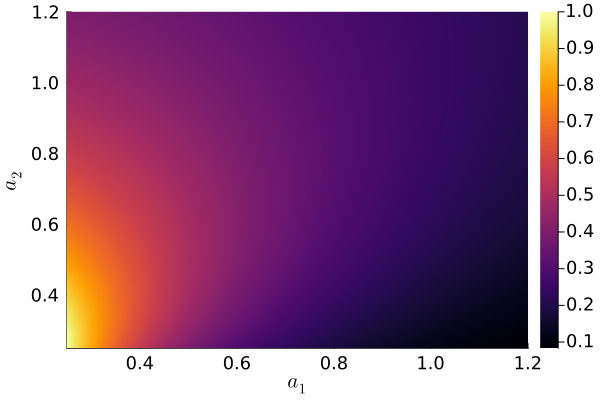} }}%
    \caption{First and second components of the solution to a two-dimensional quadratic knapsack problem with weight vector $\fnt{a}$ and right-hand-side $b$.}%
    \label{fig:L2Plots}%
\end{figure}

%FIXME some new stuff here
To illustrate how this discontinuity can result in substantial change to the blending coefficients given modest change in the vector $\fnt{a}$, we have picked a representative vector $\fnt{a}$ from a degree $N=7$ simulation of modified Sod shock tube. In this example, we perturb the vector $\fnt{a}$ randomly by a small amount. In this case, two components of $\fnt{a}$ are very close together, and the small perturbation changes which component contains the largest value. In Figure \eqref{fig:tvst} we plot both $\fnt{a}$ and $\fnt{\tilde{a}}$, the perturbed vector. We also show the corresponding blending coefficients solved by the linear and quadratic knapsack problems for both $\fnt{a}$ and $\fnt{\tilde{a}}$. We observe that with small changes in $\fnt{a}$, the linear knapsack problem can be highly sensitive, while the quadratic knapsack problem is not.

\begin{figure}%
    \centering
    \subfloat[\centering Vectors $\fnt{a}$ and $\fnt{\tilde{a}}$]{{\includegraphics[width=4.2cm]{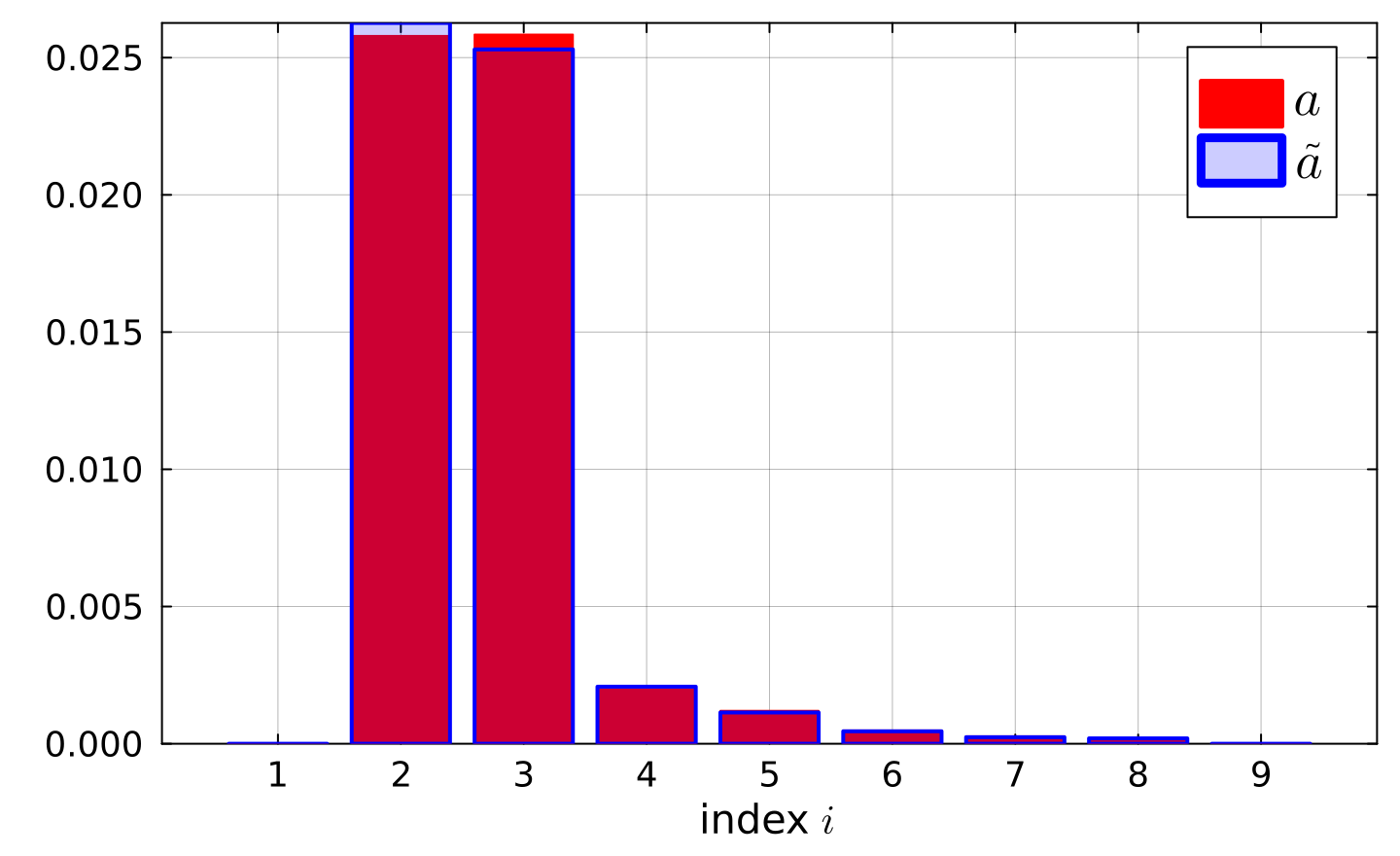} }}%
    \subfloat[\centering Linear knapsack problem]{{\includegraphics[width=4.2cm]{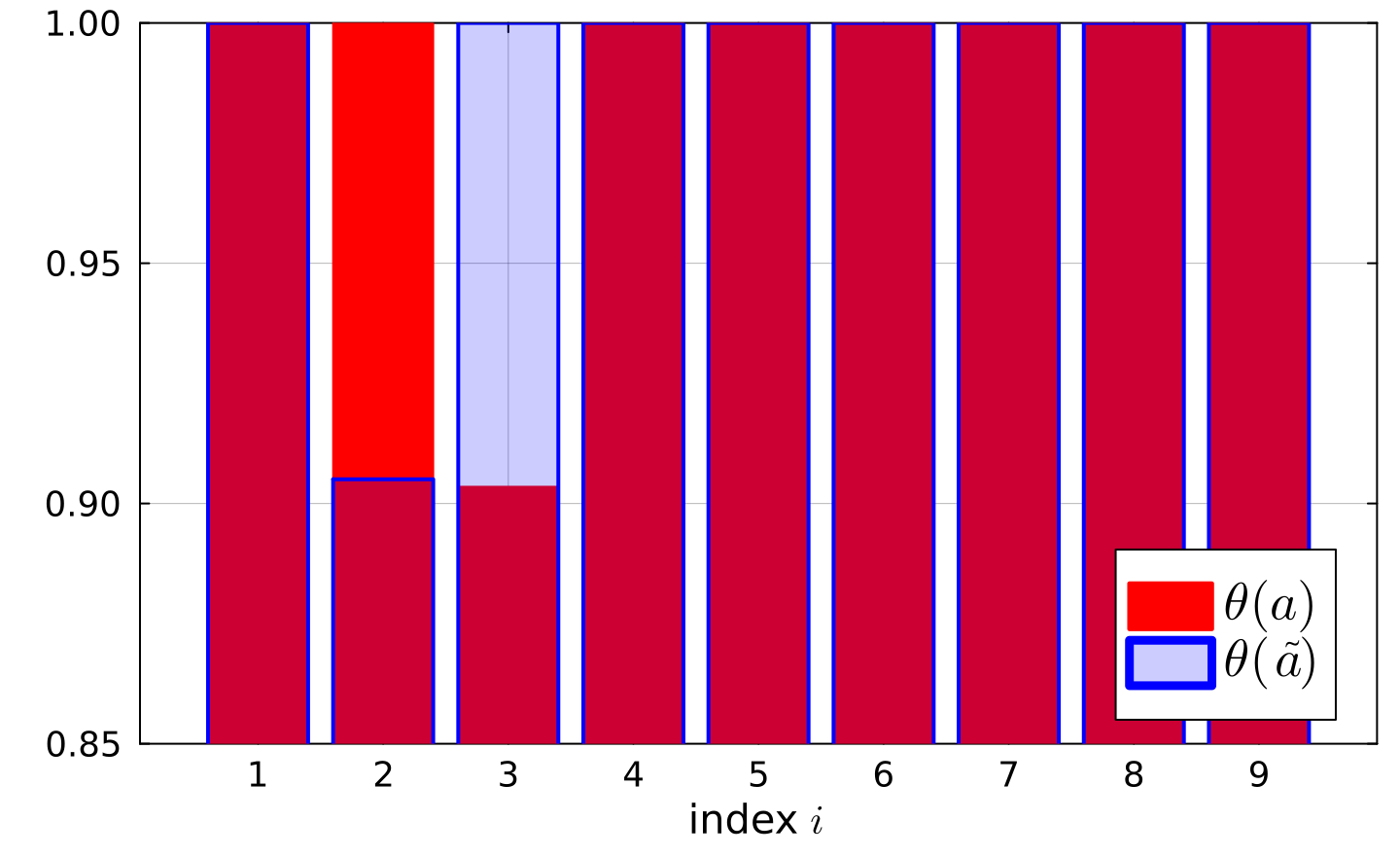} }}%
    \subfloat[\centering Quadratic knapsack problem]{{\includegraphics[width=4.2cm]{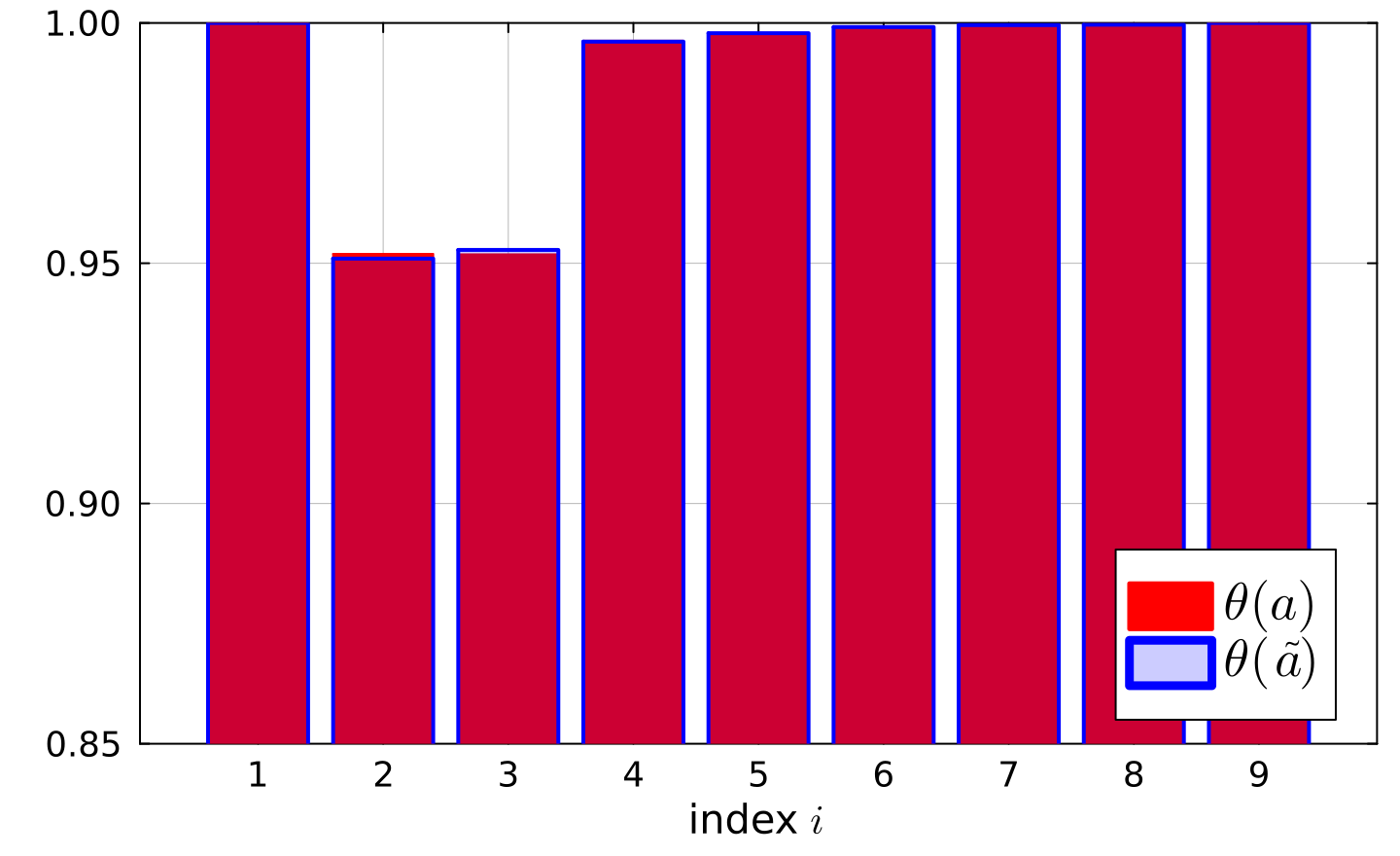} }}%
    \caption{The effect of a small perturbation on $\fnt{a}$ on the solutions of knapsack problems.}%
    \label{fig:tvst}%
\end{figure}
% FIXME end new stuff

For the remainder of this section, we prove that the solution to the quadratic knapsack problem can be found as the solution to a quasi-Newton iteration with at most $n+2$ iterations. First, note that if $b \leq 0$, $\fnt{\theta} = \fnt{0}$ is optimal for \eqref{eq:QKP}. This is due to the fact that if $b \leq 0$, $\fnt{\theta} = \fnt{0}$ is feasible. Therefore, since $\fnt{0}$ minimizes the objective function $\norm{\fnt{\theta}}_2^2$, we have that it also optimizes \eqref{eq:QKP}. For this reason, the only interesting case is when $b > 0$.

To formulate the solution to the quadratic knapsack problem as a quasi-Newton iteration, we first define a clipping function, where for arbitrary vectors $\fnt{x}, \fnt{y}, \fnt{z}\in \real^L$, $\clip_{\fnt{y},\fnt{z}}(\fnt{x})$ clips the components of $\fnt{x}$ elementwise between the components of $\fnt{y}$ and $\fnt{z}$. 
% \begin{align*}\clip_{\fnt{y},\fnt{z}}\left(\fnt{x}\right)_{i}=\begin{cases}\fnt{x}_{i}&\fnt{y}_{i}\le\fnt{x}_{i}\le\fnt{z}_{i}\\
% \fnt{y}_{i}&\fnt{x}_{i}<\fnt{y}_{i}\\
% \fnt{z}_{i}&\fnt{z}_{i}<\fnt{x}_{i}\end{cases}.\end{align*}

Define $f:\mathbb{R}_{\ge0}\to\mathbb{R}$ under $f\left(\lambda\right)=\fnt{a}^{T}\clip_{0,1 - \fnt{\ell^c}}(\lambda \fnt{a})-b$. Note then that $f$ is continuous and piecewise linear with finitely many points of nondifferentiability. This follows immediately from the facts that $\clip_{i}$ is also continuous and piecewise linear with at most two points of nondifferentiability, and $f$ is a linear combination of $\clip_{i}$ functions. Importantly, one may also show using properties of $\clip$ that $f$ is right-differentiable, with right-derivative $\delta f:\mathbb{R}_{\ge0}\to\mathbb{R}$ under
\begin{align*}\delta f\left(\lambda\right)=\sum_{0<\fnt{a}_{i}\ \text{and}\ \lambda\fnt{a}_{i}<1-\fnt{\ell^{c}}_{i}}^{ }\fnt{a}_{i}^{2}.\end{align*}

\begin{lemma}
    $f$ is nondecreasing and concave.
\end{lemma}
\begin{proof}
    We have that $f$ is everywhere right-differentiable, and that the right-derivative can be expressed as a filtered sum of nonnegative components. Therefore, $f$ is nondecreasing. Moreover, the filter condition $0<\fnt{a}_{i}$ and $\lambda\fnt{a}_{i}<1-\fnt{\ell^{c}}_{i}$ implies that as $\lambda$ increases, the number of positive components being added decreases. Therefore, the right-derivative of $f$ is nonincreasing. Therefore, $f$ is concave.
\end{proof}

We now define the quasi-Newton iteration as follows:
\begin{align*}\lambda_{0}\aligned{=}0 &
\lambda_{k+1}\aligned{=}\lambda_{k}-\frac{f\left(\lambda_{k}\right)}{\delta f\left(\lambda_{k}\right)}\labell{eq:NewtonIteration}.\end{align*}

The following theorem establishes both convergence and equivalence of \eqref{eq:NewtonIteration} to the quadratic knapsack problem \eqref{eq:QKP}.

\begin{theorem} \label{theorem:Reduction}
    If $b>0$, the quasi-Newton iteration \eqref{eq:NewtonIteration} converges in at most $L+1$ iterations to the root $\lambda_{\ast}$ such that $\clip_{0,1 - \fnt{\ell^c}}(\lambda_\ast \fnt{a})$ solves the quadratic knapsack problem \eqref{eq:QKP}. 
\end{theorem}
\begin{proof}

Consider a different form of the knapsack problem \eqref{eq:QKP}, which modifies the objective function:
\begin{align*}\min_{\begin{matrix}\fnt{a}^{T}\fnt{\theta}\ge b\\
0\le\fnt{\theta}\le1-\fnt{\ell^{c}}\end{matrix}}\frac{1}{2}\fnt{\theta}^{T}\fnt{\theta}-\lambda\left(\fnt{a}^{T}\fnt{\theta}-b\right).\labell{eq:EQKP}\end{align*}
We will show in this proof that this problem, which has an explicit solution as a function of $\lambda$, can be used to construct a solution for \eqref{eq:QKP}. It can be shown that an optimal solution for \eqref{eq:EQKP} has the form:
\begin{align*}\fnt{\theta}\left(\lambda\right)=\clip_{0,1-\fnt{\ell^{c}}}\left(\lambda\fnt{a}\right).\end{align*}
It can be shown that there exists $\lambda_{\ast}>0$ for which $f\left(\lambda_{\ast}\right)=\fnt{a}^{T}\fnt{\theta}\left(\lambda_{\ast}\right)-b=0$. In this case, $\fnt{\theta}_{\ast}=\fnt{\theta}\left(\lambda_{\ast}\right)$ is feasible for \eqref{eq:QKP}. Let $\fnt{\theta}\in\mathbb{R}^{L}$ be feasible for \eqref{eq:QKP}, and thus also \eqref{eq:EQKP}. Since $\fnt{\theta}_{\ast}$ is optimal for \eqref{eq:EQKP}, it has a smaller objective function value:
\begin{align*}\frac{1}{2}\fnt{\theta}_{\ast}^{T}\fnt{\theta}_{\ast}-\lambda\ub{\left(\fnt{a}^{T}\fnt{\theta}\left(\lambda_{\ast}\right)-b\right)}{f\left(\lambda_{\ast}\right)=0}\aligned{\le}\frac{1}{2}\fnt{\theta}^{T}\fnt{\theta}-\lambda\ub{\left(\fnt{a}^{T}\fnt{\theta}-b\right)}{\ge0}\\
\fnt{\theta}_{\ast}^{T}\fnt{\theta}_{\ast}\aligned{\le}\fnt{\theta}^{T}\fnt{\theta}.\end{align*}
Therefore, $\fnt{\theta}_{\ast}$ solves the quadratic knapsack problem \eqref{eq:QKP}. Since $f$ is concave and nondecreasing, it can be shown that the iterands $\lambda_{k}$ of the iteration \eqref{eq:NewtonIteration} are nondecreasing and lie within $\left[0,\lambda_{\ast}\right]$. Note also that \eqref{eq:NewtonIteration} performs a quasi-Newton iteration on $f$ with the right derivative of $f$. Thus, since $f$ is piecewise-linear, the number of iterations is bounded by the number of points of nondifferentiability of $f$ between $\left[0,\lambda_{\ast}\right]$. Every $\clip$ function which composes $f$ has at most two points of nondifferentiability: one which lies at $0$, and one which is positive. Therefore, since $f$ is composed of $L$ total clip functions, there are at most $L+1$ points of nondifferentiability of $f$ within $\left[0,\lambda_{\ast}\right]$. Thus, the iteration converges to the root of $f$ within $L+1$ iterations.

In summary, if $b>0$, the quasi-Newton iteration \eqref{eq:NewtonIteration} converges in at most $L+1$ iterations to the root $\lambda_{\ast}$ such that $\clip_{0,1-\fnt{\ell^{c}}}\left(\lambda_{\ast}\fnt{a}\right)$ solves the quadratic knapsack problem \eqref{eq:QKP}. 
\end{proof}

Note, the resulting algorithm for the solution of the continuous quadratic knapsack problem is similar to the one taken in \cite{Cominetti14}. In our case, the elementwise-positive bounds on the blending coefficients impose an extra restriction on the resulting function $f$. In particular, $f$ becomes nondecreasing and concave. This allows us to make a simplification to the algorithm which improves the runtime and memory usage, while also simplifying the resulting code. 

In summary, if $b \leq 0$, then $\fnt{\theta} = \fnt{0}$ solves \eqref{eq:QKP}. Otherwise, per Theorem \eqref{theorem:Reduction}, the solution to \eqref{eq:QKP} reduces to a scalar quasi-Newton rootfinding problem, which terminates in at most $L+1$ iterations, $L$ being the size of the knapsack problem. Therefore, \eqref{eq:QKP} can be solved with Algorithm \eqref{alg:algorithm2}. 

\begin{algorithm}
    \caption{Continuous quadratic knapsack problem solution.}
    \label{alg:algorithm2}
    \begin{algorithmic}
        \STATE{Given $\fnt{a},\fnt{\ell^{c}},b,tol$}
        \IF{$b\le0$}
            \RETURN{\fnt{0}}
        \ELSE
            \STATE{$\lambda_0 \gets 0$}
            \FOR{$k = 0...L$}
                \STATE{$\lambda_{k+1} \gets \lambda_k - \frac{f\left(\lambda_k\right)}{\delta f\left(\lambda_k\right)}$}
                \IF{$-f\left(\lambda_{k+1}\right)<tol$}
                    \RETURN{$\clip_{0,1-\fnt{\ell^{c}}}\left(\lambda_{k+1}\fnt{a}\right)$}
                \ENDIF
            \ENDFOR
        \ENDIF
    \end{algorithmic}
\end{algorithm}

Note that, though the iteration does provably converge in $n+2=L+1$ or fewer iterations, it very rarely takes more than $4$ total iterations, even up to $N = 15$, for our purposes. Therefore, since the computation time of both $f$ and $\delta f$ is linear in $L$, we can treat the quadratic knapsack algorithm as a roughly $\mathcal{O}\left(L\right)\sim\mathcal{O}\left(N^d\right) \sim \mathcal{O}(n)$ algorithm.

\subsection{Positivity Preservation} \label{sec:Positivity Preservation}

Recall, assuming admissible quantities, such as positive pressure and density for the compressible Euler equations, a scheme which satisfies a discrete entropy inequality such as \eqref{eq:CEI} typically behaves much more robustly. Moreover, the compressible Euler entropy function \cite{Godlewski13} is convex under the assumption that the density and pressure are positive. Therefore, we are motivated to consider schemes which satisfy positivity of physical variables. In this section, we will discuss a particular strategy for choosing $\fnt{\ell^{c}}$, the limiting coefficients for the knapsack problem, to ensure that the scheme satisfies a positivity condition for the compressible Euler equations. 

% this requires dt sufficient small for this to be feasible, it is often hard to do this in practice (cumbersome to compute how small it should be). Heuristically, relpos is something we observed that improves robustness in practice. in c1, we can compute timestep, ends up being relative to soln, as soln \to 0, so does timestep. Relpos seems to do better.
We assume that the explicit timestepper being used is strong stability preserving (SSP), so that the Runge-Kutta update is a convex combination of forward Euler stages. This way, enforcing positivity at each stage is sufficient to enforce positivity of the full Runge-Kutta update. Therefore, a potential inequality that could be used to satisfy positivity of a scalar field $u$ with pointwise evaluations $\fnt{u}$ might be $\fnt{u}+\Delta t\frac{\d\fnt{u}}{\d t}\ge0$, assuming $\Delta t$ is sufficiently small. However, it is often hard to compute how small $\Delta t$ should be in practice. Moreover, it is often more practical to compute the necessary timestep relative to the low order solution. Therefore, we often instead seek to satisfy
\begin{align*}\fnt{u}+\Delta t\frac{\d \fnt{u}}{\d t}\ge\alpha\fnt{u}^{L}=\alpha\left(\fnt{u}+\Delta t\frac{\d \fnt{u}^{L}}{\d t}\right)\labell{eq:WeakPositivity},\end{align*}
where $\alpha\in\left[0,1\right]$. The resulting inequality is called a \textit{relative} positivity constraint. Since $\alpha\in\left[0,1\right]$, the low order method is feasible for the above inequality. This is important to ensure that the quadratic knapsack problem \eqref{eq:QKP} remains feasible. One can show that the following limiting coefficients $\fnt{\ell^c}$, when coupled with the subcell limiting operators \eqref{eq:BasicSLO} implies \eqref{eq:WeakPositivity}:
% \begin{align*}\fnt{\ell^{c}}_{1}\aligned{=}1-C\left(\frac{1-\alpha}{-2\fnt{c}_{1}}\hat{\fnt{h}}_{1}\right)\\
% \forall i\in\left\{2,...,n\right\},\ \fnt{\ell^{c}}_{i}\aligned{=}1-\min\left(C\left(\frac{1-\alpha}{2\fnt{c}_{i}}\hat{\fnt{h}}_{i-1}\right),C\left(\frac{1-\alpha}{-2\fnt{c}_{i}}\hat{\fnt{h}}_{i}\right)\right)\\
% \fnt{\ell^{c}}_{n+1}\aligned{=}1-C\left(\frac{1-\alpha}{2\fnt{c}_{n+1}}\hat{\fnt{h}}_{n}\right),\end{align*}
% where
% \begin{align*}C\left(a\right)&=\sup \left\{\ell \in [0,1]:\frac{\ell}{a} \leq 1\right\}, & \fnt{s} &= \fnt{E}^T\vecf{f}^{\ast }\left(\fnt{u_f}^+,\fnt{u_f}, \hat{\fnt{n}}\right),\\
% \hat{\fnt{h}} &= \frac{1}{\Delta t}\fnt{Mu} - \fnt{s} - \fnt{r}^L,&
%     \fnt{c} &= \fnt{R}\left(\fnt{r}^H - \fnt{r}^L\right).\labell{eq:c}\end{align*}

\begin{align*}\fnt{\ell^{c}}_{1}\aligned{=}1-\fnt{\ell}^\text{right}_1\\
\forall i\in\left\{2,...,n\right\},\ \fnt{\ell^{c}}_{i}\aligned{=}1-\min\left(\fnt{\ell}^\text{left}_i,\fnt{\ell}^\text{right}_i\right)\\
\fnt{\ell^{c}}_{n+1}\aligned{=}1-\fnt{\ell}^\text{left}_{n+1},\end{align*}
where
\begin{align*}\fnt{\ell}^\text{left}_i = \begin{cases}
    \frac{1 - \alpha}{2 \fnt{c}_{i}}\hat{\fnt{h}}_{i-1} & \fnt{c}_{i} > 0 \\ 0 & \text{no limiting necessary}
\end{cases} && \fnt{\ell}^\text{right}_i = \begin{cases}
    \frac{1 - \alpha}{-2 \fnt{c}_{i}}\hat{\fnt{h}}_i & \fnt{c}_{i} < 0 \\ 0 & \text{no limiting necessary}
\end{cases}\end{align*}\begin{align*}\fnt{s} = \fnt{E}^T\vecf{f}^{\ast }\left(\fnt{u_f}^+,\fnt{u_f}, \hat{\fnt{n}}\right)&&
\hat{\fnt{h}} = \frac{1}{\Delta t}\fnt{Mu} - \fnt{s} - \fnt{r}^L&&
    \fnt{c} = \fnt{R}\left(\fnt{r}^H - \fnt{r}^L\right).\labell{eq:c}\end{align*}

In doing this, we ensure that the blended scheme \eqref{eq:BlendedScheme} satisfies the relative positivity constraint \eqref{eq:WeakPositivity}. In the case that there are multiple variables whose positivity should be preserved (such as the case for the compressible Euler equations, where positivity of pressure and density must be enforced,) limiting coefficients for each individual variable must first be determined. Then, the limiting coefficients chosen for \eqref{eq:QKP} should be the elementwise maximum of each set of limiting coefficients. 

By equation \eqref{eq:Blending}, as the limiting coefficients approach 1, more of the low order volume term $\fnt{r}^L$ will be added to the blended volume term. So when the solution tends towards zero the limiting coefficients rise, which moves the blended volume term towards the low order volume term adding extra dissipation.

\section{Numerical Experiments}

In this section, we start by examining the spatial convergence of the proposed quadratic knapsack limiting scheme (QK) compared against linear knapsack limiting (LK), the discontinuous Galerkin spectral element method (DGSEM), and flux differencing with an entropy conservative volume flux (ESFD). Then, we analyze the time convergence of QK against LK and ESFD in shock-type simulations, where some form of stabilization is required to avoid crashing. Then, we consider the local linear stability of knapsack limiting against ESFD, which is known to be linearly unstable \cite{Gassner21}. Last, we combine the positivity preserving scheme derived in Section \eqref{sec:Positivity Preservation} with QK, and analyze its behavior for the Leblanc shock tube, Kelvin-Helmholtz instability, and Sedov blastwave. 

All simulations utilize the compressible Euler equations. The entropy conservative volume flux chosen for ESFD is flux Ranocha \cite{Ranocha20}. The numerical flux for each scheme is the local Lax-Friedrichs flux (LxF) \eqref{eq:LaxFriedrichs}, unless otherwise specified. We use various timesteppers throughout this section, including \texttt{RK4}, a 4-stage 4th order method, and \texttt{SSPRK43}, a strong stability preserving (SSP) 4-stage, 3rd order method. We specify for each experiment whether the timesteppers used a fixed or adaptive timestep, which are both implemented in \texttt{OrdinaryDiffEq.jl} \cite{OrdinaryDiffEq}. Similarly, we use implementations from \texttt{Trixi.jl} \cite{Trixi1, Trixi3} for all quantities related to the compressible Euler equations.

\subsection{Spatial Convergence}

In this section, we examine the spatial convergence of quadratic knapsack limiting, compared against that for linear knapsack limiting, the discontinuous Galerkin spectral element method, and flux differencing with an entropy conservative volume flux. We also visually examine solutions generated by quadratic knapsack limiting. We start by numerically estimating the spatial convergence of a 1D density wave with a smooth initial condition. Then, we show plots for the modified Sod shock tube problem and the Shu-Osher shock tube problem. Last, we compare solutions generated for the Shu-Osher shock tube problem when using the local Lax-Friedrichs flux against using the HLLC flux \cite{Harten83}.

\subsubsection{1D Density Wave Spatial Convergence}

First, we test the high order spatial convergence of our scheme using a smooth 1D density wave. The analytic solution for the 1D compressible Euler density wave is as follows:

\begin{align*}p\left(x,t\right)=1,&&
v\left(x,t\right)=1.7,&&
\rho\left(x,t\right)=\frac{1}{2}\sin\left(\pi \left(x - 1.7t\right)\right)+1.\end{align*}

The density wave solution is computed using the \texttt{SSPRK43} timestepper with a fixed $\Delta t=10^{-4}$ to a final time of $T=1$, over the periodic domain $\left[-1,1\right]$. 

In Table \eqref{tab:PaperConvergence}, the $L^2$ norm of the solution at the final time is computed against the analytical solution for various element counts $M$. This is done for polynomial orders $N=1$ through $N=4$. The errors are used to estimate the spatial order of convergence. We observe that the optimal order of convergence $\mathcal{O}\left(\Delta x^{N+1}\right)$ is attained by the quadratic knapsack limiter. In order to justify that the limiter is active in these convergence tests, verifying that this is a meaningful result, we also show in Table \eqref{tab:thetas} the minimum norm of $\fnt{\Theta}$ over all right-hand side evaluations, as well as for how many evaluations the norm was equal to 0. $\fnt{\Theta}$ is defined as the concatenation of all blending coefficients across the domain.

\begin{table}[h!]
\centering
\resizebox{\textwidth}{!}{
\begin{tabular}{{|c|c|c|c|c|c|c|c|c|}}
\hline
$$ & \multicolumn{2}{|c|}{$N=1$} & \multicolumn{2}{|c|}{$N=2$} & \multicolumn{2}{|c|}{$N=3$} & \multicolumn{2}{|c|}{$N=4$} \\
\hline
$M$ & \multicolumn{1}{|c}{$L^{2}\ \text{Error}$} & $\text{Rate}$ & \multicolumn{1}{|c}{$L^{2}\ \text{Error}$} & $\text{Rate}$ & \multicolumn{1}{|c}{$L^{2}\ \text{Error}$} & $\text{Rate}$ & \multicolumn{1}{|c}{$L^{2}\ \text{Error}$} & $\text{Rate}$ \\
\hline
$2$ & \multicolumn{1}{|c}{$1.22$} & $-$ & \multicolumn{1}{|c}{$6.18\cdot10^{-1}$} & $-$ & \multicolumn{1}{|c}{$1.03\cdot10^{-1}$} & $-$ & \multicolumn{1}{|c}{$1.38\cdot10^{-2}$} & $-$ \\
\hline
$4$ & \multicolumn{1}{|c}{$1.23$} & $-3.03\cdot10^{-3}$ & \multicolumn{1}{|c}{$7.76\cdot10^{-2}$} & $2.99$ & \multicolumn{1}{|c}{$6.54\cdot10^{-3}$} & $3.98$ & \multicolumn{1}{|c}{$4.01\cdot10^{-4}$} & $5.11$ \\
\hline
$8$ & \multicolumn{1}{|c}{$5.37\cdot10^{-1}$} & $1.19$ & \multicolumn{1}{|c}{$9.96\cdot10^{-3}$} & $2.96$ & \multicolumn{1}{|c}{$2.84\cdot10^{-4}$} & $4.53$ & \multicolumn{1}{|c}{$1.39\cdot10^{-5}$} & $4.85$ \\
\hline
$16$ & \multicolumn{1}{|c}{$1.60\cdot10^{-1}$} & $1.75$ & \multicolumn{1}{|c}{$1.35\cdot10^{-3}$} & $2.88$ & \multicolumn{1}{|c}{$1.58\cdot10^{-5}$} & $4.17$ & \multicolumn{1}{|c}{$5.00\cdot10^{-7}$} & $4.80$ \\
\hline
$32$ & \multicolumn{1}{|c}{$4.20\cdot10^{-2}$} & $1.93$ & \multicolumn{1}{|c}{$1.71\cdot10^{-4}$} & $2.99$ & \multicolumn{1}{|c}{$9.35\cdot10^{-7}$} & $4.08$ & \multicolumn{1}{|c}{$1.63\cdot10^{-8}$} & $4.94$ \\
\hline
\end{tabular}
}
\caption{Errors and estimated orders of convergence for various degrees $N$.}
\label{tab:PaperConvergence}
\end{table}

\begin{table}[h!]
\centering
\resizebox{\textwidth}{!}{
\begin{tabular}{{|c|c|c|c|c|c|c|c|c|}}
\hline
$$ & \multicolumn{2}{|c|}{$N=1$} & \multicolumn{2}{|c|}{$N=2$} & \multicolumn{2}{|c|}{$N=3$} & \multicolumn{2}{|c|}{$N=4$} \\
\hline
$M$ & \multicolumn{1}{|c}{$\min_t \|\Theta\|_{\infty}$} & $\#\{\Theta = 0\}$ & \multicolumn{1}{|c}{$\min_t \|\Theta\|_{\infty}$} & $\#\{\Theta = 0\}$ & \multicolumn{1}{|c}{$\min_t \|\Theta\|_{\infty}$} & $\#\{\Theta = 0\}$ & \multicolumn{1}{|c}{$\min_t \|\Theta\|_{\infty}$} & $\#\{\Theta = 0\}$ \\
\hline
$2$ & \multicolumn{1}{|c}{$0$} & $40004\ /\ 40004$ & \multicolumn{1}{|c}{$0$} & $993\ /\ 40004$ & \multicolumn{1}{|c}{$0$} & $4429\ /\ 40004$ & \multicolumn{1}{|c}{$0$} & $9414\ /\ 40004$ \\
\hline
$4$ & \multicolumn{1}{|c}{$2.16\cdot10^{-2}$} & $0\ /\ 40004$ & \multicolumn{1}{|c}{$7.05\cdot10^{-3}$} & $0\ /\ 40004$ & \multicolumn{1}{|c}{$4.19\cdot10^{-4}$} & $0\ /\ 40004$ & \multicolumn{1}{|c}{$5.53\cdot10^{-6}$} & $0\ /\ 40004$ \\
\hline
$8$ & \multicolumn{1}{|c}{$2.86\cdot10^{-2}$} & $0\ /\ 40004$ & \multicolumn{1}{|c}{$2.70\cdot10^{-3}$} & $0\ /\ 40004$ & \multicolumn{1}{|c}{$5.59\cdot10^{-5}$} & $0\ /\ 40004$ & \multicolumn{1}{|c}{$3.00\cdot10^{-7}$} & $0\ /\ 40004$ \\
\hline
$16$ & \multicolumn{1}{|c}{$1.88\cdot10^{-2}$} & $0\ /\ 40004$ & \multicolumn{1}{|c}{$8.70\cdot10^{-4}$} & $0\ /\ 40004$ & \multicolumn{1}{|c}{$5.59\cdot10^{-6}$} & $0\ /\ 40004$ & \multicolumn{1}{|c}{$3.99\cdot10^{-8}$} & $0\ /\ 40004$ \\
\hline
$32$ & \multicolumn{1}{|c}{$1.02\cdot10^{-2}$} & $0\ /\ 40004$ & \multicolumn{1}{|c}{$3.58\cdot10^{-4}$} & $0\ /\ 40004$ & \multicolumn{1}{|c}{$7.72\cdot10^{-7}$} & $0\ /\ 40004$ & \multicolumn{1}{|c}{$1.82\cdot10^{-9}$} & $0\ /\ 40004$ \\
\hline
\end{tabular}
}
\caption{Minimum over all right hand side evaluations of $\|\Theta\|_\infty$, and
occurrences of $\Theta = 0$.}
\label{tab:thetas}
\end{table}

% In Figure \eqref{fig:Convergence}, the $L^{2}$ norm of the solution at the final time is computed against the analytical solution, and is plotted in a log-log plot against varying element widths $\Delta x$. This is done for polynomial orders $N=1$ through $N=4$. In Table \eqref{fig:ConvergenceTable}, the slopes of each line are estimated using a line of best fit, for $N=1$ through $4$. These slopes estimate the order of convergence of the solution for each $N$. We observe that the optimal order of convergence $\mathcal{O}\left(\Delta x^{N+1}\right)$ is attained by the quadratic knapsack limiter. 

% \begin{figure}[h]
% \centering
% \includegraphics[width=0.8\textwidth]{convergence.png}
% \caption{$L^2$ error of density wave simulation at final time $T=1$ against mesh width $\Delta x$, for various polynomial orders $N$.}
% \label{fig:Convergence}
% \end{figure}

% \begin{table}[h]
%     \centering
%     \begin{tabular}{|c|c|c|}
%         \hline
%         $N$ & \textbf{Estimated Order} & \textbf{Optimal Order} \\
%         \hline
%         1 & 1.96924 & $\mathcal{O}\left(\Delta x^2\right)$\\
%         2 & 2.9927 & $\mathcal{O}\left(\Delta x^3\right)$\\
%         3 & 4.12389 & $\mathcal{O}\left(\Delta x^4\right)$\\
%         4 & 4.80817 & $\mathcal{O}\left(\Delta x^5\right)$\\
%         \hline
%     \end{tabular}
%     \caption{Estimated orders of convergence for various degrees $N$ against the optimal order of convergence for DGSEM}
% \label{fig:ConvergenceTable}
% \end{table}

\subsubsection{Modified Sod Shock Tube}

Next, we examine the solutions generated in the case when some form of stabilization is required to avoid crashing. To do this, we use the modified Sod shock tube initial condition for the 1D compressible Euler equations:
\begin{align*}p\left(x,0\right)=\begin{cases}1&x<.3\\
    .1&\text{otherwise}\end{cases},&&
    v\left(x,0\right)=\begin{cases}.75&x<.3\\
    0&\text{otherwise}\end{cases},&&
    \rho\left(x,0\right)=\begin{cases}1&x<.3\\
    .125&\text{otherwise}\end{cases}.\end{align*}
The modified Sod shock tube solution is computed using the adaptive \texttt{SSPRK43} timestepper with the absolute and relative tolerances set to $10^{-6}$ and $10^{-4}$ respectively. The simulation is performed to a final time of $T=.2$, over the domain $[0,1]$, with Dirichlet boundary conditions, enforcing that the solution remain constant at the boundaries over time. In Figure \eqref{fig:ModSodPlots}, the density at the final time computed using LK, QK, and ESFD for $(N, M) = (3, 100)$ and $(N, M) = (7, 50)$ are shown. The modified Sod shock tube problem is commonly used to test entropy stable methods, because it requires some form of stabilization to converge to the entropy solution which suppresses an entropy glitch at $x = .3$ \cite{Lin24}. For instance, for the parameters above, DGSEM crashes prior to the final time. We observe that LK and QK obtain similar solutions for all parameters. However, the solution generated by ESFD has strong oscillations that dominate the solution. This will be further discussed in Section \eqref{sec:LinearStability}.

\begin{figure}%
    \centering
    \subfloat[\centering Linear knapsack limiting]{{\includegraphics[width=4.2cm]{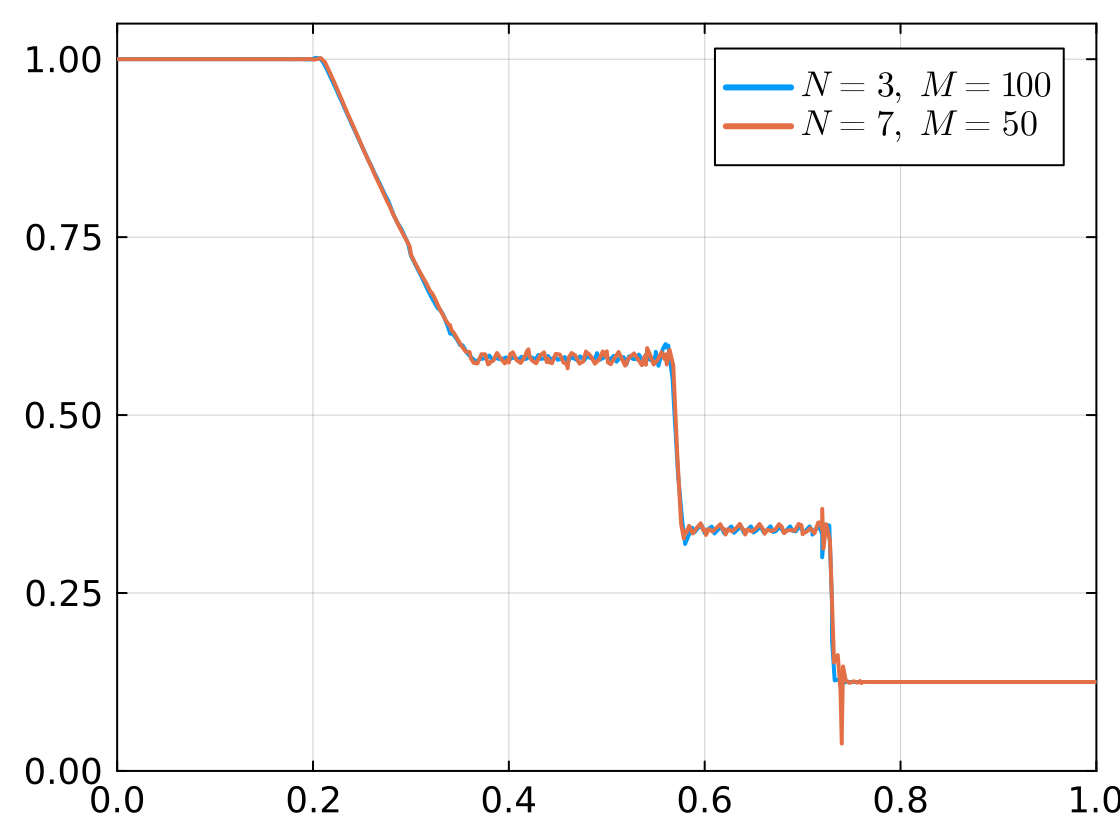} }}%
    \subfloat[\centering Quadratic knapsack limiting]{{\includegraphics[width=4.2cm]{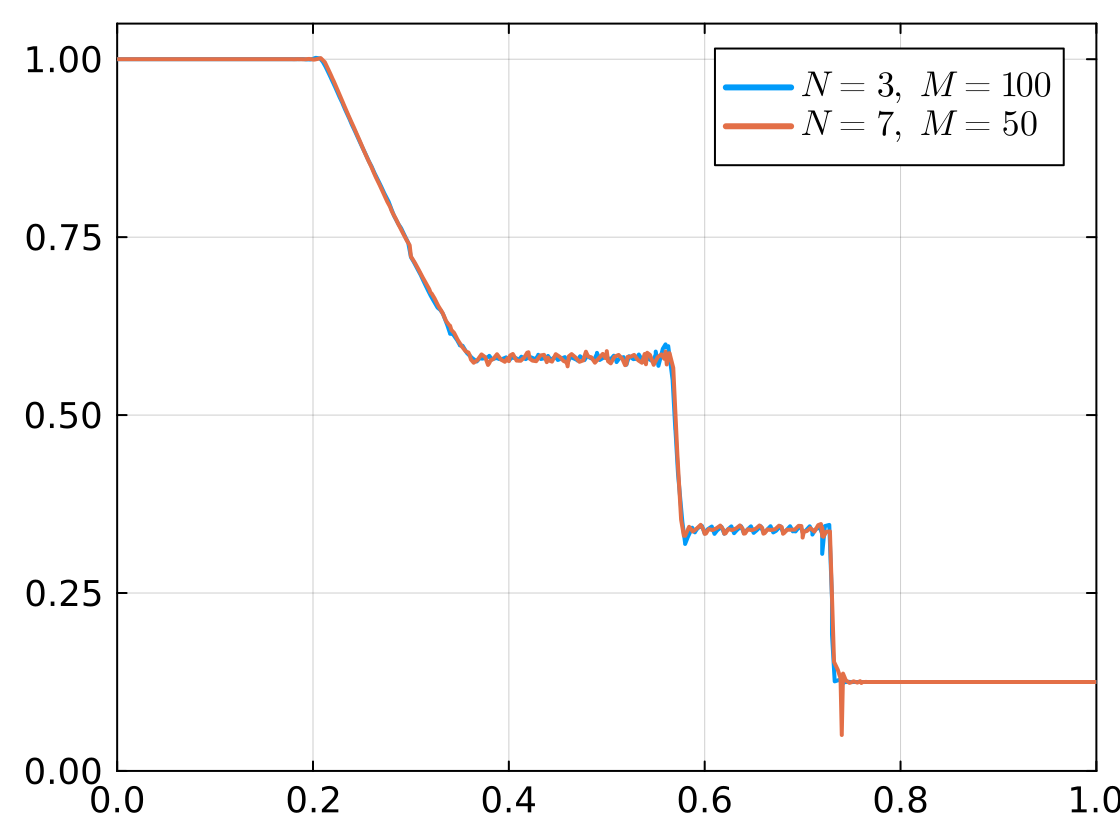} }}%
    \subfloat[\centering ESFD]{{\includegraphics[width=4.2cm]{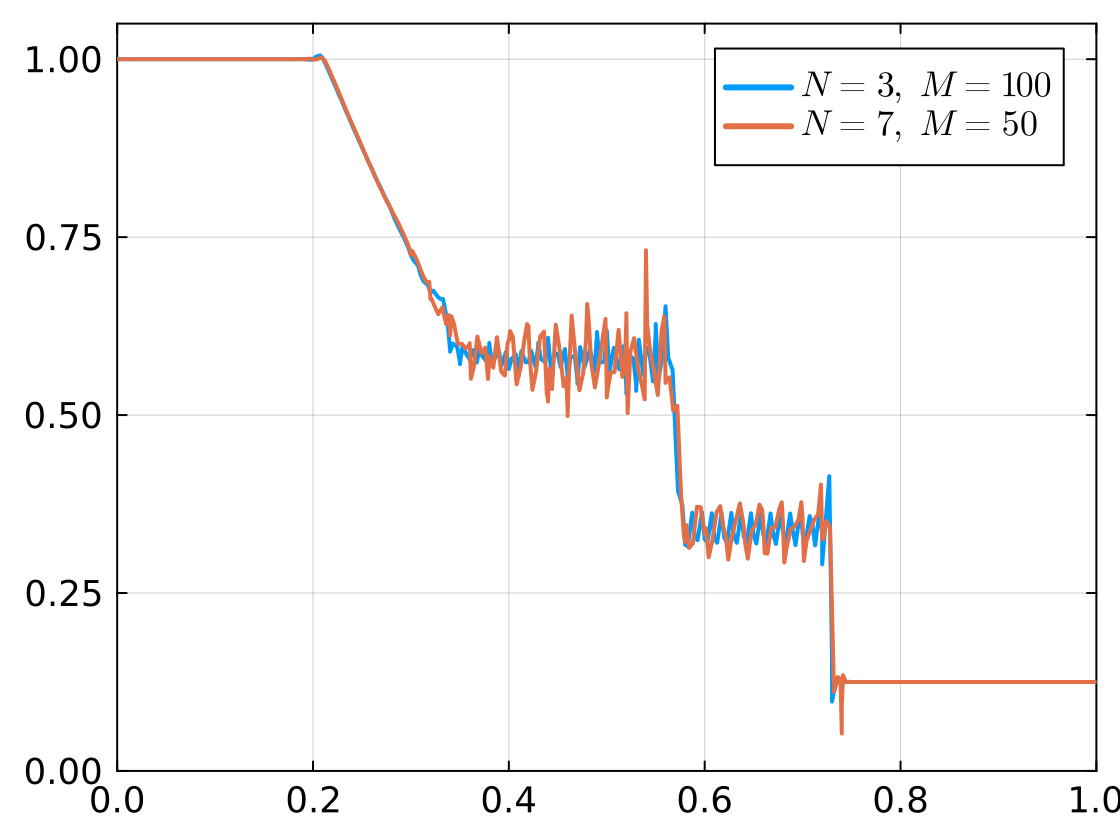} }
    \label{fig:ModSodLinearInstability}}%
    \caption{Density of the modified Sod shock tube solution at final time $T=.2$.}%
    \label{fig:ModSodPlots}%
\end{figure}

In Figure \eqref{fig:troubled}, we display elementwise $\ell^2$ norms of the blending coefficients at the final iteration of the modified Sod shock tube simulations for $N=3$ and $M=100$. We observe a spike in the blending coefficients around the rightmost shock, as well as a cluster of relatively large blending coefficients around the rightmost plateau. This is the case for both linear and quadratic knapsack limiting.

\begin{figure}%
    \centering
    \subfloat[\centering Linear knapsack limiting]{{\includegraphics[width=6cm]{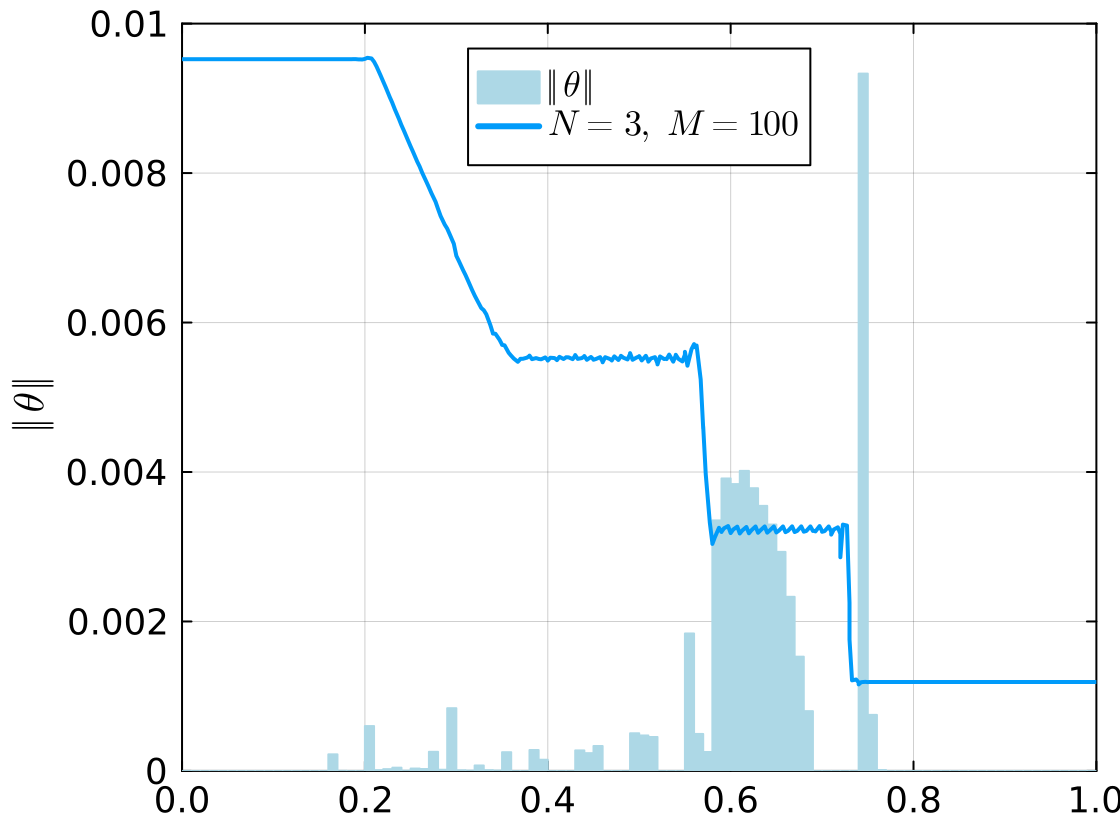} }}%
    \qquad
    \subfloat[\centering Quadratic knapsack limiting]{{\includegraphics[width=6cm]{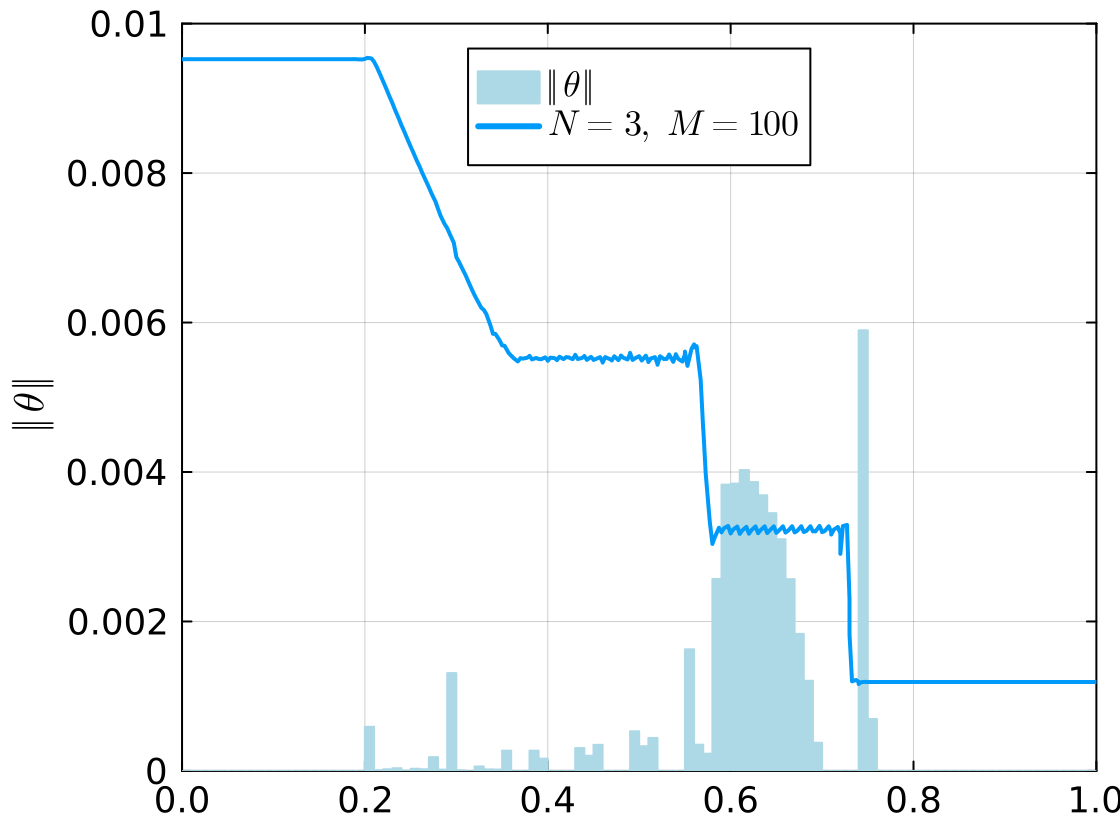} }}%
    \caption{Normalized $\ell^2$ norm of the blending coefficients of each element.}%
    \label{fig:troubled}%
\end{figure}

\subsubsection{Shu-Osher Shock Tube}
Last, we examine the solutions generated for the Shu-Osher shock tube, a more difficult version of the modified Sod shock tube, whose solution contains a stronger shock and a mixture of smooth and discontinuous solution features. For this section, we will consider the \textit{HLLC version} of each knapsack limiting scheme to be the one where the low order volume flux and each numerical flux are the HLLC flux instead of the LxF flux. Note then that the \textit{HLLC version} of the low order method is still positivity preserving for the compressible Euler and Navier-Stokes equations \cite{Batten97}, and therefore the blended scheme still can be constructed in a way that positivity is preserved. The \textit{HLLC version} of the ESFD method will replace the LxF numerical flux with the HLLC numerical flux. The initial condition for the 1D compressible Euler Shu-Osher shock tube is as follows:
\begin{align*}
(\rho(x,0), v(x,0), p(x,0)) = \begin{cases}
(3.857143, 2.629369, 10.3333) & x < -4\\
(1+.2\sin\left(5x\right), 0, 1) & \text{otherwise}.
\end{cases}
\end{align*}

The Shu-Osher solution is computed using the adaptive \texttt{SSPRK43} timestepper, with the absolute and relative tolerances set to $10^{-6}$ and $10^{-4}$ respectively. The simulation is performed to a final time of $T=1.8$, over the domain $[-5,5]$, with Dirichlet boundary conditions, enforcing that the solution remain constant at the boundaries over time. In Figure \eqref{fig:ShuOsh}, the density of the solutions at the final time computed using the LxF and HLLC versions of QK and the LxF version of ESFD for $N = 3$ and $M = 100$ are shown. Each solution is compared with a reference solution computed using a 5th order WENO method with 25000 cells. We observe minor oscillations for the QK solution, which are further dampened by the HLLC version of QK, but major oscillations are observed in the ESFD solution. Moreover, the peaks of the right-hand-side of the HLLC QK solution are very close to the target WENO solution. Since the HLLC version of ESFD crashes prior to the final time $T=1.8$, it has been omitted.

\begin{figure}%
    \centering
    \subfloat[\centering Quadratic knapsack limiting]{{\includegraphics[width=4.2cm]{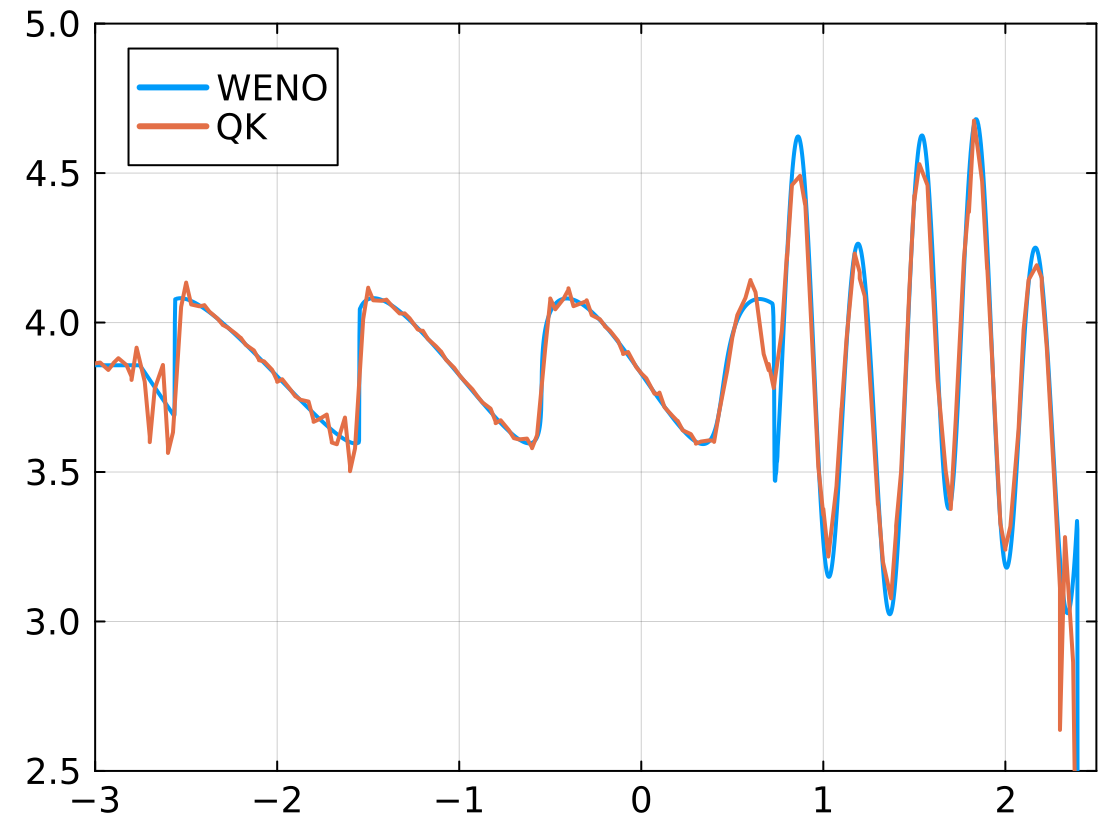} }}%
    \subfloat[\centering HLLC quadratic knapsack limiting]{{\includegraphics[width=4.2cm]{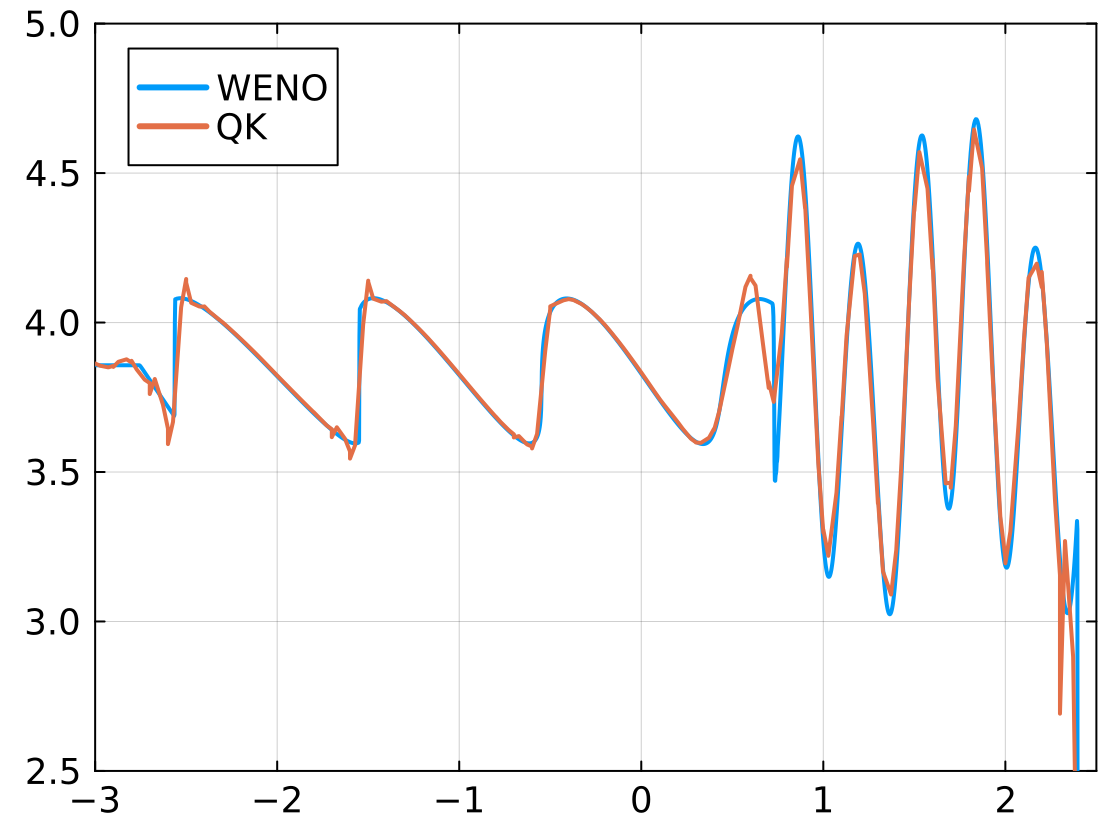} }}%
    \subfloat[\centering ESFD]{{\includegraphics[width=4.2cm]{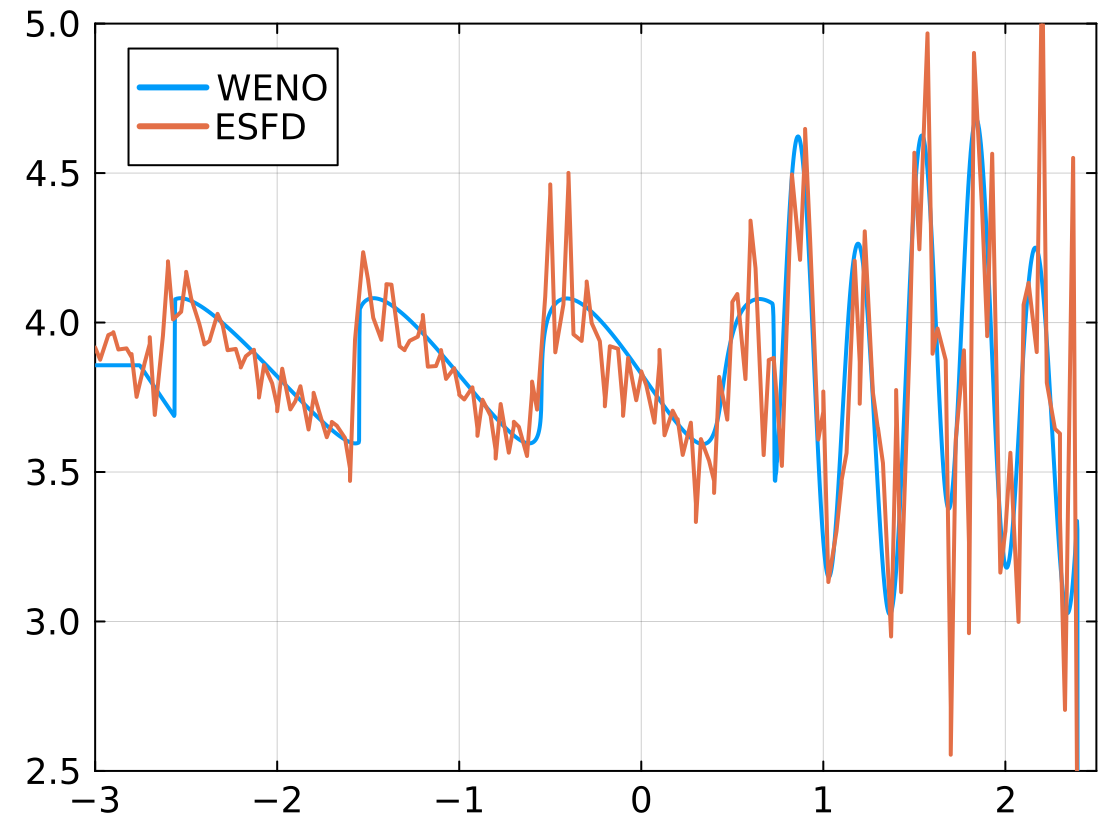} }\label{fig:ShuOsherLinearInstability}}%
    \caption{Density of the Shu-Osher shock tube solution at final time $T=1.8$.}%
    \label{fig:ShuOsh}%
\end{figure}

\subsection{Time Convergence}

In this section, we examine the behavior of quadratic knapsack limiting with adaptive timestepping in shock-type problems, compared against that of linear knapsack limiting and flux differencing with an entropy conservative volume flux. We will observe that quadratic knapsack limiting tends to require far fewer adaptive timesteps than linear knapsack limiting in high order simulations of shock-type problems. We will also conclude that in shock-type problems, linear knapsack limiting attains a $\mathcal{O}(\Delta t)$ order of convergence in time, compared against $\mathcal{O}(\Delta t^2)$ for quadratic knapsack limiting and flux differencing with an entropy conservative volume flux.

\subsubsection{Adaptive Timestepping}

Adaptive timesteppers will often apply two different estimates for the solution at the next timestep. Then, it will compare the absolute and relative errors between the two estimate solutions. If either exceeds a chosen tolerance, $\Delta t$ is made smaller for the next timestep, or even for the current timestep if deemed necessary. Otherwise, $\Delta t$ is made larger \cite{OrdinaryDiffEq}.

The solution to the linear knapsack problem is discontinuous in terms of its input variables, while the solution to the quadratic knapsack problem is continuous. Since the solutions to the knapsack problems are incorporated into their corresponding volume terms via \eqref{eq:Blending}, we should expect that the regularity of the knapsack solver may affect the regularity of the scheme. Low regularity (in time) systems drive adaptive timesteppers to resolve the discontinuities. This results in the simulation requiring more adaptive timesteps. We will observe that this results in LK requiring more adaptive timesteps to complete shock-type simulations than QK. 

First, we analyze the number of adaptive timesteps taken to complete the modified Sod shock tube problem. In Tables \eqref{fig:RK4ModSodAdaptive} and \eqref{fig:SSPRK43ModSodAdaptive}, we examine the total number of adaptive timesteps taken by the adaptive timesteppers \verb|RK4| and \texttt{SSPRK43} to complete the modified Sod shock tube problem, for LK, QK, and ESFD, using various orders, mesh sizes, and tolerances. \comment{To ensure that patterns remain consistent across various timesteppers, we do the same for the adaptive timestepper \verb|Tsit5| in Table \eqref{fig:Tsit5ModSodAdaptive}. For each table, }DNF means the solution took more than one million timesteps to converge. We observe, especially for high order $N = 7$, LK takes many more adaptive timesteps to complete the modified Sod shock tube simulation than QK and ESFD. We observe similar behavior for various shock type problems and adaptive timesteppers. To justify that this adaptive timestep discrepancy saves time, we computed the average time it takes to perform a knapsack solve, as well as the total time spent solving knapsack problems during each simulation. These times are shown in Tables \eqref{tab:TimeTableRK4} and \eqref{tab:TimeTableSSPRK43}. The gray cells correspond to DNF. Though we observe that the average quadratic knapsack problem call is longer than the linear knapsack calls, we observe that the adaptive timestep discrepancy more than makes up the difference. That is, the total time spent solving knapsack problems is longer for linear knapsack than quadratic knapsack. 

\begin{table}[h!]
\centering
\begin{tabular}{{|c|c|c|c|c|}}
\hline
$\text{abstol, reltol}$ & $N,M$ & $\text{Linear Knapsack}$ & $\text{Quadratic Knapsack}$ & $\text{ESFD}$ \\
\hline
\multirow{2}{*}{$10^{-6},10^{-4}$} & $3,\ 64$ & $2757$ & $1635$ & $1078$ \\
\cline{2-5}
 & $7,\ 32$ & $44540$ & $2459$ & $1749$ \\
\hline
\multirow{2}{*}{$10^{-8},10^{-6}$} & $3,\ 64$ & $8372$ & $4955$ & $3393$ \\
\cline{2-5}
 & $7,\ 32$ & \cellcolor[HTML]{A8A8A8}$\text{DNF}$ & $7450$ & $5572$ \\
\hline
\multirow{2}{*}{$10^{-10},10^{-8}$} & $3,\ 64$ & $18896$ & $13260$ & $10232$ \\
\cline{2-5}
 & $7,\ 32$ & \cellcolor[HTML]{A8A8A8}$\text{DNF}$ & $36013$ & $83191$ \\
\hline
\end{tabular}
\cprotect\caption{Total timesteps for \verb|RK4| in the modified Sod shock tube simulations.}
\label{fig:RK4ModSodAdaptive}
\end{table}

\begin{table}[h!]
\centering
\begin{tabular}{{|c|c|c|c|c|}}
\hline
$\text{abstol, reltol}$ & $N,M$ & $\text{Linear Knapsack}$ & $\text{Quadratic Knapsack}$ & $\text{ESFD}$ \\
\hline
\multirow{2}{*}{$10^{-6},10^{-4}$} & $3,\ 64$ & $1654$ & $1022$ & $909$ \\
\cline{2-5}
 & $7,\ 32$ & $6920$ & $1478$ & $1864$ \\
\hline
\multirow{2}{*}{$10^{-8},10^{-6}$} & $3,\ 64$ & $6383$ & $4656$ & $4441$ \\
\cline{2-5}
 & $7,\ 32$ & $456549$ & $7193$ & $8147$ \\
\hline
\multirow{2}{*}{$10^{-10},10^{-8}$} & $3,\ 64$ & $24014$ & $20351$ & $20631$ \\
\cline{2-5}
 & $7,\ 32$ & \cellcolor[HTML]{A8A8A8}$\text{DNF}$ & $31275$ & $37888$ \\
\hline
\end{tabular}
\caption{Total timesteps for \texttt{SSPRK43} in the modified Sod shock tube simulations.}
\label{fig:SSPRK43ModSodAdaptive}
\end{table}

\begin{table}[h!]
\centering
\resizebox{\textwidth}{!}{
\begin{tabular}{{|c|c|c|c|c|c|}}
\hline
$$ & $$ & \multicolumn{2}{|c|}{$\text{Linear Knapsack}$} & \multicolumn{2}{|c|}{$\text{Quadratic Knapsack}$} \\
\hline
$\text{abstol, reltol}$ & $N,M$ & \multicolumn{1}{|c}{$\text{Average time (ns)}$} & $\text{Total time (ms)}$ & \multicolumn{1}{|c}{$\text{Average time (ns)}$} & $\text{Total time (ms)}$ \\
\hline
\multirow{2}{*}{$10^{-6},10^{-4}$} & $3,\ 64$ & \multicolumn{1}{|c}{$42$} & $26.2$ & \multicolumn{1}{|c}{$58$} & $21.8$ \\
\cline{2-6}
 & $7,\ 32$ & \multicolumn{1}{|c}{$74$} & $382$ & \multicolumn{1}{|c}{$83$} & $23.3$ \\
\hline
\multirow{2}{*}{$10^{-8},10^{-6}$} & $3,\ 64$ & \multicolumn{1}{|c}{$42$} & $84.2$ & \multicolumn{1}{|c}{$59$} & $60.8$ \\
\cline{2-6}
 & $7,\ 32$ & \multicolumn{1}{|c}{\cellcolor[HTML]{A8A8A8}$74$} & \cellcolor[HTML]{A8A8A8}$8210$ & \multicolumn{1}{|c}{$84$} & $72.1$ \\
\hline
\multirow{2}{*}{$10^{-10},10^{-8}$} & $3,\ 64$ & \multicolumn{1}{|c}{$42$} & $201$ & \multicolumn{1}{|c}{$60$} & $199$ \\
\cline{2-6}
 & $7,\ 32$ & \multicolumn{1}{|c}{\cellcolor[HTML]{A8A8A8}$77$} & \cellcolor[HTML]{A8A8A8}$6180$ & \multicolumn{1}{|c}{$84$} & $348$ \\
\hline
\end{tabular}}
\caption{Average time of knapsack calls and contribution to simulation time in \texttt{RK4} simulations.}
\label{tab:TimeTableRK4}
\end{table}

\begin{table}[h!]
\centering
\resizebox{\textwidth}{!}{
\begin{tabular}{{|c|c|c|c|c|c|}}
\hline
$$ & $$ & \multicolumn{2}{|c|}{$\text{Linear Knapsack}$} & \multicolumn{2}{|c|}{$\text{Quadratic Knapsack}$} \\
\hline
$\text{abstol, reltol}$ & $N,M$ & \multicolumn{1}{|c}{$\text{Average time (ns)}$} & $\text{Total time (ms)}$ & \multicolumn{1}{|c}{$\text{Average time (ns)}$} & $\text{Total time (ms)}$ \\
\hline
\multirow{2}{*}{$10^{-6},10^{-4}$} & $3,\ 64$ & \multicolumn{1}{|c}{$42$} & $10.6$ & \multicolumn{1}{|c}{$58$} & $8.98$ \\
\cline{2-6}
 & $7,\ 32$ & \multicolumn{1}{|c}{$75$} & $39.1$ & \multicolumn{1}{|c}{$83$} & $9.17$ \\
\hline
\multirow{2}{*}{$10^{-8},10^{-6}$} & $3,\ 64$ & \multicolumn{1}{|c}{$41$} & $47.4$ & \multicolumn{1}{|c}{$58$} & $39.8$ \\
\cline{2-6}
 & $7,\ 32$ & \multicolumn{1}{|c}{$74$} & $2370$ & \multicolumn{1}{|c}{$83$} & $44.9$ \\
\hline
\multirow{2}{*}{$10^{-10},10^{-8}$} & $3,\ 64$ & \multicolumn{1}{|c}{$43$} & $182$ & \multicolumn{1}{|c}{$60$} & $195$ \\
\cline{2-6}
 & $7,\ 32$ & \multicolumn{1}{|c}{\cellcolor[HTML]{A8A8A8}$71$} & \cellcolor[HTML]{A8A8A8}$7350$ & \multicolumn{1}{|c}{$83$} & $188$ \\
\hline
\end{tabular}}
\caption{Average time of knapsack calls and contribution to simulation time in \texttt{SSPRK43} simulations.}
\label{tab:TimeTableSSPRK43}
\end{table}

\subsubsection{Order of Convergence in Time}

As we observed when experimenting with adaptive timesteppers, the linear knapsack problem makes LK less regular in time than QK in shock-type problems. In fact, we observe that it also reduces the order of convergence in time in shock-type problems.

For convenience of notation, we define $LK(\Delta t)$, $QK(\Delta t)$, and $ESFD(\Delta t)$ to be the solution of a given problem at the final time when using a given timestepper with fixed timestep $\Delta t$ when using LK, QK, and ESFD respectively. One way to measure time convergence numerically is to examine $\norm{S(\Delta t_\text{ref}) - S(\Delta t)}$ where the scheme $S\in\{LK,QK,ESFD\}$, and $\Delta t_\text{ref} \ll  \Delta t$ is some reference timestep. In this way, $\norm{S(\Delta t_\text{ref}) - S(\Delta t)} = \mathcal{O}(\Delta t^K)$. 

In Figure \eqref{fig:ModSodFixed}, we plot the $L^2$ error $\norm{S(\Delta t_\text{ref}) - S(\Delta t)}$ for the modified Sod shock tube problem, for each $S\in\{LK,QK,ESFD\}$ against various timesteps $\Delta t \geq 5\Delta t_\text{ref}$, where $\Delta t_\text{ref} = 10^{-6}$. If a scheme crashes prior to the final time for some $\Delta t$, it is omitted from the plot. The chosen timestepper was \texttt{SSPRK43}, and plots are shown for each $(N, M)=(3,64)$ and $(N, M)=(7, 32)$. Note however that we observe similar results for various timesteppers of order at least $\mathcal{O}(\Delta t^2)$. For $(N, M) = (3, 64)$, we observe that both QK and ESFD obtain the same rates of convergence in time, while a lower rate of convergence in time is observed for LK. For $(N,M)=(7,32)$, we observe the same difference between LK and QK, but we observe a faster rate of convergence for ESFD. The $L^2$ errors measured were within the same level as those for QK, for the range of timesteps tested. It is unclear why we observe second order convergence in time for $N=3$ but a higher convergence rate in time for $N=7$. 

\begin{figure}%
    \centering
    \subfloat[\centering $N = 3, M = 64$]{{\includegraphics[width=5.5cm]{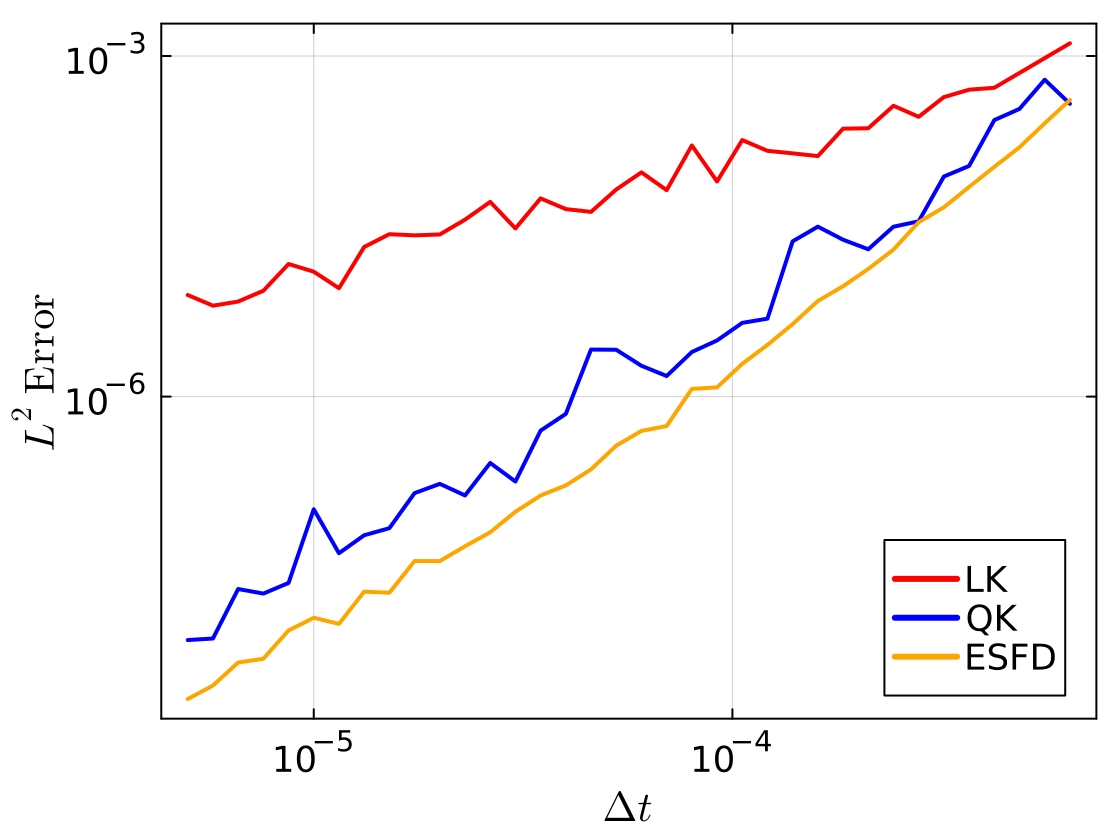} }}%
    \qquad
    \subfloat[\centering $N = 7, M = 32$]{{\includegraphics[width=5.5cm]{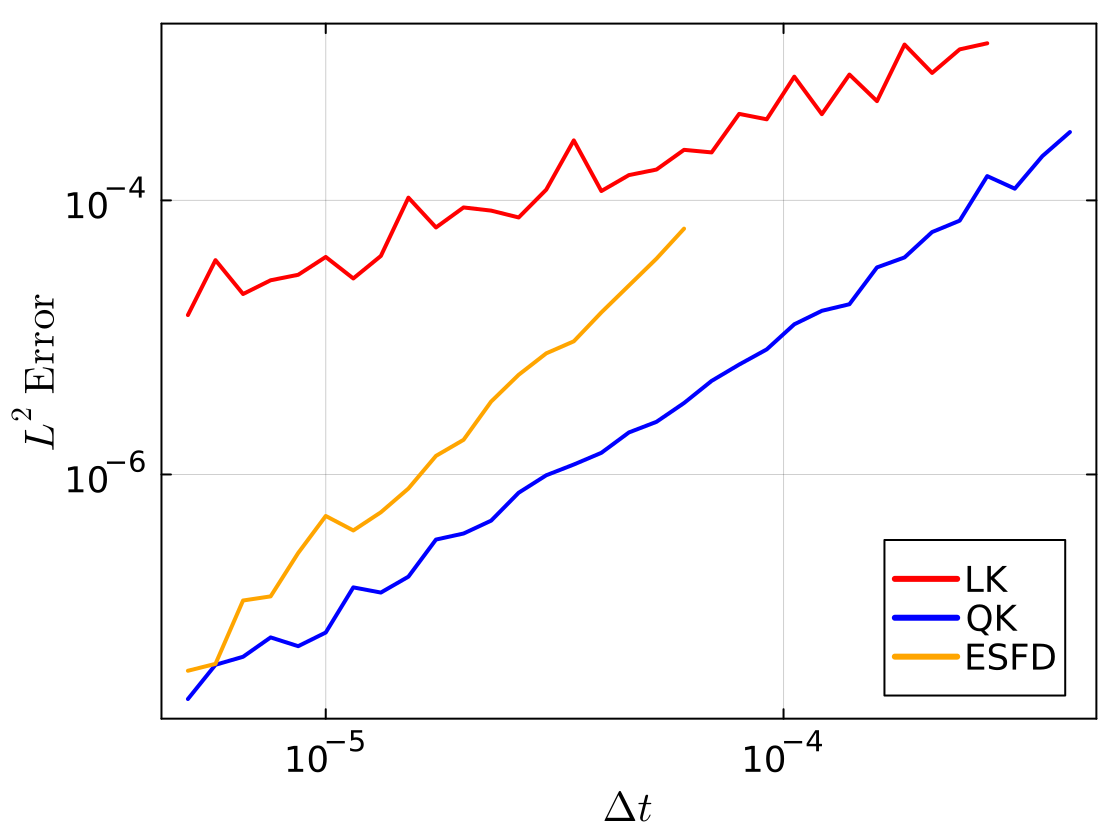} }
    \label{fig:RK4TimeLinearInstability}}%
    \caption{$L^2$ error of the modified Sod shock tube simulations against reference solutions at
    final time $T = .2$ against fixed timestep.}%
    \label{fig:ModSodFixed}%
\end{figure}

% In order to reveal the average behavior of the time convergence of LK, QK, and ESFD in shock-type problems, we apply the same study for various timesteppers, shock-type problems, and mesh parameters $(N, M)$ (for Burgers' equation, $\Delta t_\text{ref} = 5\cdot 10^{-6}$ and we consider timesteps $\Delta t$ such that $\Delta t \geq 2\Delta t_\text{ref}$). Then, we measure the slope of the corresponding log-log relationship between $\Delta t$ and $\norm{S(\Delta t_\text{ref}) - S(\Delta t)}$, and place the slopes in Table \eqref{fig:TimeConvergenceTable}. We observe from these results that, for shock-type problems, both QK and ESFD obtain $\mathcal{O}(\Delta t^2)$ time convergence on average, while LK obtains $\mathcal{O}(\Delta t)$ time convergence. 

In order to reveal the average behavior of the time convergence of LK, QK, and ESFD, we apply the same study for various timesteppers and mesh parameters $(N, M)$. Then, we measure the slope of the corresponding log-log relationship between $\Delta t$ and $\norm{S(\Delta t_\text{ref}) - S(\Delta t)}$, and place the slopes in Table \eqref{fig:TimeConvergenceTable}. We observe from these results that both QK and ESFD obtain roughly $\mathcal{O}(\Delta t^2)$ time convergence on average, while LK obtains $\mathcal{O}(\Delta t)$ time convergence. We observe these results for various shock-type problems. 

% \begin{table}[!htp]\centering
% \caption{Estimated orders of convergence in time}
% \resizebox{\textwidth}{!}{
% \begin{tabular}{lrrrrrr}\toprule
% & &$N$, $M$ &Linear Knapsack &Quadratic Knapsack &ESFD \\\midrule
% \multirow{4}{*}{\texttt{RK4}} &\multirow{2}{*}{Burgers} &3, 64 &0.992165 &1.92386 &1.60823 \\\cmidrule{3-6}
% & &7, 32 &1.02833 &1.93396 &1.93245 \\\cmidrule{2-6}
% &\multirow{2}{*}{\parbox{2.4cm}{Modified Sod \newline Shock Tube}} &3, 64 &0.989224 &2.11371 &2.10993 \\\cmidrule{3-6}
% & &7, 32 &1.01271 &1.99411 &2.6281 \\\cmidrule{1-6}
% \multirow{4}{*}{\texttt{SSPRK43}} &\multirow{2}{*}{Burgers} &3, 64 &1.04671 &1.95564 &2.11135 \\\cmidrule{3-6}
% & &7, 32 &0.980398 &1.98372 &2.64165 \\\cmidrule{2-6}
% &\multirow{2}{*}{\parbox{2.4cm}{Modified Sod Shock Tube}} &3, 64 &0.997749 &2.26205 &2.46284 \\\cmidrule{3-6}
% & &7, 32 &1.09286 &2.09931 &2.8862 \\\midrule
% &Average & &1.01751825 &2.033295 &2.29759375 \\
% \bottomrule
% \end{tabular}}
% \label{fig:TimeConvergenceTable}
% \end{table}

\begin{table}[h!]
\centering
\begin{tabular}{{|c|c|c|c|c|}}
\hline
$\text{Timestepper}$ & $N,M$ & $\text{Linear Knapsack}$ & $\text{Quadratic Knapsack}$ & $\text{ESFD}$ \\
\hline
\multirow{2}{*}{$\texttt{RK4}$} & $3,\ 64$ & $0.99$ & $2.11$ & $2.11$ \\
\cline{2-5}
 & $7,\ 32$ & $1.01$ & $1.99$ & $2.63$ \\
\hline
\multirow{2}{*}{$\texttt{SSPRK43}$} & $3,\ 64$ & $1.00$ & $2.26$ & $2.46$ \\
\cline{2-5}
 & $7,\ 32$ & $1.09$ & $2.10$ & $2.89$ \\
\hline
$\text{Average}$ &  & $1.02$ & $2.12$ & $2.52$ \\
\hline
\end{tabular}
\caption{Estimated orders of convergence in time.}
\label{fig:TimeConvergenceTable}
\end{table}

\subsection{Local Linear Stability}\label{sec:LinearStability}

In this section, we examine the local linear stability of the knapsack limiting schemes compared against the discontinuous Galerkin spectral element method and flux differencing with an entropy conservative volume flux. Since the knapsack limiting schemes are a linear combination of the high and low order schemes \eqref{eq:HighOrderScheme} and \eqref{eq:LowOrderScheme}\footnote{The low order scheme \eqref{eq:LowOrderScheme} is a central scheme with an added viscosity in terms of the conservative variables, and is thus locally linearly stable.}, which are each locally linearly stable, we expect them to also be locally linearly stable. We will conclude that the knapsack limiting schemes are locally linearly stable, while flux differencing with an entropy conservative volume flux is not. We will also examine the effect of linear instability on the errors of the simulation as time progresses, as well as the eigenvalues of the linearized Jacobian of each scheme. 
\comment{Consider a linear system of ODEs:
\begin{align*}\frac{d\vecf{\vecf{y}}}{d t}=\fnt{A}\vecf{y}\end{align*}
where $\vecf{y}:\mathbb{R}_{\ge0}\to\mathbb{R}^{v}$, and $\fnt{A}\in\mathbb{R}^{v\times v}$. For illustrative purposes, we will assume that $\fnt{A}$ is diagonalizable. Hence, $\fnt{A}=\fnt{V}\diag\left(\lambda_{1},...,\lambda_{v}\right)\fnt{V}^{-1}$, where $\fnt{V}\in\mathbb{R}^{v\times v}$ is invertible, and $\lambda_{1},...,\lambda_{v}$ are the eigenvalues of $\fnt{A}$. It is known that the explicit solution is:
\begin{align*}\vecf{y}\left(\vecf{x},t\right)=\fnt{V}\diag\left(e^{\lambda_{1}t},...,e^{\lambda_{v}t}\right)\fnt{V}^{-1}\vecf{y}\left(\vecf{x},0\right)\end{align*}
If there is some $\lambda\in\sigma\left(\fnt{A}\right)$ for which $\Re\left(\lambda\right)>0$, then $e^{\lambda t}=e^{\Re\left(\lambda\right)t}\left(\cos\left(\Im\left(\lambda\right)t\right)+i\sin\left(\Im\left(\lambda\right)t\right)\right)$ is unbounded. Thus, the solution may grow indefinitely, exponentiating the modes corresponding to eigenvalues with positive real parts. We call the system above \textbf{linearly stable} if $\Re\left(\sigma\left(\fnt{A}\right)\right)\le0$. This form of stability ensures that exponentiation of oscillations does not occur. However, this analysis does not extend to nonlinear PDEs. }Consider the ODE system:
\begin{align*}\frac{\partial\vecf{u}}{\partial t}=\fnt{rhs}\left(\vecf{u},t\right),\end{align*}
where $\fnt{rhs}:\mathbb{R}^{v}\times\mathbb{R}_{\ge0}\to\mathbb{R}^{v}$ denotes the right-hand-side of the PDE. $\fnt{rhs}$ is a stand-in for the evaluation of $\frac{\d\fnt{u}}{\d t}$. For instance, $\fnt{rhs}$ may represent the DGSEM, ESFD, LK, or QK evaluations of $\frac{\d\vect{u}}{\d t}$. The ODE system above is called \textit{locally linearly stable} if its linearization is linearly stable. That is, $\Re\left(\sigma\left(\nabla_{\vecf{u}}\fnt{rhs}\left(\vecf{u},t\right)\right)\right)\le0$, where $\nabla_{\vecf{u}}\fnt{rhs}\left(\vecf{u},t\right)$ denotes the Jacobian of $\fnt{rhs}\left(\vecf{u},t\right)$. The maximum real part of the spectrum of the linearized Jacobian describes how sensitive a semi discretization of a PDE (such as LK, QK, ESFD, or DGSEM) is to small perturbations around a smooth, well-resolved solution. 
%We should expect that any viable simulation of real-world physics is locally linearly stable. 
However, it was shown in \cite{Gassner21} that, in contrast to a standard nodal DG method, the use of flux differencing with an entropy conservative volume flux (for example, ESFD) results in a linearly unstable scheme. The result is a simulation which amplifies nonphysical oscillations over time, as seen in Figures \eqref{fig:ModSodLinearInstability} and \eqref{fig:ShuOsherLinearInstability}. We see in the same plots that the knapsack limiting solutions do not seem to undergo the same level of development of nonphysical oscillations. 

\subsubsection{Error Growth over Time}

A consequence of linear instability is an increase in error as time progresses. To examine this, we consider the Sod shock tube initial condition for the 1D compressible Euler equations:
\begin{align*}p\left(x,0\right)=\begin{cases}1&x<.5\\
    .1&\text{otherwise}\end{cases},&&
    v\left(x,0\right)=0,&&
    \rho\left(x,0\right)=\begin{cases}1&x<.5\\
    .125&\text{otherwise}\end{cases}.\end{align*}

The Sod shock tube solution is computed using the \texttt{SSPRK43} timestepper with a fixed $\Delta t=10^{-5}$. The simulation is performed to a final time of $T=.2$, over the domain $[0,1]$, with Dirichlet boundary conditions, enforcing that the solution remain constant at the boundaries over time. The Sod shock tube problem is also often used to test entropy stable schemes, because it requires some form of stabilization to converge to the entropy solution. The exact solution to the Sod shock tube problem is computed by the Julia library \verb|SodShockTube.jl|. 

In Figure \eqref{fig:SodShockErrorsOverTime}, we plot the $L^2$ error of the Sod shock tube simulation against the exact solution over time for each scheme LK, QK, ESFD, and DGSEM. In order to make the data easier to visualize, the errors are postprocessed using a rolling mean, with a window width of $.01024$. If the simulation crashed early, its error is omitted. We observe that the knapsack limiters obtain smaller errors when compared against ESFD. We attribute this to the linear instability of ESFD. 

% It can also be seen in each simulation that DGSEM crashes early. The reason for this can be seen by measuing the total entropy of the simulation over time. In Figure \eqref{fig:SodShockEntropyOverTime}, we show the measured total entropy of the Sod shock tube simulation over time. We observe that the total entropy of the DGSEM simulation diverges as the simulation crashes. This is due to the lack of enforcement of an entropy inequality. The total entropy for each knapsack limiter is roughly equal, and decreases slightly faster than that of ESFD. 

\begin{figure}%
    \centering
    \subfloat[\centering $N = 3, M = 64$]{{\includegraphics[width=6cm]{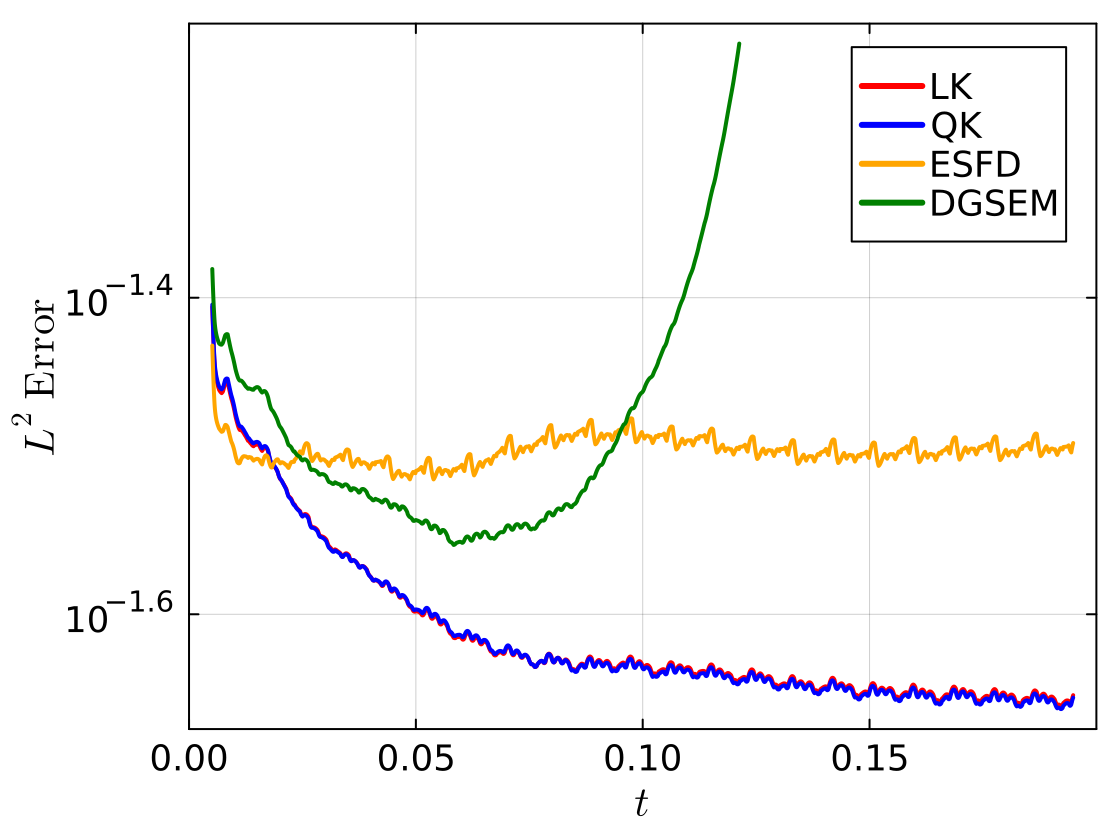} }}%
    \qquad
    \subfloat[\centering $N = 7, M = 32$]{{\includegraphics[width=6cm]{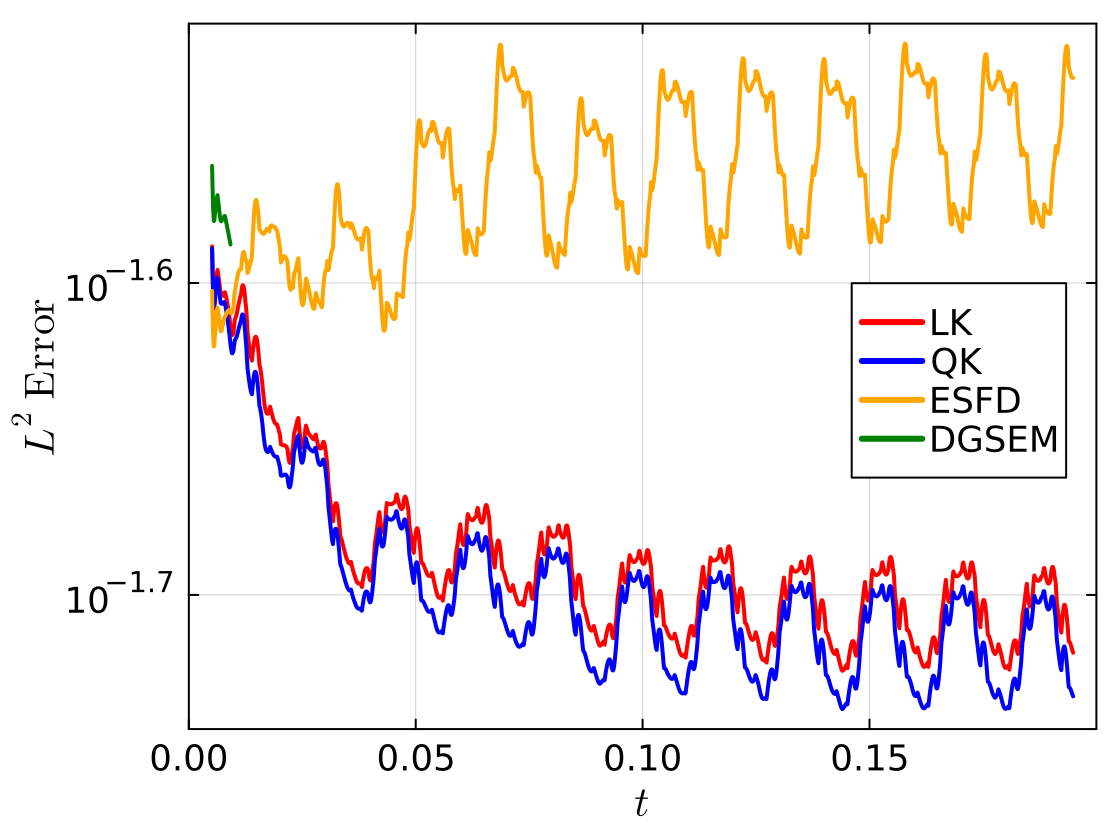} }}%
    \caption{$L^2$ error of the Sod shock tube simulations against exact solution over time.}%
    \label{fig:SodShockErrorsOverTime}%
\end{figure}

\subsubsection{Eigenvalues of Linearized Jacobian}

Last, to measure the local linear stability of each scheme, we measure the maximal real parts of the eigenvalues of the linearized Jacobian for each scheme. For convenience of notation, we denote $\fnt{LK}$, $\fnt{QK}$, $\fnt{ESFD}$,and $\fnt{DGSEM}$ as the evaluations of $\frac{\d\vect{u}}{\d t}$ from LK, QK, ESFD, and DGSEM respectively. In Table \eqref{fig:Eigenvalues}, we show the maximal real parts of the eigenvalues of the linearized Jacobians of $\fnt{rhs}\in \{\fnt{LK}, \fnt{QK}, \fnt{ESFD}, \fnt{DGSEM}\}$, at the initial timestep of the modified Sod shock tube, for various mesh parameters $(N, M)$. The linearized Jacobians were estimated using the \verb|ForwardDiff.jl| Julia package \cite{ForwardDiff}. We observe that the maximal eigenvalues for LK, QK, and DGSEM are all of the order $10^{-13}$, while those for ESFD are of the order $10^2$ or $10^3$. Though the maximal eigenvalues for LK, QK, and DGSEM are positive, they are also close to machine epsilon, which signifies that these schemes are locally linearly stable up to machine epsilon. \comment{The eigenvalue spectra for $(N, M) = (3, 64)$ for each of the linearized Jacobians of $\fnt{LK}$, $\fnt{QK}$, $\fnt{ESFD}$ and $\fnt{DGSEM}$ are shown in Figure \eqref{fig:EigenvalueSpectra}.}

\begin{table}[h!]
\centering
\begin{tabular}{{|c|c|c|c|c|}}
\hline
$N,M$ & $\text{Linear Knapsack}$ & $\text{Quadratic Knapsack}$ & $\text{ESFD}$ & $\text{DGSEM}$ \\
\hline
$3,\ 64$ & $4.41\cdot10^{-13}$ & $4.64\cdot10^{-13}$ & $48.0$ & $4.41\cdot10^{-13}$ \\
\hline
$7,\ 32$ & $2.96\cdot10^{-13}$ & $5.57\cdot10^{-13}$ & $116$ & $2.96\cdot10^{-13}$ \\
\hline
$15,\ 16$ & $5.24\cdot10^{-13}$ & $1.46\cdot10^{-13}$ & $176$ & $5.24\cdot10^{-13}$ \\
\hline
\end{tabular}
\caption{$\max \Re \left(\sigma\left(\nabla_{\fnt{u}} \fnt{rhs}(\fnt{u})\right)\right)$.}
\label{fig:Eigenvalues}
\end{table}

% \begin{figure}%
%     \centering
%     \subfloat[\censpectratering Linear Knapsack Limiting]{{\includegraphics[width=6cm]{L1_eigvals.png} }}%
%     \qquad
%     \subfloat[\centering Quadratic Knapsack Limiting]{{\includegraphics[width=6cm]{L2_eigvals.png} }}%
%     \qquad
%     \subfloat[\centering ESFD]{{\includegraphics[width=6cm]{ES_eigvals.png} }}%
%     \qquad
%     \subfloat[\centering DGSEM]{{\includegraphics[width=6cm]{DG_eigvals.png} }}%
%     \caption{Eigenvalue spectra of the linearized Jacobians of each scheme at the initial timestep of the modified Sod shock tube with $(N, M) = (3, 64)$}%
%     \label{fig:EigenvalueSpectra}%
% \end{figure}

\subsection{Positivity Preservation}

In the numerical simulations performed thus far, preservation of positivity was not required to attain a solution. However, for many benchmarks, such as the Leblanc shock tube, Sedov blast wave, and long-time Kelvin-Helmholtz instability, some relative positivity constraint, such as \eqref{eq:WeakPositivity} must be enforced to ensure that physical constraints, such as pressure and density, do not dip below 0. In this section, we benchmark quadratic knapsack limiting in various difficult simulations, often used to test positivity preserving methods. The strategy taken for preserving positivity is the one derived in Section \eqref{sec:Positivity Preservation}. For the more challenging positivity preservation benchmarks, Leblanc shock tube and Sedov blast wave, we apply a fixed timestep. This is because adaptive timesteppers do not take positivity preservation into account, and therefore require a smaller tolerance. This results in a much longer computation time than is necessary.

\subsubsection{Leblanc Shock Tube} \label{sec:Leblanc}

First, we test quadratic knapsack limiting on the 1D compressible Euler, Leblanc shock tube problem. This problem mimics pressure waves in near vacuum conditions, common in hypersonic flow. The initial condition for this problem is as follows:
\begin{align*}p\left(x,0\right)=\begin{cases}10^9&x<0\\
    1&\text{otherwise}\end{cases},&&
    v\left(x,0\right)=0,&&
    \rho\left(x,0\right)=\begin{cases}2&x<0\\
    10^{-3}&\text{otherwise}\end{cases}.\end{align*}
The initial condition has a pressure ratio of $10^9$ and a density ratio of $2000$. Due to oscillations formed around shocks in high order simulations which are proportional to the jump, preservation of positivity is necessary to produce a stable solution. The Leblanc shock tube solution is computed using the strong stability preserving \verb|SSPRK43| timestepper with a fixed $\Delta t=7.1\cdot10^{-8}$. This timestep is within $ 10^{-9}$ of the smallest fixed timestep where the solution crashes due to nonpositivity. The simulation is performed to a final time of $T = 10^{-4}$, over the domain $[-10,10]$, with Dirichlet boundary conditions, enforcing that the solution remain constant at the boundaries over time. 

In Figure \eqref{fig:LeblancShocktube}, we plot the density, velocity, and pressure of the solution generated by the HLLC version of quadratic knapsack limiting with $(N, M)=(3, 1000)$, against a reference solution. The HLLC version is used to remain consistent with \cite{Peyvan23}. The relative positivity constraint with $\alpha = .55$ is enforced, and plots are shown for both nodewise limiting coefficients \eqref{sec:Positivity Preservation}, and elementwise limiting coefficients, where the limiting coefficients are chosen as the elementwise maximum of the limiting coefficients derived in Section \eqref{sec:Positivity Preservation}. We observe that the solution is close to the reference solution, though the solution with elementwise limiting coefficients does a better job at limiting oscillations around the peaks of shocks. 

\begin{figure}%
    \centering
    \subfloat[\centering Density of solution with nodewise limiting coefficients]{{\includegraphics[width=4.2cm]{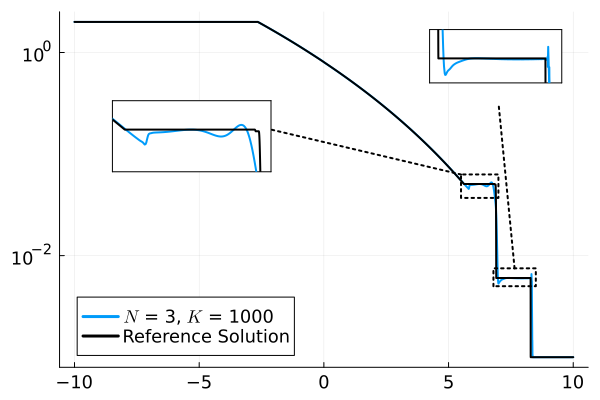} }}%
    \subfloat[\centering Velocity of solution with nodewise limiting coefficients]{{\includegraphics[width=4.2cm]{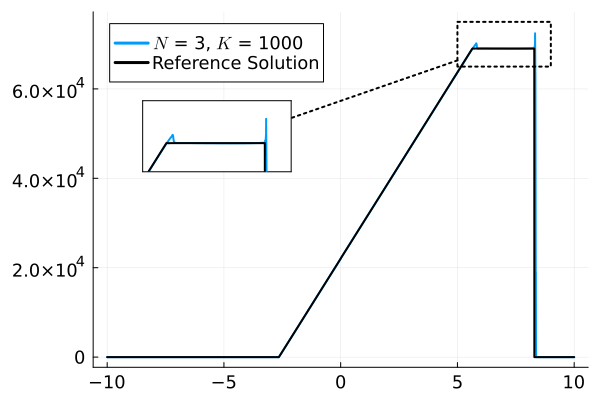} }}%
    \subfloat[\centering Pressure of solution with nodewise limiting coefficients]{{\includegraphics[width=4.2cm]{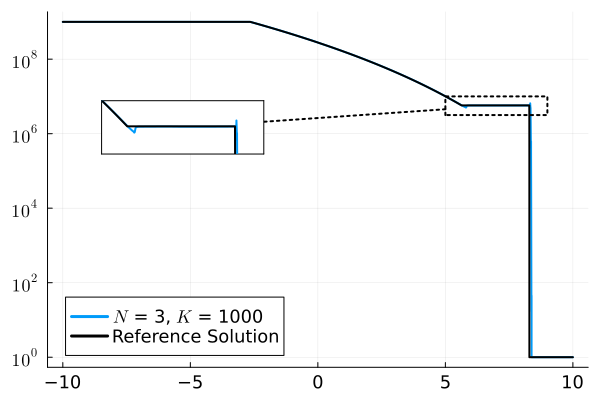} }}%
    \qquad
    \subfloat[\centering Density of solution with elementwise limiting coefficients]{{\includegraphics[width=4.2cm]{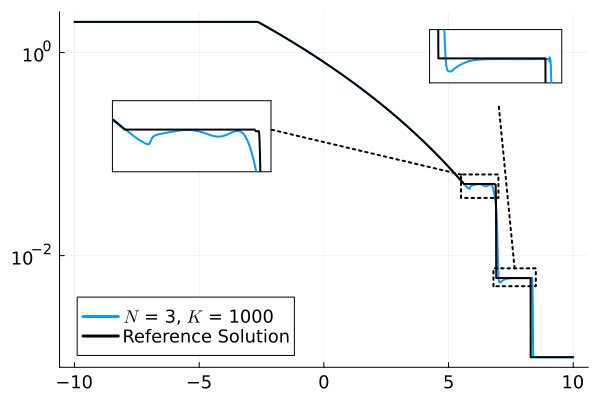} }}%
    \subfloat[\centering Velocity of solution with elementwise limiting coefficients]{{\includegraphics[width=4.2cm]{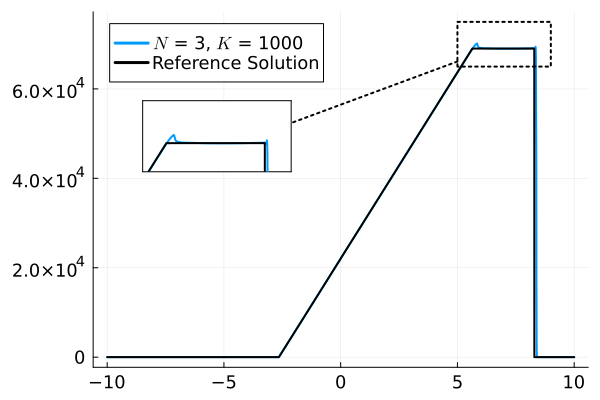} }}%
    \subfloat[\centering Pressure of solution with elementwise limiting coefficients]{{\includegraphics[width=4.2cm]{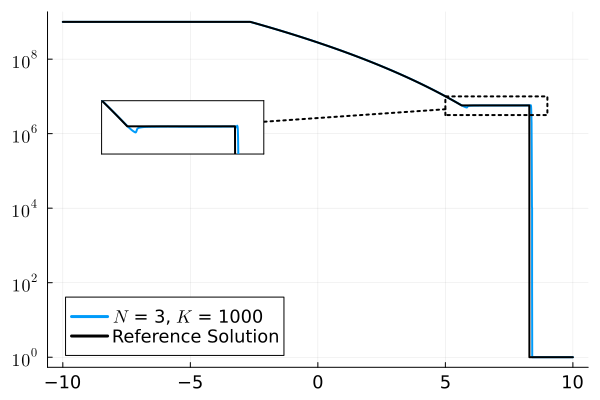} }}%
    \caption{HLLC quadratic knapsack limiting solution to the Leblanc shock tube with $N = 3, M = 1000$ and relative positivity constraint $\alpha = .55$.}%
    \label{fig:LeblancShocktube}%
\end{figure}

While we have used an experimentally determined fixed timestep for all positivity preserving experiments, positivity preservation is only guaranteed under a stricter CFL condition \cite{Lin23}. To show that our choice of fixed timestep size is not overly small relative to the conservative positivity preserving timestep, we compute the effective CFL condition, which we define as the ratio of our fixed timestep size with the maximum positivity preserving timestep size derived in \cite{Lin23}. \comment{To measure the effective CFL condition corresponding to the chosen timestep $\Delta t = 7.1\cdot 10^{-8}$, we define the following:
\begin{align}
    \text{CFL} &= \frac{2 \Delta t}{\min_i \frac{\Delta x \cdot \omega_i}{2}\cdot \frac{1}{\lambda_i}}=4 \frac{\Delta t}{\Delta x} \max_i \frac{\lambda_i}{\omega_i}\\
    \lambda_i &= \norm{\fnt{v}_i} + \sqrt{\frac{\gamma p_i}{\rho_i}}\\
    \fnt{v}_i &= \begin{pmatrix}
        v_{i,1}\\
        \dots\\
        v_{i,d}
    \end{pmatrix}
\end{align}
where $\rho_i$, $v_{i,1}\dots v_{i,d}$, and $p_i$ correspond to the density, velocity components, and pressure of the solution at the $i$th spatial node. Moreover, $\lambda$ is an estimate of the maximum wave speed, and $\omega_i$ corresponds to the LGL quadrature weight corresponding to node $i$. In Figure \eqref{fig:LeblancCFL}, we show the effective CFL condition corresponding to the timestep $\Delta t = 7.1 \cdot 10^{-8}$ over time. We observe that the effective CFL condition is of order $\mathcal{O}(1)$ throughout the simulation. 

\begin{figure}%
    \centering
    \subfloat[\centering Nodewise Limiting Coefficients]{{\includegraphics[width=6cm]{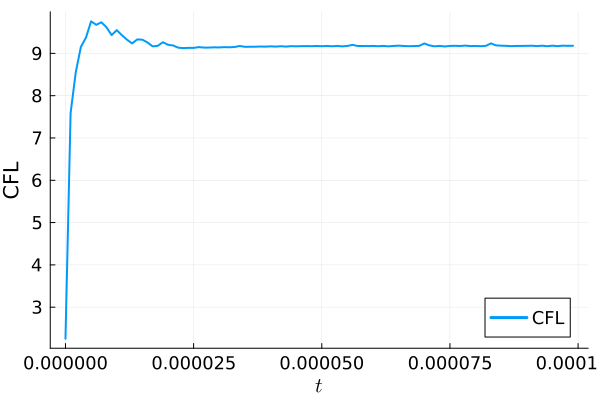} }}%
    \qquad
    \subfloat[\centering Elementwise Limiting Coefficients]{{\includegraphics[width=6cm]{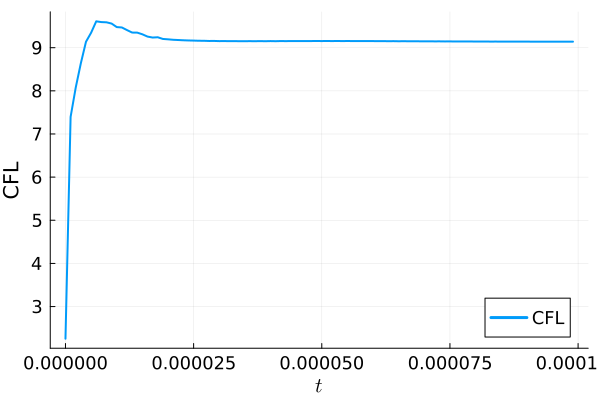} }}%
    \caption{Effective CFL condition for the Leblanc shock tube simulation corresponding to $\Delta t$}%
    \label{fig:LeblancCFL}%
\end{figure}}We observe that the effective CFL condition corresponding to the timestep $\Delta t = 7.1 \cdot 10^{-8}$ is of order $\mathcal{O}(1)$ throughout the simulation. 

\subsubsection{Sedov Blast Wave}

Next, we examine the behavior of quadratic knapsack limiting in solving the 2D compressible Euler, Sedov blast wave problem. This problem is often used to test the behavior of 2D numerical methods in near vacuum conditions. The initial condition for this problem is as follows:
\begin{align*}p\left(x,y,0\right)=\begin{cases}\frac{.4}{\pi r_{0}^{2}}&r<r_{0}\\
10^{-5}&\text{otherwise}\end{cases},&&\vecf{v}\left(x,y,0\right)=\fnt{0},&&\rho\left(x,y,0\right)=1,\\
r=\sqrt{x^{2}+y^{2}},&&r_{0}=4\Delta x,&&\Delta x=\frac{3}{M}.\end{align*}
The Sedov blast wave problem is computed using the strong stability preserving \verb|SSPRK43| timestepper with a fixed $\Delta t = 4.4 \cdot 10^{-5}$. This timestep is within $10^{-6}$ of the smallest fixed timestep where the solution crashes due to nonpositivity for $\alpha = .6$. The simulation is performed to a final time of $T = 1$, over the periodic domain $[-1.5,1.5]^2$. In Figure \eqref{fig:SedovBlast}, we plot the solution generated by quadratic knapsack limiting with $(N, M)=(3,150)$, along with its 10 logarithmically-spaced contours between $[.01, 6]$. For comparison, the low order method is also shown. The relative positivity constraints with $\alpha = .6$ and $\alpha = .9$ are enforced, with elementwise limiting coefficients. The coloring was truncated to $[.01,6]$ for consistency with \cite{Lin23}. To get a clearer picture of the behavior of each solution, we also plot a density cross section in Figure \eqref{fig:SedovCross}. We observe that the low order method solution is the most dissipative. The $\alpha=.6$ and $\alpha=.9$ solutions are roughly identical to each other, though there is less oscillation present in the $\alpha=.9$ solution.

\begin{figure}%
    \centering
    \subfloat[\centering $\alpha = .6$]{{\includegraphics[width=4.2cm]{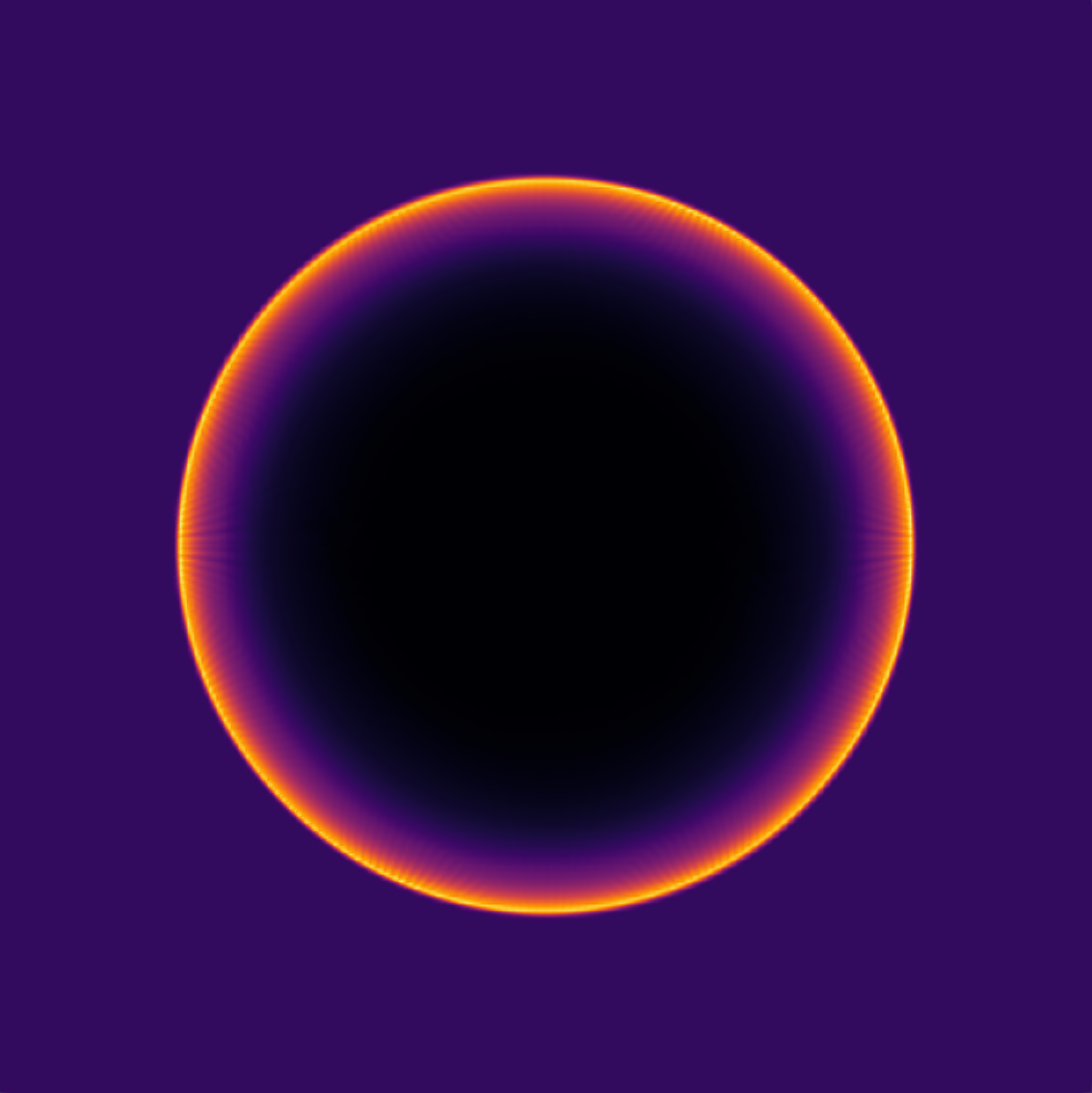} }}%
    \subfloat[\centering $\alpha = .9$]{{\includegraphics[width=4.2cm]{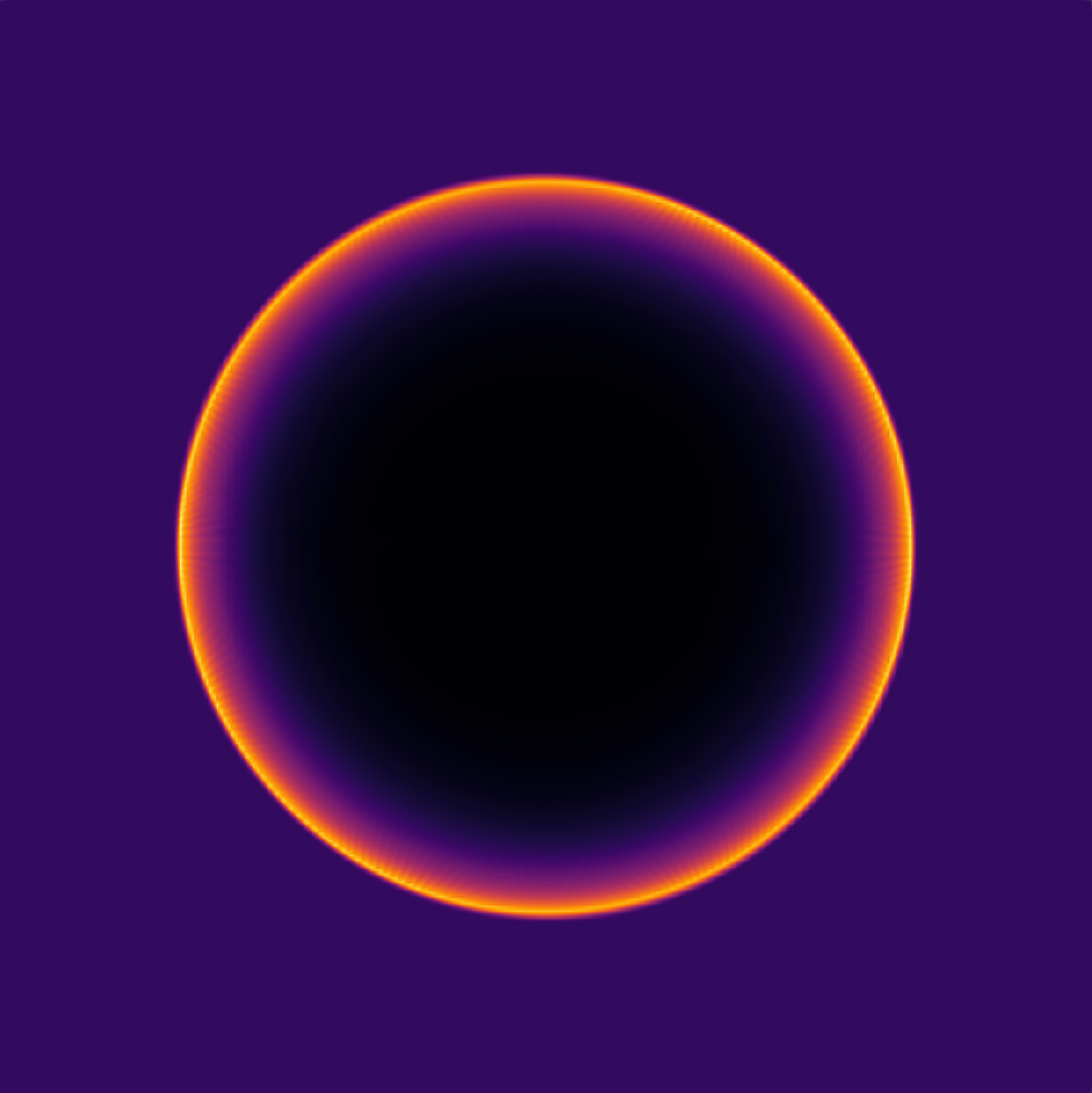} }}%
    \subfloat[\centering Low order method]{{\includegraphics[width=4.2cm]{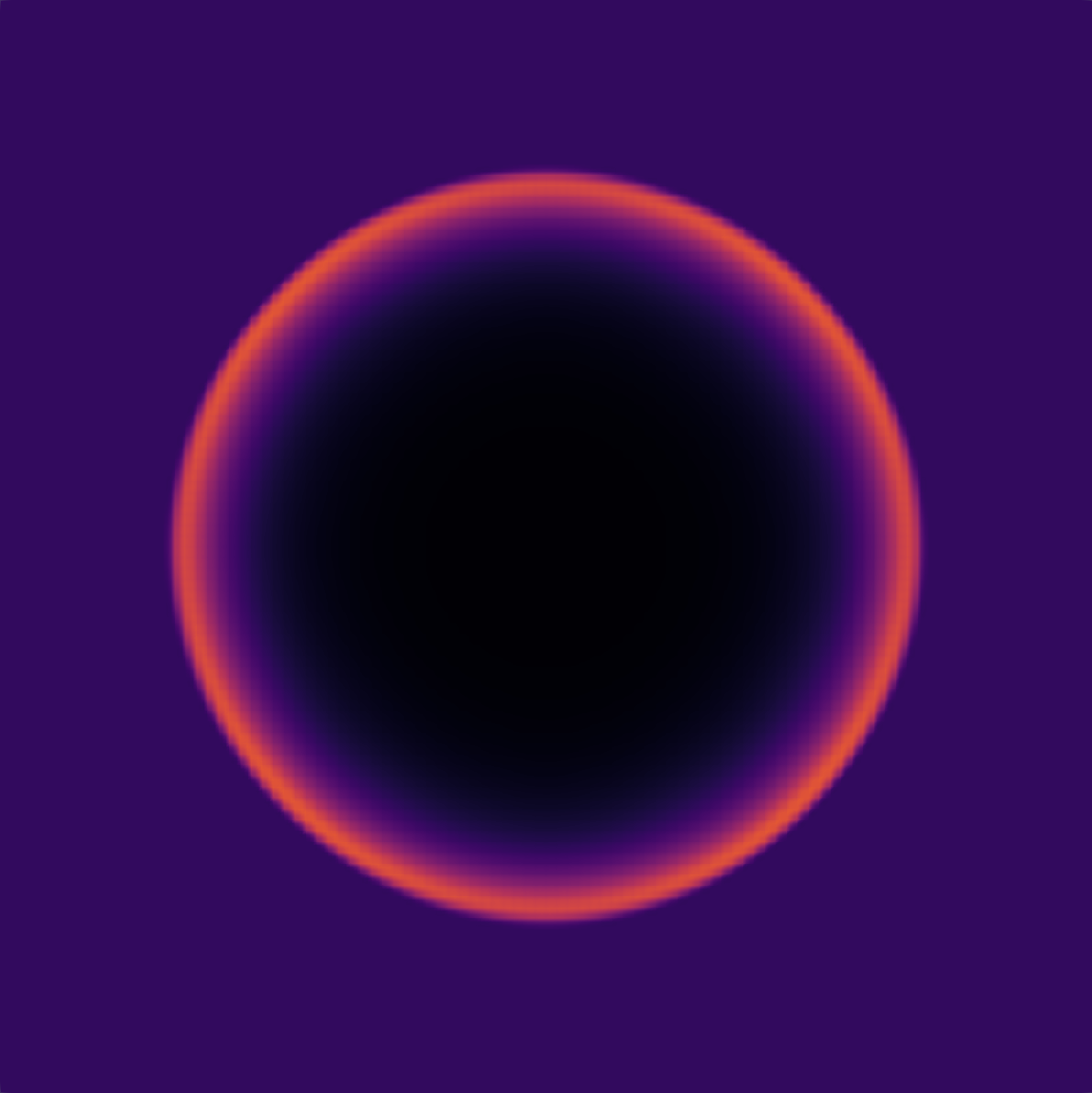} }}%
    \qquad
    \subfloat[\centering Logarithmic contours from $\alpha = .6$ solution]{{\includegraphics[width=4.2cm]{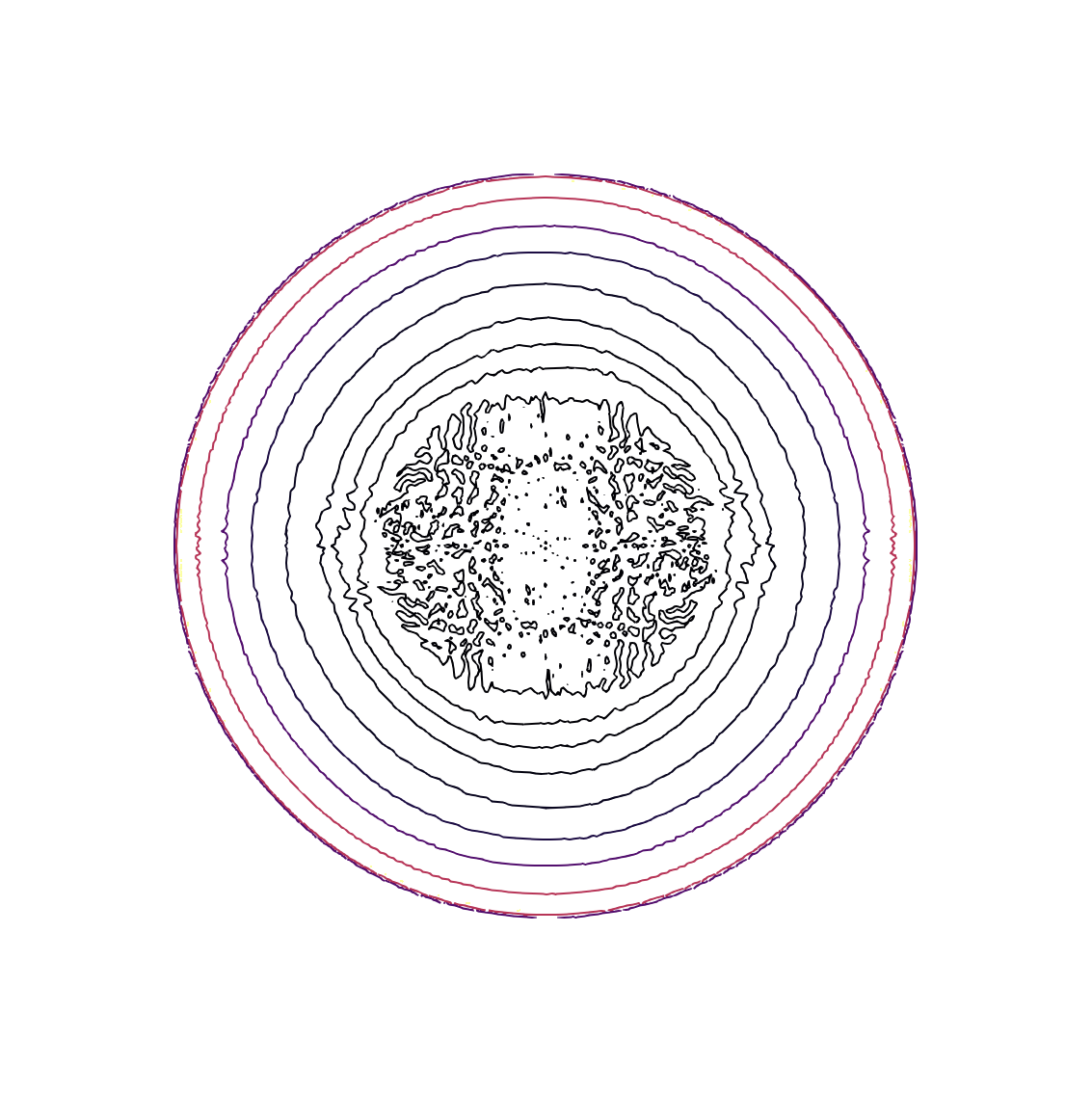} }}%
    \subfloat[\centering Logarithmic contours from $\alpha = .9$ solution]{{\includegraphics[width=4.2cm]{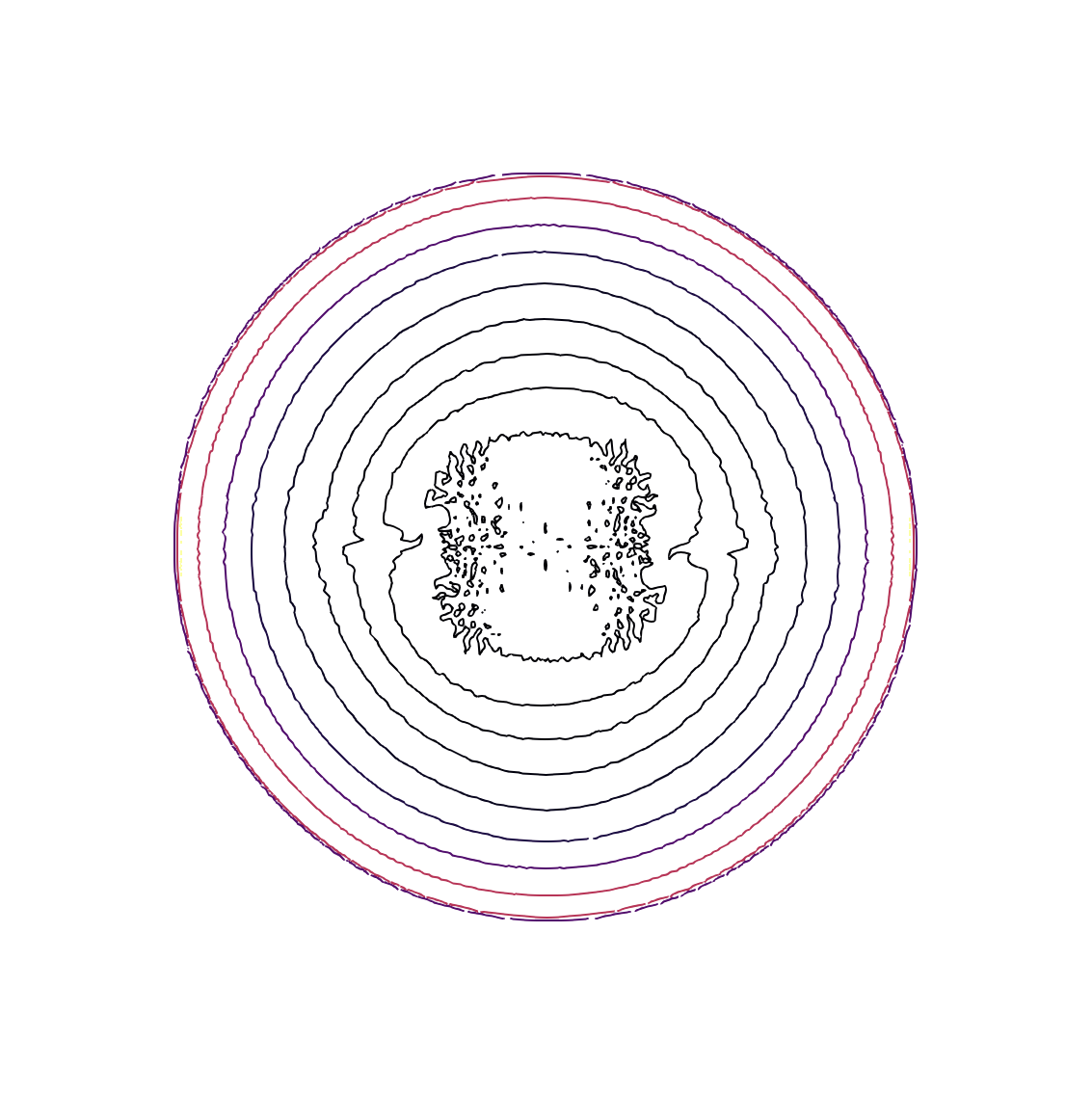} }}%
    \subfloat[\centering Logarithmic contours from low order method solution]{{\includegraphics[width=4.2cm]{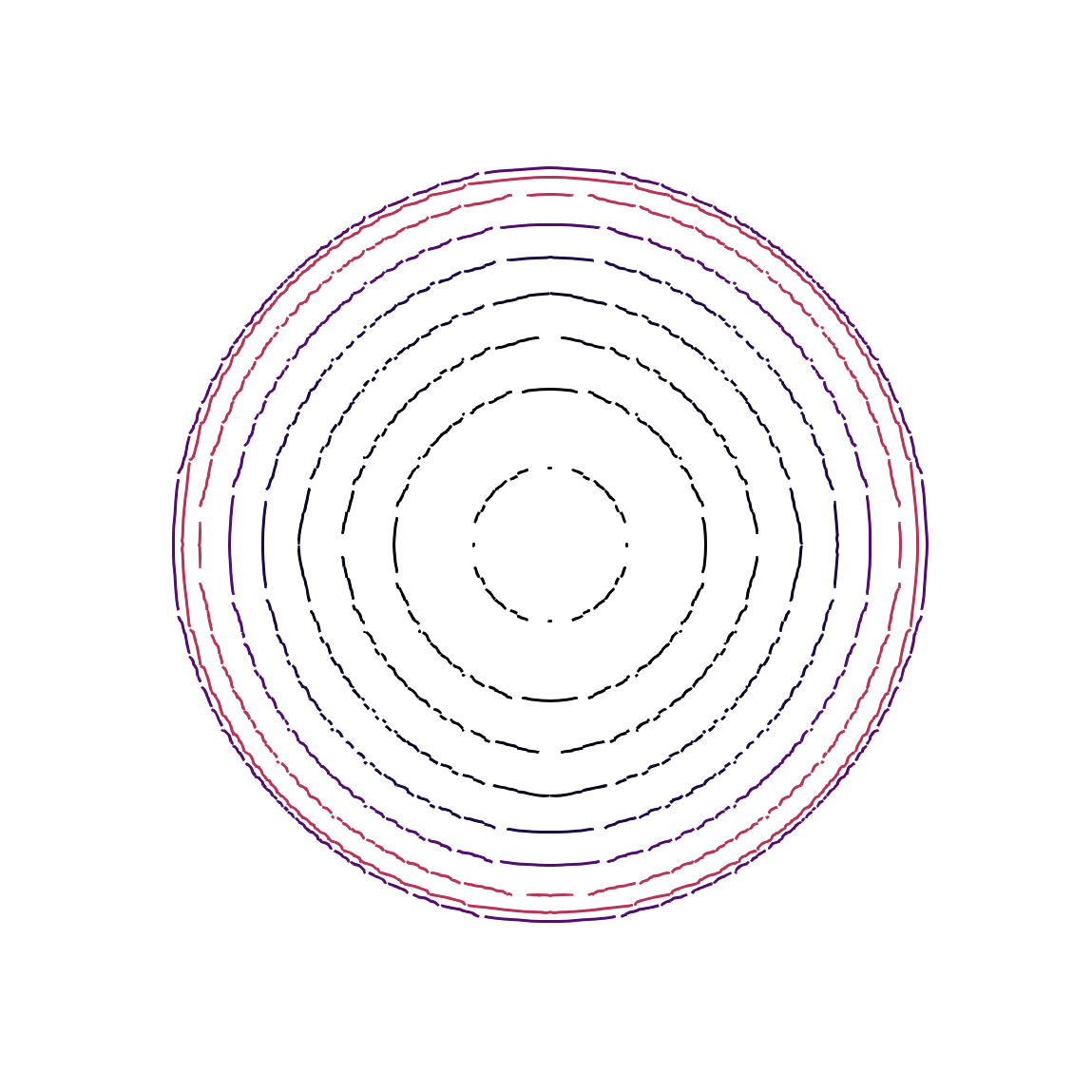} }}%
    \caption{Density of the Sedov blast wave solution at the final time $T = 1$ for $N = 3$, $M = 150$, along with logarithmically-spaced contours.}%
    \label{fig:SedovBlast}%
\end{figure}

\begin{figure}%
    \centering
    \subfloat{{\includegraphics[width=6cm]{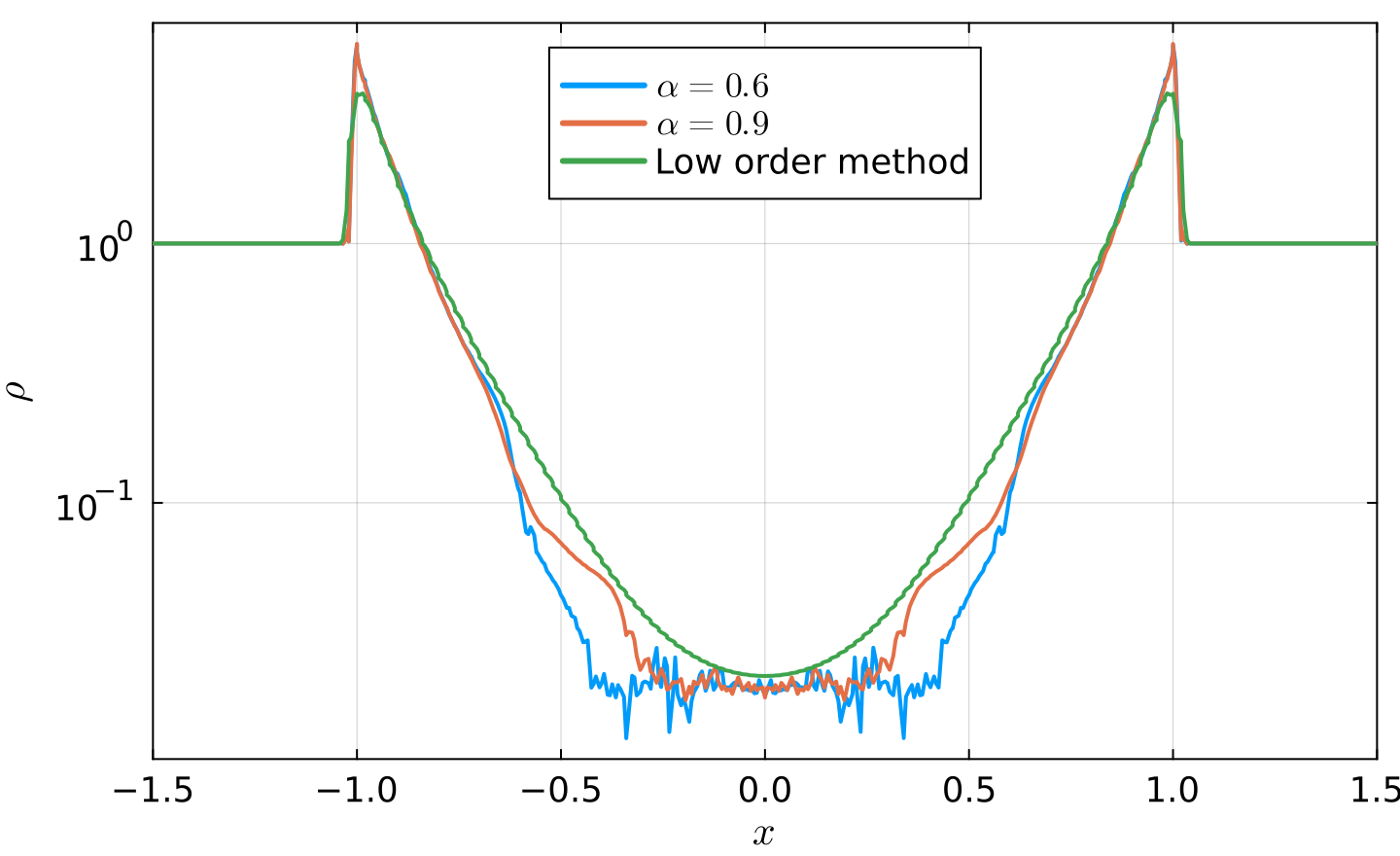} }}%
    \caption{Density cross section of the Sedov blast wave solution.}%
    \label{fig:SedovCross}%
\end{figure}

To measure the effective CFL condition corresponding to the chosen timestep $\Delta t = 4.4 \cdot 10^{-5}$, we apply the same procedure in Section \eqref{sec:Leblanc}. Since we choose the smallest fixed timestep where the solution crashes due to nonpositivity for the $\alpha = .6$ simulation, we only measure the effective CFL condition for the $\alpha = .6$ simulation.\comment{ In Figure \eqref{fig:SedovCFL}, we show the effective CFL condition corresponding to the timestep $\Delta t = 4.4\cdot 10^{-5}$ over time.} We again observe that the effective CFL condition is of order $\mathcal{O}(1)$ throughout the simulation. 

% \begin{figure}%
%     \centering
%     {{\includegraphics[width=8cm]{Sedov-CFL.png} }}%
%     \caption{Effective CFL condition for the Sedov blastwave simulation corresponding to $\Delta t$}%
%     \label{fig:SedovCFL}%
% \end{figure}

\subsubsection{Kelvin-Helmholtz Instability}

Last, we examine the behavior of quadratic knapsack limiting in solving the 2D compressible Euler Kelvin-Helmholtz instability problem (KHI). This problem is often used to test the behavior of 2D numerical methods in the presence of turbulent-like flow. Without enforcing positivity, collocation-type entropy stable methods tend to crash with negative density for KHI around $T \approx 3.5$ \cite{Chan22}. When enforcing positivity, quadratic knapsack limiting is able to simulate KHI until and past $T = 25$. In this case, we call this \textit{long-time} Kelvin-Helmholtz instability. The initial condition for KHI is:
\begin{align*}\rho\left(x,y,0\right)&=\frac{1}{2}+\frac{3}{4}B,&v_{1}\left(x,y,0\right)&=\frac{1}{2}\left(B-1\right),\\v_{2}\left(x,y,0\right)&=\frac{1}{10}\sin\left(2\pi x\right),&p\left(x,y,0\right)&=1,\end{align*}%
where $B=\tanh(15y+7.5) - \tanh(15y-7.5)$. The long-time KHI solution is computed on the periodic domain $[-1,1]^2$ to final time $T=25$ using the adaptive \verb|SSPRK43| timestepper, with an absolute tolerance of $10^{-6}$ and a relative tolerance of $10^{-4}$. In Figure \eqref{fig:KHI}, we plot the solutions generated by quadratic knapsack limiting with $(N,M)=(3,150)$ and $(N, M)=(7, 75)$, at times $t=10$ and $t=25$, with relative positivity constraint $\alpha = 0$. We observe that the solution with $(N, M)=(7, 75)$ has more resolved features than that of $(N, M)=(3, 150)$. 

\begin{figure}%
    \centering
    \subfloat[\centering $N = 3, M = 150,\ t = 10$]{{\includegraphics[width=6cm]{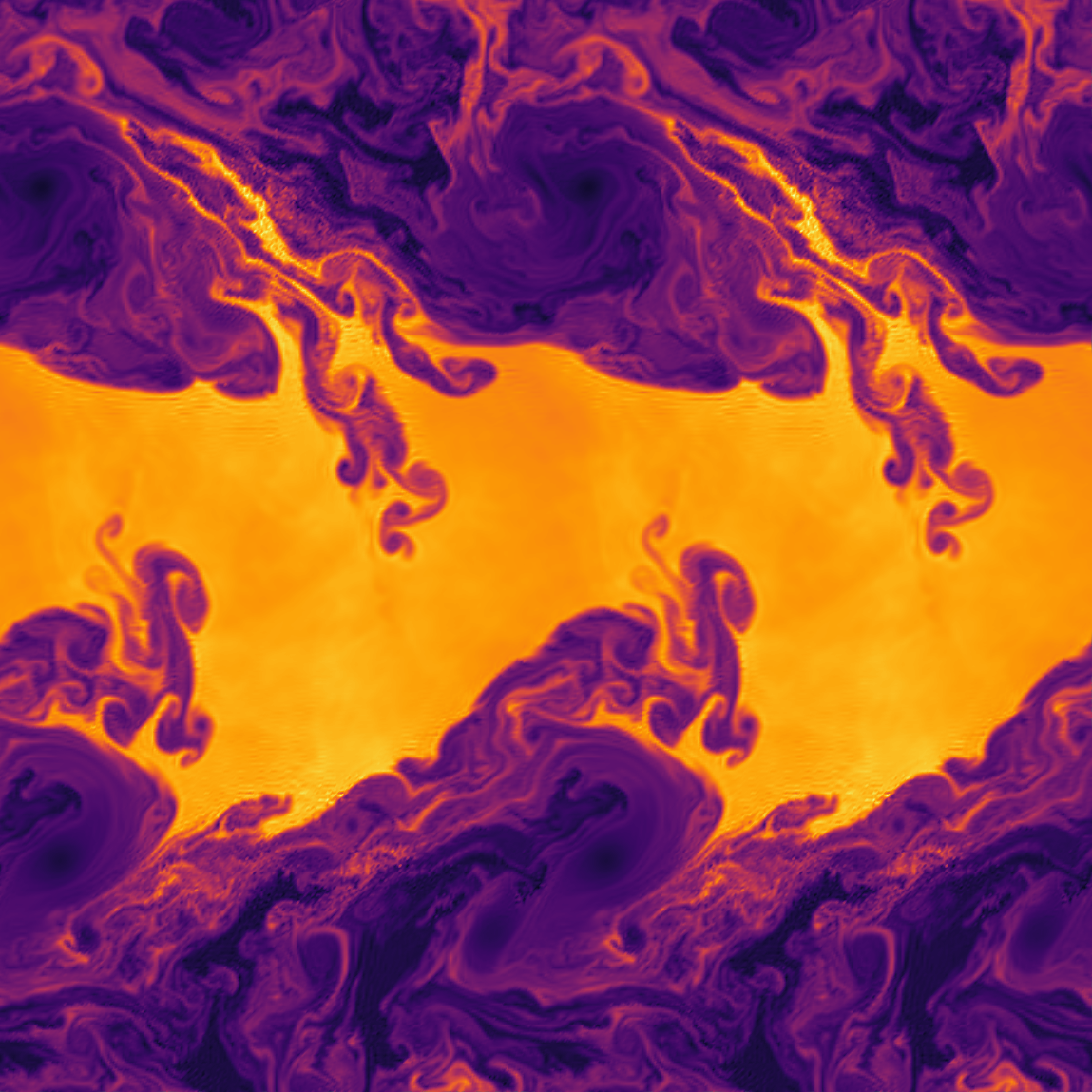} }}%
    \qquad
    \subfloat[\centering $N = 3, M = 150,\ t = 25$]{{\includegraphics[width=6cm]{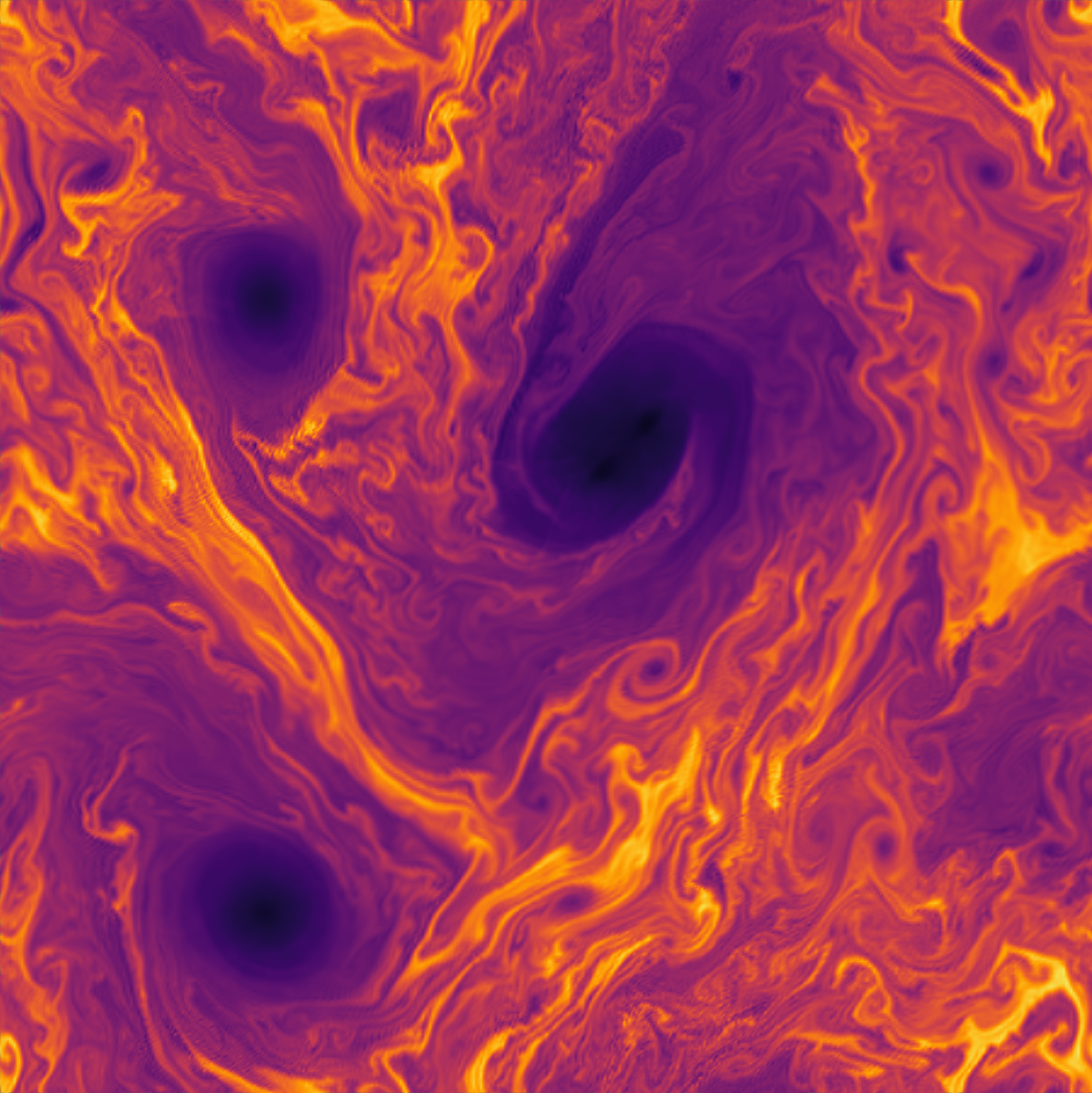} }}%
    \qquad
    \subfloat[\centering $N = 7, M = 75,\ t = 10$]{{\includegraphics[width=6cm]{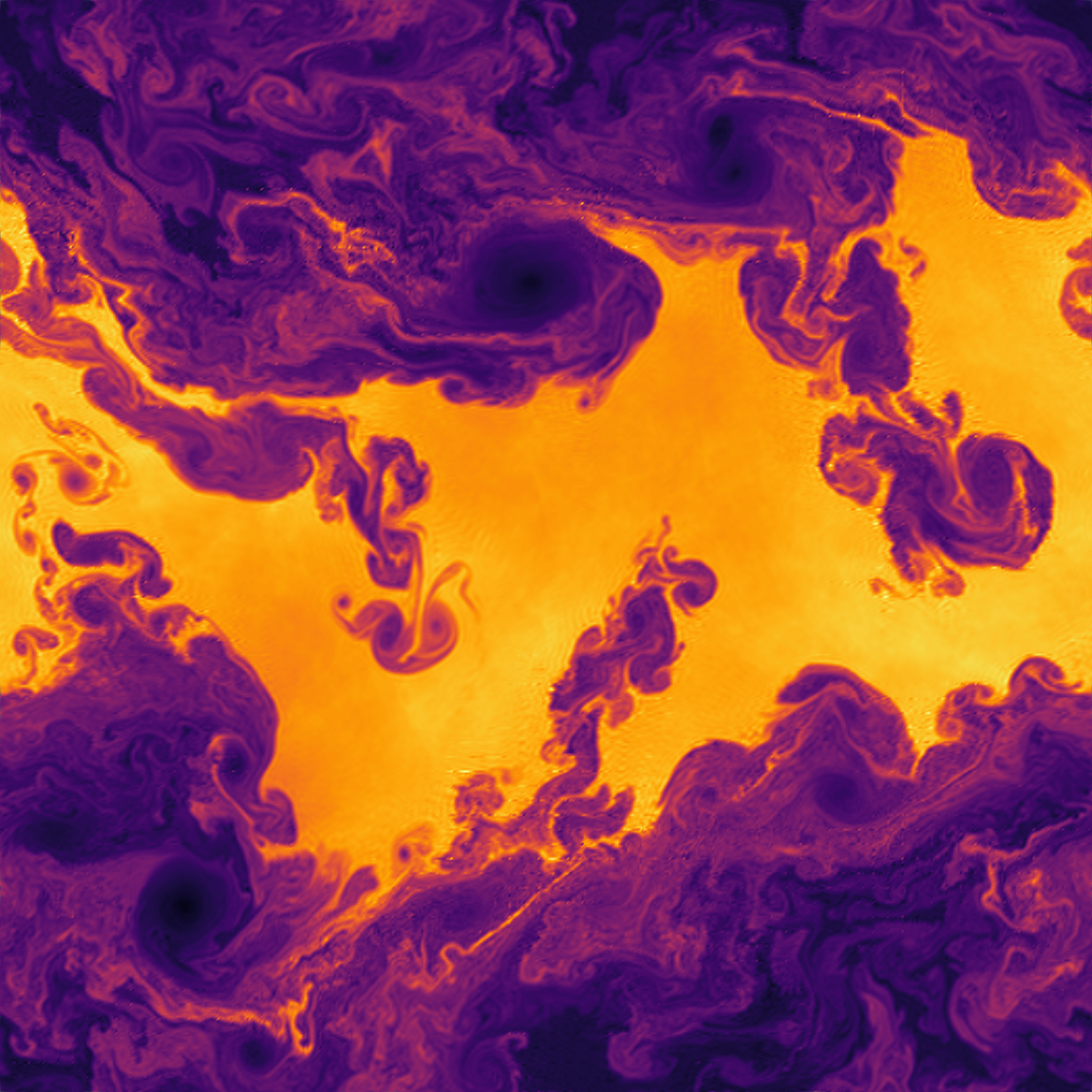} }}%
    \qquad
    \subfloat[\centering $N = 7, M = 75,\ t = 25$]{{\includegraphics[width=6cm]{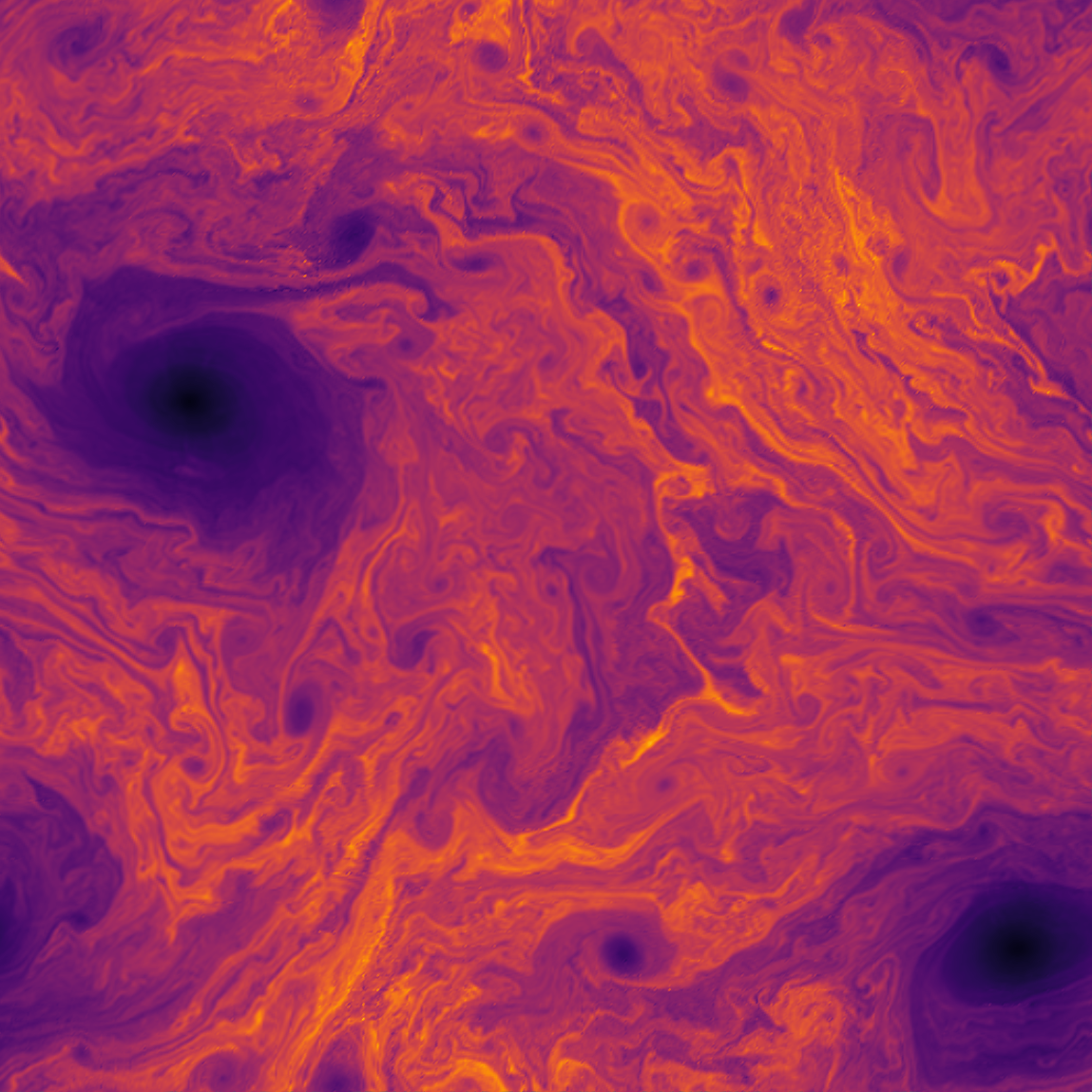} }}%
    \caption{Quadratic knapsack limiting solution to the long-time Kelvin-Helmholtz instability problem. The color range was truncated to $[.25, 2.5]$.}%
    \label{fig:KHI}%
\end{figure}

\section{Conclusion}

In this work, we proposed a high order, entropy stable discontinuous Galerkin method based on quadratic knapsack limiting. Quadratic knapsack limiting blends together a low order, entropy stable, positivity preserving, finite volume-type scheme with a high order accurate, discontinuous Galerkin spectral element method. The blending is parameterized by a vector of blending coefficients, which are elementwise bounded to ensure that the blended scheme satisfies a positivity inequality. The blending coefficients are determined through a quadratic knapsack optimization problem. We showed that such a problem can be solved using an efficient, scalar, quasi-Newton root-finding problem, which takes at most $n+2$ iterations to converge. We also showed that the resulting method is high order accurate, locally linearly stable, and positivity preserving for the compressible Euler equations, while also behaving well under adaptive timestepping and retaining a high order of convergence in time in shock-type problems. 

\section*{Acknowledgments}

The authors gratefully acknowledge helpful discussions with Raymond Park and Sebastian Perez-Salazar. 

\bibliographystyle{siamplain}
\bibliography{references}
\end{document}

%% file: ex_shared.tex
\usepackage{lipsum}
\usepackage{amsfonts}
\usepackage{graphicx}
\usepackage{epstopdf}
\usepackage{algorithmic}
\ifpdf
  \DeclareGraphicsExtensions{.eps,.pdf,.png,.jpg}
\else
  \DeclareGraphicsExtensions{.eps}
\fi

\newsiamremark{remark}{Remark}
\newsiamremark{hypothesis}{Hypothesis}
\crefname{hypothesis}{Hypothesis}{Hypotheses}
\newsiamthm{claim}{Claim}
\newsiamremark{fact}{Fact}
\crefname{fact}{Fact}{Facts}

\headers{Entropy Stable Nodal DG via Quadratic Knapsack}{B. Christner and J. Chan}

\title{Entropy Stable Nodal Discontinuous Galerkin Methods via Quadratic Knapsack Limiting\thanks{
\funding{Supported by National Science Foundation awards DMS-1943186 and DMS-223148.}}}

\author{Brian Christner\thanks{Department of Computational Applied Mathematics and Operations Research, Rice University 
  (\email{bc71@rice.edu}).}
\and Jesse Chan\thanks{Oden Institute for Computational Engineering and Sciences and Department of Aerospace Engineering and Engineering Mechanics, The University of Texas at Austin}
  (\email{jesse.chan@oden.utexas.edu}).
  }

\usepackage{amsopn}
\DeclareMathOperator{\diag}{diag}